\documentclass[11pt,a4paper]{amsart}
\usepackage{tikz-cd}
\usepackage[margin=1in]{geometry}  % set the margins to 1in on all sides
\usepackage{graphicx}              % to include figures
\usepackage{amsmath,euscript,comment}               % great math stuff
\usepackage{amsfonts}              % for blackboard bold, etc
\usepackage{amssymb,amscd,epsf,amsthm, epsfig}  

\usepackage{enumerate}              % better theorem environments
\usepackage[toc,page]{appendix}
\usepackage{amsmath,calligra,mathrsfs}
\usepackage{float}
\usepackage[symbol]{footmisc}
\usepackage{tikz-cd}
\usepackage{xcolor}
\usepackage[frame,cmtip,arrow,matrix,line,graph]{xy}
\usepackage{lpic}
\usepackage{stmaryrd}
\usepackage{scrextend}
\usepackage{enumitem}
\newcommand{\subscript}[2]{$#1 _ #2$}

\usepackage{tikz-cd}
\usepackage[margin=1in]{geometry}  % set the margins to 1in on all sides
\usepackage{graphicx}              % to include figures
\usepackage{amsmath,euscript,comment}               % great math stuff
\usepackage{amsfonts}              % for blackboard bold, etc
\usepackage{amssymb,amscd,epsf,amsthm, epsfig}

\newtheorem{theorem}{Theorem}[section]
\newtheorem{lemma}[theorem]{Lemma}
\newtheorem{proposition}[theorem]{Proposition}
\newtheorem{corollary}[theorem]{Corollary}

\newtheorem{lemma-definition}[theorem]{Lemma-Definition}

\newtheorem{Th}{Theorem}

\theoremstyle{definition}
\newtheorem{definition}[theorem]{Definition}
\newtheorem{example}[theorem]{Example}

\newtheorem{remark}[theorem]{Remark}

\numberwithin{equation}{section}

\numberwithin{theorem}{section}

\newcommand{\Z}{\mathbb{Z}}
\newcommand{\R}{\mathbb{R}}

\newcommand{\calD}{\mathcal{D}}

\newcommand{\calR}{\mathcal{R}}

\DeclareMathOperator{\HOM}{\mathscr{H}\text{\kern -3pt {\calligra\large om}}\,}

\newcommand{\DD}{\EuScript D}

\newcommand{\sA}{\mathop{s\overline{A}}\nolimits}

\newcommand{\Hom}{\mathop{\mathrm{Hom}}\nolimits}

\newcommand{\Map}{\mathop{\mathrm{Map}}\nolimits}
\newcommand{\id}{\mathop{\mathrm{id}}\nolimits}

\newcommand{\op}{\mathop{\mathrm{op}}\nolimits}

\newcommand{\sg}{\mathop{\mathrm{sg}}\nolimits}

\newcommand{\HH}{\mathop{\mathrm{HH}}\nolimits}

\newcommand{\Perf}{\mathop{\mathrm{Perf}}\nolimits}

\newcommand{\F}{\mathcal{F}}
\newcommand{\Le}{\mathcal{L}}
\newcommand{\NN}{\mathbb{N}}

\newcommand{\Si}{\Sigma}
\newcommand{\De}{\Delta}
\newcommand{\eps}{\epsilon}
\newcommand{\la}{\lambda}
\newcommand{\ga}{\gamma}
\newcommand{\be}{\beta}

\newcommand{\del}{\partial}

\newcommand{\inc}{\hookrightarrow}

\newcommand{\lar}{\longleftarrow}
\newcommand{\rar}{\longrightarrow}
\newcommand{\arsim}{\xrightarrow{\sim}}
\newcommand{\minus}{\backslash}
\newcommand{\x}{\times}
\newcommand{\ot}{\otimes}

\newcommand{\lgl}{\langle}
\newcommand{\rgl}{\rangle}

\newcommand{\un}{\underline}

\newcommand{\GH}{\mathrm{GH}}
\newcommand{\CS}{\mathrm{CS}}
\DeclareMathOperator\concat{concat}
\DeclareMathOperator\cut{cut}

\newcommand{\Figeight}{{\mathrm{Fig}(8)}}

\newcommand{\RR}{\mathcal{R}}
\newcommand{\EE}{\mathcal{E}}

\newcommand{\ev}{\operatorname{ev}}

\newcommand{\intp}{\operatorname{int}}

\newcommand{\Apl}{\mathcal{A}_\text{pl}}
\newcommand{\cone}{\operatorname{cone}}

\newcommand{\U}{\mathcal{U}}

\input xy
\xyoption{all}

\let\oldtocsection=\tocsection
\let\oldtocsubsection=\tocsubsection
\let\oldtocsubsubsection=\tocsubsubsection
\renewcommand{\tocsection}[2]{\hspace{0em}\oldtocsection{#1}{#2}}
\renewcommand{\tocsubsection}[2]{\hspace{1em}\oldtocsubsection{#1}{#2}}
\renewcommand{\tocsubsubsection}[2]{\hspace{2em}\oldtocsubsubsection{#1}{#2}}

\usepackage{lipsum}

\usepackage{hyperref}

\definecolor{darkgreen}{RGB}{47,139,79}
\definecolor{darkblue}{RGB}{36,24,130}
\hypersetup{
    colorlinks=true,     linktoc=all,         linktocpage,
    citecolor=darkgreen,
    filecolor=black,
    linkcolor=darkblue,
    urlcolor=black
}

\begin{document}

\title{String topology in three flavours}
\author{Florian Naef}
\author{Manuel Rivera}
\author{Nathalie Wahl}

\maketitle
{\centering\footnotesize \emph{Dedicated to Dennis Sullivan on the occasion of his 80th birthday}.\par}

\begin{abstract}
We describe two major string topology operations, the
Chas-Sullivan product and the Goresky-Hingston coproduct, from
geometric and algebraic perspectives. The geometric construction uses Thom-Pontrjagin intersection theory while the algebraic
construction is phrased in terms of Hochschild homology. We give computations of products and coproducts on lens
spaces via geometric intersection, and deduce that the coproduct
distinguishes 3-dimensional lens spaces. 
Algebraically, we describe the 
structure these operations define together on the Tate-Hochschild complex.
We use rational homotopy theory methods to sketch the equivalence between the geometric and algebraic definitions for simply connected manifolds and real coefficients,
emphasizing the role of configuration spaces.
Finally, we study invariance properties of the operations, both algebraically and geometrically. 
\end{abstract}

\section{Introduction}
\addtocontents{toc}{\protect\setcounter{tocdepth}{1}}

%\Nnote{encouragements from Dennis:  It would be great if all of the elements in this discussion could be combined , illuminated and exposed in a useful and enlightened discussion between all of you.... In the past one has been very glad when a contradiction or conceptual tension arises ,  this has been an opportunity to  learn something .... and in this case it is the question that has fascinated me since grad school\\
% what is the algebraic chain level meaning of a space being a combinatorial or smooth manifold.}

String topology is concerned with algebraic structures defined by intersecting, concatenating, and cutting families of paths and loops
in a manifold $M$. It began with Chas and Sullivan's construction of an intersection type product on $H_*(LM)$, the homology of the space $LM=
\text{Map}(S^1,M)$ of all loops in $M$, also known as the free loop space of $M$ \cite{CS99}. The loop product induces a Lie bracket on $H_*^{S^1}(LM)$, the $S^1$-equivariant homology of $LM$, generalizing an earlier construction of Goldman for curves on surfaces \cite{Gol86}.

Over the last twenty years, string topology has branched out to many corners of mathematics:
\begin{itemize}
\item It has an algebraic counterpart in Hochschild homology through the Jones \cite{Jon87} and Goodwillie \cite{Goo85} isomorphisms
$$ H^*(LM;\mathbb F)\cong HH_*(C^*(M;\mathbb F),C^*(M;\mathbb F))\  \textrm{and}\   H_*(LM)\cong HH_*(C_*(\Omega M),C_*(\Omega M))$$
for $\mathbb F$ a field, $\Omega M$ the based loop space of $M$ and where $M$ is assumed to be simply connected for the first isomorphism, see e.g.~\cite{CohJon,Mer,Felix-Thomas,NaeWil19,Malm};
\item It has a symplectic interpretation through the Viterbo \cite{Vit95} isomorphism (with appropriate coefficients) $$H_*(LM)\cong FH_*(T^*M)$$ with target the Floer homology of the cotangent bundle of $M$, see e.g.~\cite{SaWe06, AbbSch06, AbbSch10, CieLat09}; 
\item Rich families of string operations have been defined, in particular using the algebraic model of string topology, with for instance BV structures, Lie bialgebras, 2-dimensional field theories of various flavours, and more, see e.g.~\cite{godin07, Dru-ColPoirRou,TraZei06, Kau08, Kau07, WahWes08}; 
\item String topology has been used to study closed geodesics on Riemanian manifolds through Morse theory on the energy functional, see  e.g.~\cite{GorHin,HinRad}; 
\item String operations can be defined instead on the loop space $LBG$ for $G$ a Lie group, or more generally on the loop space of stacks, see \cite{ChaMen,HepLah,KGNX} and see e.g.~\cite{Lah,GroLah} for applications to group homology.  
\end{itemize}
We will not be able to cover all aspects of string topology in this note and will instead focus on a few highlights that, we hope, illustrate the richness of the subject. In particular, we will restrict our attention to the original loop product of Chas and Sullivan and its ``dual'', the Goresky-Hingston coproduct. We will describe these two operations geometrically as well as algebraically, and use methods from rational homotopy theory to compare the two descriptions, where the role of configuration spaces will be emphasized.  The geometric aspect of string topology will be illustrated through computations of loop products and coproducts via intersections of geometric cycles in examples from lens spaces.  Algebraically, we will see that the two operations together define a single product on the Tate-Hochschild complex, defined below, and are encoded by the data of a Manin triple. Finally, we will address the question of invariance of the product and coproduct. 

We describe now in more detail the content of this text. Throughout, $M$ will be a closed oriented manifold of dimension $n$, and homology is with $\Z$-coefficients unless otherwise stated. 

\begin{figure}
  \includegraphics[width=0.6\textwidth]{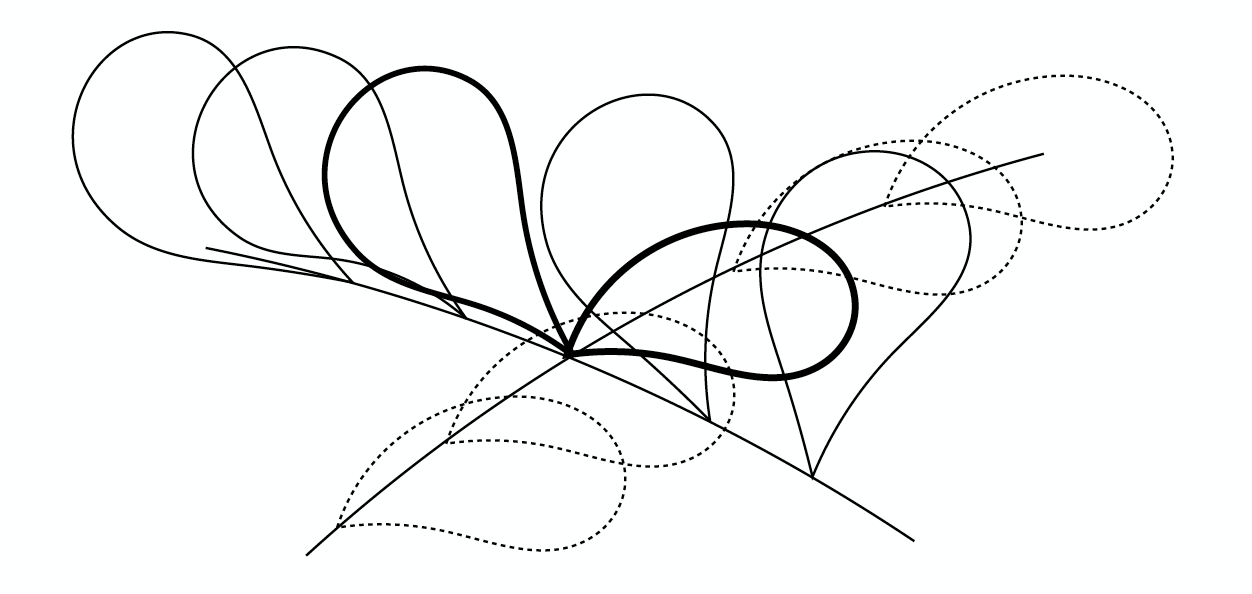}
 % \caption{The Chas-Sullivan product}\label{fig:product}
\end{figure}
%\vspace{-6cm}

\subsection*{Intersection products}
Recall that the classical intersection product $$\bullet: H_p(M) \otimes H_q(M) \to H_{p+q-n}(M)$$ 
 can be computed by geometric intersection for transverse cycles: if $A,B\in H_*(M)$ are homology classes represented by smooth transversally embedded submanifolds, then their product $A\bullet B$ is given by the geometric intersection $A\cap B$ of the cycles.
The original idea behind the Chas-Sullivan product is to define a product on $H_*(LM)$ by likewise transversally intersecting two families of loops in $M$ at their basepoints, which is an intersection in $M$, and concatenating loops at the locus of intersection. This results in a graded commutative and associative product
$$\wedge: H_p(LM) \otimes H_q(LM) \to H_{p+q-n}(LM),$$  that is, 
by construction, compatible with 
%This product, known as the \textit{loop product}, lifts
the intersection product under the evaluation map $\ev_0:LM\to M$. We will refer to the Chas-Sullivan product as the {\em loop product}.

Following ideas going back to Cohen-Jones \cite{CohJon}, we give in Section~\ref{sec:defproco} a formal definition of this product by lifting the definition of the classical intersection product phrased in terms of a Thom-Pontrjagin construction for the diagonal embedding $\De:M\to M\x M$.

The Goresky-Hingston coproduct  \cite{GorHin}, also considered by Sullivan \cite{Sul04} and refered to as the {\em loop coproduct} here,  has the form 
$$\vee: H_p(LM,M) \to H_{p-n+1}(LM \times LM, LM \times M \cup M \times LM).$$  The idea of  the coproduct is, given a family of loops, to look for all the self-intersections in the family of the form $\ga(0)=\ga(t)$, for $\ga$  a loop and $t\in I$ is a time coordinate, and then cut. %the loop at such self-intersections.
Following Hingston-Wahl \cite{HinWah0}, we show that it 
can be defined using a simple variant of the definition of the loop product. The operation is most naturally a relative operation because the interval $I$ has non-trivial boundary; see Remark~\ref{rem:lifts} for non-relative versions of the coproduct. %We will refer to the Goresky-Hingston coproduct as the {\em loop coproduct}. 

The loop product and coproduct can be diagramatically described as
$$\xymatrix{LM\x LM \ar[d]_{\ev_0\x \ev_0} \ar@/^2ex/@{-->}[r] & \ar[l] \Figeight \ar[d]^{\ev_0} \ar[r]^-{concat} & LM & &  LM\x I \ar[d]_{\ev_0\x \ev_t} \ar@/^2ex/@{-->}[r] & \ar[l] \F \ar[d]^{\ev_0} \ar[r]^-{cut} & LM \x LM \\
  M\x M & \ar[l]_-\De M &  &  & M\x M & \ar[l]_-\De M & }$$
where the middle spaces $\Figeight\cong LM\x_M LM$  and $\F\subset LM\x I$ are the subspaces where the desired intersection holds, and where the dashed arrows are ``intersection products'' that are only defined on homology (or on chains). In Sections~\ref{sec:41} and \ref{sec:42}, we will formulate the data used from $M$ to define these intersection products in terms of an {\em intersection context} (see Definition~\ref{def:intcont}). Our preferred intersection context associated to a manifold $M$ will be
$$\begin{tikzcd}
UTM \ar[r] \ar[d] & FM_2 \ar[d] \\
 M \ar[r] & M \times M,
\end{tikzcd}$$
where $FM_2$ is the configuration space of two points in $M$ and $UTM$ the unit tangent bundle of $M$. 

\subsection*{Geometric computations}

Just like the intersection product $\bullet$ can be computed by geometric intersection for nice enough cycles, the loop product and coproduct can be computed by a direct intersection for cycles that are appropriately transverse. This is made precise in Proposition~\ref{prop:compute}, following \cite{HinWah0},  and illustrated through the computation of the loop product and coproduct of a family of classes generating $H_3(L\Le_{p,q})$, for $\Le_{p,q}$ a 3-dimensional lens spaces; see Propositions~\ref{prop:prodrho} and~\ref{prop:vrho}. As an application of the computation, we
prove the following
\begin{Th}[Theorem \ref{thm:lens}]\label{thmA}
The loop coproduct distinguishes non-homeomorphic 3-dimensional lens spaces.  
% Let $\Le_{p,q_1},\Le_{p,q_2}$ be 3-dimensional lens spaces.  A homotopy equivalence
%  $f:\Le_{p,q_1}\to \Le_{p,q_2}$ is a degree 1 homeomorphism if and only if the induced map $Lf_*:H_*(L\Le_{p,q_1},\Le_{p,q_1})\to H_*(L\Le_{p,q_2},\Le_{p,q_2})$ preserves the loop coproduct of degree 3 classes. 
  \end{Th}
% show \Nnote{maybe :)} that the coproduct distinguises non-homeomorphic 3-dimensional lens spaces, see \Nnote{XX}.
  This result is an extension of a computation of the first author in \cite{Nae21}, used in that paper to show  that the loop coproduct is not homotopy invariant; see below for more details about the invariance properties of the loop product and coproduct.

\subsection*{String topology algebraically }
Assume now that $M$ is a simply connected closed manifold. The isomorphism  $HH_*(C^*(M;\mathbb F),C^*(M;\mathbb F))\cong H^*(LM;\mathbb F)$ mentioned above, actually holds independently of the fact that $M$ is a manifold. However, the algebraic structure of the Hochschild complex becomes much richer once one inputs that $H^*(M)$ statisfies Poincar\'e duality, or in other words that it is a {\em Frobenius algebra} (see Definition~\ref{def:symfrob}). In the above isomorphism, we can replace $C^*(M;\mathbb F)$ by any algebra $A$ quasi-isomorphic to it in the category of dg algebras. By a theorem of Lambrechts-Stanley, it is possible to find a model $A$ for the rational cochains $C^*(M;\mathbb Q)$ that has the structure of a (strict) commutative dg Frobenius algebra compatible with the Frobenius structure on $H^*(M;\mathbb Q)$ (see Theorem~\ref{lastthm} and Example~\ref{ex:frobmodel}).
The relevant consequence for us is that: 
\begin{center}
{\em 
  The algebraic structure of the Hochschild chains or cochains of dg Frobenius algebras \\ reflects rational string topology.}
\end{center}
% note here that the isomorphism $A\cong A^\vee$ between $A$ and its dual induce
For Frobenius algebras, we indeed  have an isomorphism between the linear dual of the Hochschild chain complex $C_*(A,A)$ and the Hochschild cochain complex $C^*(A,A)$, so both complexes are relevant (see Remark~\ref{rem:HHdual}). 

There is a wealth of literature on the algebraic structure of the Hochschild chains and cochains of Frobenius algebras, including algebraic versions of the product and coproduct just described, see e.g.~\cite{CohJon, Mer, Abb16, FTV} for the loop product and \cite{Abb16,Kla13B} for the loop coproduct, or e.g.~\cite{TraZei06, Kau08, Kau07, Kau21,WahWes08} for larger structures encompassing both,  or \cite{Kla13B,Wah16} for a prop of universal operations on the Hochschild complex of symmetric or commutative Frobenius algebras. (See also \cite{BerKau} in the present volume.) 

It turns out that the loop product identifies with the classical cup product on Hochschild cochains \cite{Felix-Thomas}, while the loop coproduct becomes the following product on {\em relative} Hochschild chains (see Definition~\ref{def:relHH}):
%\Nnote{rewritten a little}
%\Mnote{I tried rewriting the following theorem in an appropriate form for the introduction (the previous formulation was a bit confusing) but I am not sure if I succeeded}
\begin{Th}\cite{NaeWil19}\label{thmB}
  Let $A$ be a dg Frobenius algebra model for $C^*(M; \mathbb{R})$. Under a relative version of the Jones isomorphism $H^*(LM;\mathbb R)\cong HH_*(C^*(M;\mathbb R),C^*(M;\mathbb R)) \cong HH_*(A,A)$, the linear dual of the loop coproduct is given on cochains by the formula
  %given by the formula  
  $$( \overline{a_1} \otimes \dotsb \otimes \overline{a_p} \otimes a_{p+1})* (\overline{b_1} \otimes \dotsb \otimes \overline{b_q} \otimes b_{q+1}) =\sum_i \pm \overline{b_1}\otimes \cdots \otimes \overline{b_{q+1}e_i} \otimes \overline{a_1}\otimes \cdots \otimes \overline{a_p} \otimes a_{p+1}f_i,$$
  where $\Delta(1)=\sum_ie_i\otimes f_i\in A\otimes A$ represents the Thom class of the diagonal in $M \times M$. %determined by the Frobenius structure of $A$
  (See Example~\ref{ex:PDint} and %Definitions~ \ref{def:symfrob} and
  Definition \ref{defn:algebraic coproduct}).
% $\alpha = \overline{a_1} \otimes \dotsb \otimes \overline{a_p} \otimes a_{p+1}$ and $\beta = \overline{b_1} \otimes \dotsb \otimes \overline{b_q} \otimes b_{q+1}$  Hochschild chains, 
%where $\eta_i = |\alpha| |f_i| +|b_{q+1}| + (|\alpha|+n-1) (|\beta| +n-1)$. 
\end{Th} 
  % Section~\ref{sec:algop}.
 This result is stated as Theorem~\ref{thm:alggeo} in the present paper, and we give a sketch proof of the result in Section~\ref{sec:real}. 

\smallskip
 
%With this in mind, in Section~\ref{sec:HH}, we will consider the Hochschild chains and cochains of dg Frobenius algebras.
In Section~\ref{sec:Tate2}, we will focus on the following aspect of the algebraic structure defined by the algebraic product and coproduct:
\begin{Th}\cite{RivWan19}\label{thmC}
  The algebraic product and coproduct extend to define toghether a single $A_\infty$-structure on the Tate-Hochschild complex
  $$\calD^{*,*}(A,A) = \cdots \xrightarrow{\partial_h} s^{1-k}C_{-1,*}(A,A) \xrightarrow{\partial_h} s^{1-k}C_{0,*}(A,A) \xrightarrow{\gamma} C^{0,*}(A,A) \xrightarrow{\delta_h} C^{1,*}(A,A) \xrightarrow{\delta_h} \cdots
  $$
  that is compatible with the natural pairing between Hochschild chains and cochains and with an extension of Connes' operator $B$ to the Tate-Hochschild complex. On cohomology, the product is graded commutative, and $H^*(\calD^*(A,A))$ identifies, as an algebra, with the endomorphism algebra of $A$ in the {\em singularity category} of $A$-$A$-bimodules (see Remark~\ref{rem:Tatealgtop}).
\end{Th}
%following Rivera-Wang  \cite{RivWan19}, 
%we will explain in Section~\ref{sec:Tate2} how these two operations define together a single product on the Tate-Hochschild complex (see Theorem~\ref{thm1}), which is an unbounded cochain complex obtained by
Here the Tate-Hochschild complex ``glues together'' the Hochschild chains and cochains along the map $\gamma$ that can be thought of as an Euler characteristic, constructed using the Frobenius structure of $A$, 
%or, more precisely, using the ``Euler characteristic" at degree $0$:
%$$\calD^{*,*}(A,A) = \cdots \xrightarrow{\partial_h} s^{1-k}C_{-1,*}(A,A) \xrightarrow{\partial_h} s^{1-k}C_{0,*}(A,A) \xrightarrow{\gamma} C^{0,*}(A,A) \xrightarrow{\delta_h} C^{1,*}(A,A) \xrightarrow{\delta_h} \cdots
%$$
see Section~\ref{sec:Tate} for a complete definition of this complex. In  Remark~\ref{rem:Manin}, we give a description of this structure in terms of Manin triples, and this implies a form of infinitesimal bialgebra compatibility between the Goresky-Hingston coproduct and the Chas-Sullivan loop product. 
Note that Cieliebak-Hingston-Oancea have given a geometric version of the above Tate construction, including its algebra structure, using Rabinowitz-Floer homology, a theory that combines symplectic homology and cohomology via a ``V-shaped" Hamiltonian \cite{CiFrOa10, CieOan20, CieHinOan1, CieHinOan2}. Theorem~\ref{thmC} is stated as Theorem~\ref{thm1} in the text. 

The Tate-Hochschild complex satisfies the following strong invariance property, that is a consequence of the interpretation in terms of the singularity category:
\begin{Th}\cite{RivWan19}\label{thmD}
 If two simply connected symmetric dg Frobenius algebras are quasi-isomorphic as dg associative algebras, then their Tate-Hochschild cohomologies are isomorphic as algebras.
  \end{Th}
This result is stated as Theorem~\ref{thm2} in the text. A direct consequence of the result is that the algebraic version of the loop coproduct is a homotopy invariant in the simply connected setting (see Corollary~\ref{corollary1}).

\subsection*{Naturality and invariance}
One of the original motivations of Chas and Sullivan in studying free loop spaces was to understand what characterizes the algebraic topology of manifolds and to construct algebraic invariants that could detect beyond the homotopy type; in Sullivan's own words to us
\begin{center}
``{\em ...it is the question that has fascinated me since grad school:
  what is the algebraic chain level meaning of a space being a combinatorial or smooth manifold?}''
\end{center}
%(see Acknowledgments section below).
The particular instance of this question we will adress here is the following: a homotopy equivalence  $M\xrightarrow{\simeq} N$ induces an isomorphism  $H_*(LM)\xrightarrow{\cong} H_*(LN)$, and likewise on relative homology, and one can ask whether this induced map respects the loop product or coproduct. We summarize in the following result what is known about the question: 
\begin{Th}\label{thmE}
  \begin{enumerate}
  \item \cite{CKS} The Chas-Sullivan product on $H_*(LM)$ is invariant under homotopy equivalences of manifolds $M\arsim N$.
  \item (\cite{RivWan19} and \cite{NaeWil19}) The Goresky-Hingston coproduct on $H_*(LM;\R)$ is invariant under homotopy equivalences of simply connected manifolds $M\arsim N$.
    \item \cite{Nae21} The Goresky-Hingston coproduct on $H_*(LM)$ is not homotopy invariant in general. 
    \end{enumerate}
  \end{Th}
  Alternative proofs of part (1) of the theorem were given by \cite{GruSal,Felix-Thomas,Cra}. We give here a sketch proof of this result, in Theorem~\ref{thm:invEE}, stated in terms of homotopy invariance of general {\em intersection products}.
Part (2) of the theorem is a direct consequence of combining Theorems~\ref{thmB} and~\ref{thmD}, while part (3) is a consequence of Theorem~\ref{thmA}. 
  
  The essential difference between the loop product and coproduct is that the loop coproduct uses a {\em relative intersection product}, and the proof of homotopy invariance of intersection product does not extend to proving the relative result. The article \cite{Nae21} suggests that the failure of invariance of the loop coproduct is related to Reidemeister torsion, which is compatible with Theorem~\ref{thmA}. See also \cite{HinWah1} for conditions on the homotopy equivalences under which the coproduct is invariant.

%\Mnote{I reorganized and added some sentences to the following two paragraphs, please read} 
%\Nnote{just changed the last sentence  a little. }
  A  non-invariance result was earlier obtained by Basu for a modified version of the coproduct \cite{Bas11}.
Naef used the lens spaces  $\Le_{1,7}$ and $\Le_{2,7}$ in \cite{Nae21} to show non-homotopy invariance of the coproduct on homology. 
The very same lens spaces where used by Longoni-Salvatore in \cite{LonSal} to show that the configuration space of two points in a manifold is likewise not a homotopy invariant of the manifold. Although we do not directly relate these two computations of non-homotopy invariance, we have already seen above that the configuration space of two points is an important ingredient in the definition of the loop coproduct, being part of the data needed to define the corresponding (relative)  intersection product, see Sections~\ref{sec:41} and \ref{sec:42}.

The Lie bialgebra structure at the level of $S^1$-equivariant homology is a homotopy invariant for simply connected manifolds by \cite{Nae21}. The recent paper \cite{CHV22} proves that homotopy invariance over the reals is also satisfied for  a chain level version of the Lie bialgebra structure (also known as $IBL_{\infty}$-algebra) in the case of 2-connected manifolds.
It is so far unknown whether  the chain level Lie bialgebra structure on $S^1$-equivariant chains (or a chain level version of the coalgebra structure in the non-equivariant case) is a homotopy invariant for simply connected manifolds. 
%To determine exactly where the homotopy invariance breaks down one must answer the following additional question: is the chain level Lie bialgebra structure on $S^1$-equivariant chains (or a chain level version of the coalgebra structure in the non-equivariant case) a homotopy invariant for simply connected manifolds?

%We study in Section~\ref{sec:invariance}  the naturality properties of general ``intersection products'', as defined in Sections~\ref{sec:41} and \ref{sec:42}\Nnote{not ideal as a ref}, and state the main ingredient in the homotopy invariance of the loop product in those terms in Theorem~\ref{thm:invEE}. The main difference between the loop product and coproduct is whether this intersection product is of a relative form or not. This turns out to make a crucial difference. 

 %Naef has observed in \cite{Nae21} that the Goresky-Hingston coalgbebra structure on $H_*(LM,M)$ is not a homotopy invariant in the non-simply connected case: it can distinguish homotopy equivalent Lens spaces $L(1,7)$ and $L(2,7)$. Basu made a similar observation in his thesis where a modified version of the coproduct was studied \cite{Bas11}. Rivera and Wang have proven an algebraic version of the homotopy invariance in the simply connected setting and over a field using the Tate-Hochschild construction applied to a Poincar\'e duality model for the underlying manifold \cite{RivWan21}. 

%\medskip

\subsubsection*{Organisation of the paper} 
In Section~\ref{sec:inter}, after recalling the Thom-Pontrjagin definition of the intersection product, we give a chain level definition of the loop product and coproduct.
Section~\ref{sec:computations} gives the computations of the loop products and coproducts on $H_3(L\Le_{p,q})$ for 3-dimensional lens spaces $\Le_{p,q}$. The coproduct computation is used in Section~\ref{sec:invariance0} to show that the loop coproduct is not homotopy invariant. Then Section~\ref{sec:badcop} gives an alternative definition of the loop coproduct as a relative version of the so-called ``trivial coproduct'', the coproduct on the loop space that only looks for basepoint self-intersections at time $t=\frac{1}{2}$. This definition will be used in Section~\ref{sec:conf} to show the equivalence between the algebraic and geometric descriptions of the coproduct.

Section~\ref{sec:HH} is concerned with the algebraic version of string topology. It starts with recalling and setting in context the concepts of  Frobenius algebras, Hochshild chains and cochains. Section~\ref{sec:Tate} then gives the definition of the Tate-Hochschild complex of a dg Frobenius algebra. The loop product and coproduct are defined algebraically in Section~\ref{sec:algop} as products on the Hochschild cochains and chains respectively. These two products are assembled to a single product on the Tate-Hochschild complex in Section~\ref{sec:Tate2}, where it is also interpreted in the language of Manin triples.  The invariance of the product on the Tate-Hochschild complex is stated at the end of the section. 

Section~\ref{sec:conf} takes a closer look at the ``intersection products'' that appear in the definition of the loop product and coproduct. After revisiting the definitions of the loop product and coproduct in Section~\ref{sec:41}, the notion of {\em intersection context} is defined in Section~\ref{sec:42}, a data one can construct intersection and relative intersection products from. The naturality and invariance properties of such intersection products are discussed in Section~\ref{sec:invariance}. Finally, 
Section~\ref{sec:real} gives a sketch proof of the equivalence between the algebraic and geometric coproduct (Theorem~\ref{thm:alggeo}) using an intersection context featuring the configuration space of two points in $M$ and its real model \cite{Campos-Willwacher,Idrissi}.

\subsection*{Acknowledgements}
The authors would like to thank Dennis Sullivan for being a constant source of mathematical inspiration and for encouraging us to %compare different approaches to string topology: this paper arose from
compare and combine our different approaches to the question of invariance of the loop coproduct, with the goal of illuminating the conceptual tensions in our (at times contradictory!) results. 
%
%A few months ago, Dennis wrote to all of us and ended his message saying:
%\\
%\textit{``It would be great if all of the elements in this discussion could be combined, illuminated, and exposed in a useful and enlightened discussion between all of you.
%In the past, one has been very glad when a contradiction or conceptual tension arises, this has been an opportunity to learn something, and in this case it is the question that has fascinated me since grad school:
%what is the algebraic chain level meaning of a space being a combinatorial or smooth manifold?"}
%\\
%We hope this note serves as a first step in this direction. 
We see this paper as  a first step in this direction. 

FN and NW have received funding from the European Union’s Horizon 2020 research and innovation programme under the Marie Sklodowska-Curie (grant agreement No. 896370) and the European Research
Council  (grant agreement No. 772960) respectively, and were both supported by the Danish National Research Foundation through
 the Copenhagen Centre for Geometry and Topology (DNRF151). 
 MR was supported by NSF Grant 210554 and the Karen EDGE Fellowship. 
 %, and NW by the  the European Research Council (ERC)  under the European Union’s Horizon 2020 research and innovation programme (grant agreement No. 772960). FN and NW also acknowledge the support of the Danish National Research Foundation through the Copenhagen Centre for Geometry and Topology (DNRF151). 

\tableofcontents

\section{String topology via geometric intersection}\label{sec:inter}
\addtocontents{toc}{\protect\setcounter{tocdepth}{2}}

Let $M$ be a closed oriented manifold of dimension $n$, and pick a Riemannian metric on $M$.  
The loop space $LM=\Map(S^1,M)$ is homotopy equivalent to the space $\Lambda M$ of $H^1$--loops on which the energy functional is defined:
$$LM\simeq \Lambda M \xrightarrow{\ E\ } \R, \ \ \ \textrm{where } \ E(\gamma)=\int_{S^1}|\gamma'(t)|^2dt.$$ 
The critical points of the energy are precisely the closed geodesics. Given that the energy is nice enough to do Morse theory, it follows that the homology $H_*(LM)\cong H_*(\Lambda M)$ ``knows'', or even ``is build out of'' closed geodesics. (See e.g., \cite{Oan15} for a  survey of Morse theory on the free loop space.)

As a graded abelian group, $H_*(LM)$ depends only on the homotopy type of $M$, whereas the closed geodesics depend on $M$ as a Riemannian manifold. This naturally leads to the question of whether there is some additional structure on $H_*(LM)$ that depends on a more refined structure than just the homotopy type of $M$.
When $M$ is a closed manifold, its homology satisfies Poincar\'e duality, and this duality takes  the cup product of $H^*(M)$ to the {\em intersection product}:
$$H_p(M)\otimes H_q(M) \xrightarrow{\ \bullet\ } H_{p+q-n}(M).$$
%Chas and Sullivan suggested in \cite{CS99} the following potential answer to the question of using the manifold structure of $M$ to refine $H_*(LM)$: to lift the
%intersection product of $H_*(M)$ to a product
The lifts of the intersection product given by the  Chas-Sullivan product 
$$H_p(LM)\otimes H_q(LM) \xrightarrow{\ \wedge\ } H_{p+q-n}(LM)$$
%obtained from intersecting the chains of basepoints of families of loops, and then concatenating the loops that now share a common basepoint. 
%\Nnote{I'm not super happy with the picture. Any suggestion to improve welcome!} 
%A related operation, often called the Goresky-Hingston coproduct, also considered by Sullivan (see \cite{GorHin,Sul04}), takes instead a single family of loops $\ga$ and looks for self-intersections of the loops of the form $\gamma(0)=\gamma(t)$ for some time $t\in [0,1]$; as we will see, this yelds a map
and Goresky-Hingston coproduct 
$$H_p(LM,M) \xrightarrow{\ \vee\ }  H_{p+1-n}(LM\x LM,M\x LM\cup LM\x M)$$
%in homology relative to the constant loops. 
briefly described in the introduction, give a potential answer to the above question. 
Following ideas of Cohen-Jones \cite{CohJon} as implemented in \cite{HinWah0}, we explain here how both operations can be defined on chains as direct lifts of the intersection product, by using a chain-level definition of the intersection product in terms of a Thom-Pontrjagin construction. Section~\ref{sec:computations} will give example computations, obtained from intersecting geometric cycles, from which we will be able to deduce  in Section~\ref{sec:invariance0} that the coproduct does detect more than the homotopy type. Finally, Section~\ref{sec:badcop} will give an alternative definition of the coproduct.

 Note that homology in this section will always mean homology with integral coefficients: $H_*(\_\,):=H_*(\_\,;\Z)$, and the same for cohomology. 

\subsection{The intersection product as a Thom-Pontrjagin construction}\label{sec:int0}

The normal bundle of the diagonal embedding $\De: M\inc M\x M$ is isomorphic to the tangent bundle $TM$.
Identifying $TM\equiv TM_\eps$ with its subbundle of {\em small vectors}, i.e.~vectors of length at most $\eps\ll \rho$ for $\rho$ the injectivity radius, the map
$$\nu_M: TM \inc M\x M \ \ \textrm{defined by }\ \nu_M(x,V)=(x,x+\exp_xV)$$ 
is an explicit tubular neighborhood for $\De$, with image the  $\eps$--neighborhood of the diagonal
$$\nu_M: TM \xrightarrow{\cong} U_M=\{(x,y)\in M\x M \ |\ |x-y|<\eps\}.$$
Under this identification, the bundle projection map $TM\to M$  becomes the retraction  $r:U_M\to M$ defined by $r(x,y)=x$. 
We let
\begin{equation}\label{equ:tau}
  \tau_M\in C^n(M\x M,M\x M\backslash M)\xleftarrow{\sim} C^n(TM,TM\backslash M)
  \end{equation}
denote the image of a cochain representative for the Thom class for $TM$, where $M\subset M\x M$ is the diagonal, and the arrow is the map $\nu_M^*$, which is a quasi-isomorphism by excision.  

Out of this data, we can give the following chain level description of the intersection product on $H_*(M)$: 
\begin{equation}\label{equ:int}
  \bullet\colon C_p(M)\ot C_q(M)\xrightarrow{\x}  C_{p+q}(M\x M)\xrightarrow{[\tau_M \cap]}  C_{p+q-n}(U_M) \xrightarrow{r}  C_{p+q-n}(M),
\end{equation}
where the middle map is the following composition:
\begin{equation}\label{equ:cap}
[\tau_M \cap]: C_*(M\x M) \to C_*(M\x M,M\x M\backslash M) \xrightarrow{\sim} C_*(U_M,U_M\backslash M) \xrightarrow{\tau_M \cap}  C_{*-n}(U_M),
\end{equation}
with the middle map being a homotopy inverse to excision, as can be obtained, for example, by subdividing simplices. 
(To be precise, this definition differs by a sign from the intersection product defined as the Poincar\'e dual of the cup product, see eg.~\cite[Proposition B.1]{HinWah0}.)

An important property of the intersection product, for computational purposes, is that it can indeed be computed by geometric intersection for homology classes that can be represented by transverse embedded submanifolds: if $A,B\subset M$ are embedded transverse submanifolds of $M$, with $[A]\in H_p(M)$ and $[B]\in H_q(M)$ the corresponding homology classes, then
$$[A]\bullet [B]=[A\cap B]\in H_{p+q-n}(M). $$
See eg.~\cite[VI Theorem 11.9]{Bredon}.

\subsection{Definition of the product and coproduct as lifts of the intersection product}\label{sec:defproco}

Let $ev_0:LM\to M$ denote the evaluation at $0$. 
The Chas-Sullivan product $\wedge$ being a lift of the intersection product $\bullet$ means that both products should fit in a commutative diagram of the form
\begin{equation}\label{equ:CSlift}
\xymatrix{H_p(LM)\ot H_q(LM) \ar[r]^-{\wedge}\ar[d]_{\ev_0\ot \ev_0} & H_{p+q-n}(LM)\ar[d]^{\ev_0}\\
  H_p(M)\ot H_q(M) \ar[r]^-{\bullet} & H_{p+q-n}(M)
}\end{equation}
We explain now how this can be achieved simply by ``pulling back'' all the ingredients of the above definition of the intersection product to the loop space along the evaluation map $\ev_0\x \ev_0$.

Recall from above the $\eps$--neighborhood $U_M$ of the diagonal in $M\x M$ and define $$U_{\CS}=(e\x e)^{-1}U_M=\big\{(\ga,\la)\in LM\x LM \ |\ |\ga(0)-\la(0)|<\eps\big\}.$$
The retraction $r:U_M\to M$ lifts to a retraction 
\begin{equation*}
  R_{\CS}\colon U_{\CS}\longrightarrow \Figeight =\big\{(\ga,\la)\in LM\x LM \ |\ \ga(0)=\la(0)\big\}= (\ev_0\x \ev_0)^{-1}(M\subset M\x M)
\end{equation*}%
by concatenating with a geodesic stick to connect the loops so that they form a ``figure 8'': 
\begin{equation*}
R_{\CS}(\gamma ,\lambda )=(\gamma \ ,\ \overline{\gamma(0)\lambda(0)}\star\lambda \star \overline{\lambda(0)\gamma(0)})
\end{equation*}%
where, for $x,y\in M$ with $|x-y|<\rho $, $\overline{xy}$ denotes the unique
minimal geodesic path $[0,1]\rightarrow M$ from $x$ to $y$, which is possible by our choice of $\eps$,  
and $\star $ is the concatenation of paths.\footnote{See eg., \cite[Sec 1.2]{HinWah0} for a definition of an associative concatenation.}
See also Figure~\ref{fig:sticks}(a).

\begin{figure}[h]
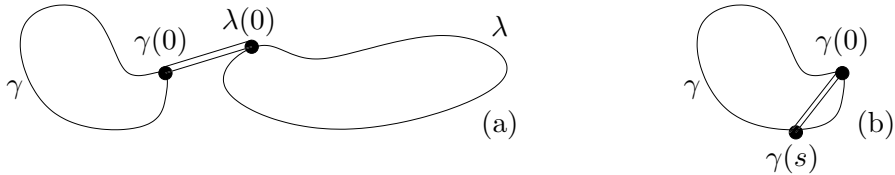

\centering
\begin{lpic}{sticks(0.55,0.55)}
\lbl[b]{-1,10;$\ga$}
\lbl[b]{34,20;$\ga(0)$}
\lbl[b]{55,25;$\lambda(0)$}
\lbl[b]{115,25;$\lambda$}
\lbl[b]{161,10;$\ga$}
\lbl[b]{197,20;$\ga(0)$}
\lbl[t]{185,-1;$\ga(s)$}
\lbl[b]{115,0;(a)}
\lbl[b]{205,0;(b)}
\end{lpic}
\caption{The retraction maps $R_{\CS}$ and $R_{\GH}$.}
\label{fig:sticks}
\end{figure}

Pulling back our representative of the  Thom class $\tau_M$ along the evaluation map gives a cochain 
$$\tau_{\CS}:=(e\x e)^*\tau_M\in C^*(LM\x LM,\Figeight^c).$$ 
Together, $U_{\CS}, R_{\CS}$ and $\tau_{\CS}$ are all the ingredients we need to define the desired product:
\begin{definition}\label{def:GeoPro}
The following sequence of chain maps is a chain model for the Chas-Sullivan product: 
\begin{multline}\label{equ:CS}
 \wedge\colon C_p(LM)\ot C_q(LM) \xrightarrow{\x} C_{p+q}(LM\x LM)\xrightarrow{[\tau_{\CS}\cap]} C_{p+q-n}(U_{\CS}) \\
  \xrightarrow{R_{\CS}} C_{p+q-n}(\Figeight)  \xrightarrow{\concat} C_{p+q-n}(LM),
\end{multline}
where, just as in (\ref{equ:int}), the middle map is the composition of an homotopy inverse to excision followed by the capping map.
\end{definition}
%a little more subtle than may at first appear at it requires a chain homotopy inverse to excision, as may be obtained by going to ``small chains'', that only has good naturality properties in homology. See eg. \cite[Sec A.2]{HinWah0}. 
%more precisely given by the following composition:
%$$\tau_{CS}\cap:  C_{*}(LM\x LM) \rar C_{*}(LM\x LM,U_{CS,\eps_0}^c) \rar C_{*}(U_{CS},U_{CS,\eps_0}^c)  \xrightarrow{\tau_{CS}\cap} C_{*-n}(U_{CS})$$ 

Naturality of the maps gives that the resulting homology product on the homology $H_*(LM)$ makes
Diagram~(\ref{equ:CSlift}) commute. And it is shown in \cite[Proposition 2.4]{HinWah0} that this simple minded chain description of the Chas-Sullivan product agrees in homology with the definition of Cohen-Jones \cite{CohJon} given in terms of a tubular neighborhood of the figure 8 space $\Figeight$ inside $LM\x LM$. 
 
\medskip

The coproduct can be defined completely analogously, replacing the evaluation map $\ev_0\x \ev_0: LM\x LM \to M\x M$ by  the evaluation map
$$e_I: LM\x I\to M\x M \ \ \ \textrm{defined by}\ \ e_I(\ga,s)=(\ga(0),\ga(s)).$$
Indeed, setting
 $$U_{\GH}=e_I^{-1}U_M=\big\{(\ga,s)\in LM\x I \ |\ |\ga(0)-\ga(s)|<\eps\big\},$$
we again have a retraction map
\begin{equation*}
  R_{\GH}\colon U_{\GH}\longrightarrow \F = \big\{(\ga,s)\in LM\x I \ |\ \ga(0)=\ga(s) \big\}= e_I^{-1}(M\subset M\x M)
\end{equation*}%
by concatenating with a geodesic stick to force a self-intersection: 
\begin{equation*}
R_{\CS}(\gamma ,s )=(\gamma[0,s]\star\overline{\gamma(s)\ga(0)} \star_s \overline{\ga(0)\gamma(s)} \star \ga[s,0]\ ,\ s )
\end{equation*}%
where we choose the parametrization of the concatenated loop so that it exactly passes through $\ga(0)$ at time $s$; this is possible even if $s=0$ or $1$  as in that case $\ga(0)=\ga(s)$ to begin with and the geodesic sticks are thus length $0$. See also Figure~\ref{fig:sticks}(b).  

We can consider the sequence of maps
\begin{multline*}
  C_p(LM) \xrightarrow{\x I} C_{p+1}(LM\x I)\xrightarrow{[\tau_{\GH}\cap]} C_{p+1-n}(U_{\GH}) 
  \xrightarrow{R_{\GH}} C_{p+1-n}(\F)  \xrightarrow{\cut} C_{p+q-n}(LM\x LM).
\end{multline*}
totally analogous to the maps (\ref{equ:CS}) defining the product above. 
The only new issue that arises in the coproduct compared to the product is that the first map in the sequence, crossing with an interval, is not a chain map because the interval has non-trivial boundary. 
This corresponds to the fact that the operation is now parametrized by an interval $I$. To obtain an induced operation on homology, we need to appropriately kill the resulting ``boundary operation'' at the endpoints of the interval. The simplest way to do this is to consider the operation as a relative operation, noting that, when $s=0$ or $1$, the above sequence of maps creates a left or right constant loop.
\begin{definition}\label{def:GeoCo}
  The following sequence of chain maps is a chain model for the Goresky-Hingston-Sullivan coproduct: 
\begin{multline}\label{equ:GH}
 \vee\colon  C_p(LM,M) \xrightarrow{\x I} C_{p+1}(LM\x I,LM\x \del I \cup M\x I)\xrightarrow{[\tau_{\GH}\cap]} C_{p+1-n}(U_{\GH}, LM\x \del I \cup M\x I) \\
  \xrightarrow{R_{\GH}} C_{p+1-n}(\F,LM\x \del I \cup M\x I)  \xrightarrow{\cut} C_{p+q-n}(LM\x LM,M\x LM\cup LM\x M)
\end{multline}
\end{definition}
 This sequence of maps now indeed induces a well-defined degree $1-n$ coproduct on $H_*(LM,M)$:
$$\vee: H_p(LM,M)\rar H_{p+1-n}(LM\x LM,M\x LM\cup LM\x M);$$
if we work with field coefficients, the target is isomorphic to  $H_*(LM,M)^{\ot 2}$. It is shown in \cite[Proposition 2.12]{HinWah0} that this chain level description of the Goresky-Hingston-Sullivan coproduct agrees with the definition given in \cite{GorHin} using a tubular neighborhood of $\F$ inside $LM\x I$ away from the boundary $LM\x \del I$, together with a limit argument reach to the boundary. 

Applying the evaluation map $e_I$ gives a diagram of the same form as Diagram~(\ref{equ:CSlift}), with the coproduct replacing the Chas-Sullivan product on the top row, but now with intersection product {\em relative to $M$} on the bottom row, which is a trivial operation! Hence there is no formal way in which the homology loop coproduct is a lift of the homology intersection product.  We will however see in Section~\ref{sec:computations} that the coproduct still can be computed by an appropriate geometric intersection, for nice enough geometric cycles, away from the ``trivial self-intersections'' coming from constant loops or from the intersection times  $s=0$ and $s=1$. 

%\Nnote{remove as discussed below?} Note that the coproduct obtained by cutting the loops precisely at time $s=\frac{1}{2}$ only is trivial \cite{Tam}, essentially because it is homotopic to the coproduct coming from cutting at time $0$ or $1$ instead.

\begin{remark}[Lifting the coproduct to a non-relative operation]\label{rem:lifts}
There exists several ways to lift  the coproduct $\vee$ to a non-relative operation.
\begin{enumerate}
\item One such lift is the {\em extension by zero} of \cite[Sec 4]{HinWah0}, that uses the splitting $H_*(LM)\cong H_*(LM,M)\oplus H_*(M)$ coming from the inclusion of the constant loops and the evaluation  $\operatorname{cst}:M\rightleftarrows LM :\ev_0$, declaring the coproduct to be zero on constant loops.

\item If the Euler characteristic of the manifold is zero, one can instead use a nowhere vanishing vector field $\overline v$ to define such an extension, by replacing the diagonal $M\subset M\x M$ in the above definition of the coproduct, with the homotopy equivalent subspace $\De_{\overline v}M=\{(m,\exp_m\overline v_m)\in M\x M\ |\ m\in M\}$. Indeed, if the vector field has no zeros, the coproduct will then automatically be trivial at the special points with $s=0$ or $s=1$. See also \cite[Sec 3.4]{NaeWil19} for an analogous definition of a lifted coproduct in the $\chi(M)=0$ case, using instead a lift of the Thom class. %\Fnote{changed the word "different" to "analogous", the two definitions most likely agree. The word "different" could be misread as "probably inequivalent".}

  If the Euler characteristic is not zero, one can instead pick a vector field vanishing only in the neighborhood of a single point, which will yield a coproduct in reduced homology of the loop space instead, corresponding to what we will see in the algebraic version of the coproduct, see Definition~\ref{defn:algebraic coproduct}. 
%
%  Different choices of vector fields may give different lifts; this choice is equivalent to finding a lift of $\tau_M\in H^n(M\x M,U_M^c)\cong H^n(TM,STM)$  to a  class
%in $H^{n-1}(STM)$ in the long exact sequence of the pair $(TM,STM)$, which has space of choices given by $H^{n-1}(TM)\cong H^{n-1}(M)\cong H_1(M)$. 

\item The following variant of the previous idea has been described for the case of surfaces in \cite{kawazumi2014regular}. Instead of attaching the non-vanishing vector field to the manifold $M$ one can attach it to the loop. That is one considers loops in the unit tangent bundle of $M$. In the case of surfaces, such loops can be identified with regular homotopy classes of immersed curves. Moreover, in case the surface has a non-vanishing vector field, the above construction is recovered by using that every homotopy class of a loop in a surface has a unique representative as an immersed loop with rotation number $0$ with respect to the vector field. This is the point of view taken in \cite{alekseev2018goldman}.

%\Nnote{Something about curves of 0 winding number in surfaces? Florian, can you give a ref for that?}
  
\item  As we will see in Section~\ref{sec:Tate2} in the algebraic context, following the paper \cite{RivWan19} (see \cite{CiFrOa10, CieOan20} for a geometric version), the loop product and coproduct  together define a single (non-relative) product on the {\em Tate-Hochschild complex}, a complex that combines both the chains and cochains of the loop space, attached together using the Euler class (see Section~\ref{sec:Tate}). When the Euler characteristic of the manifold vanishes, the Tate complex splits and this recovers a non-relative cohomology product, dual to the homology coproduct.

\end{enumerate}
\end{remark}

\subsection{Computation via geometric intersections}\label{sec:computations}
Recall that two smooth maps $f:X\to M$ and $g:Y\to M$ are {\em transverse} if for every $x,y$ such that $f(x)=m=g(y)$,  we have $f_*T_xX+ g_*T_yY =T_mM$. 
Because the product and coproduct are defined as lifts of the intersection product along evaluation maps, they can both be computed by geometric intersection, under appropriate transversality assumptions on the cycles representing the homology classes:
\begin{proposition} \cite[Propositions 3.1 and 3.7]{HinWah0}\label{prop:compute}
\begin{enumerate}
\item If  $Z_1\colon \Si_1\to LM$ and $Z_2\colon \Si_2\to LM$ are smooth cycles  with the property that the maps  $\ev_0\circ Z_1\colon \Si_1\to M$ and $\ev_0\circ Z_2\colon \Si_2\to M$ are transverse, then the loop product
  $$Z_1\wedge Z_2=(Z_1\star Z_2)|_{\Si_1\x_{\ev_0} \Si_2} \ \in \ H_*(LM)$$
  is the concatenation of the loops of $Z_1$ and $Z_2$ along the locus of basepoint-intersections $\Si_1\x_{\ev_0} \Si_2\subset \Si_1\x \Si_2$, oriented as stated in \cite{HinWah0}.
\item   If $Z\colon (\Sigma,\Si_0) \rightarrow  (LM,M)$ is a smooth relative cycle with the property that the restriction of $e_I\circ (Z\times I): \Si\x I \to M\x M$ to $(\Si\minus \Si_0)\x (0,1)$ is  transverse to the diagonal, 
%\begin{align*}
%E(Z):=e_I\circ (Z\times I)|_{(\Si\x I)\minus \Si_\B } \colon (\Si\x I)\minus \Si_\B & \ \longrightarrow\ M\times M\\
%(\sigma ,t) &\ \ \mapsto\ \; \big(Z(\sigma )(0),Z(\sigma)(t)\big)
%\end{align*}
 %is a smooth map transverse to the diagonal $\Delta\colon M\to M\x M$. 
 then 
\begin{equation*}
\vee Z=\cut\circ (Z\times I)|_{\overline{\Si_{\De}}}\in H_{*}(LM\times LM,M\times LM \cup LM\times M)
\end{equation*}
for $\overline{\Si_\De}$ the closure in $\Si\x I$ of the locus of basepoint self-intersecting loops $\Si_\Delta \subset \ (\Si\minus \Si_0)\x (0,1)$,
oriented as stated in \cite{HinWah0}.
% :=(e_I\circ (Z\times I))^{-1}(\De M)\  
  \end{enumerate}
  \end{proposition}

We illustrate this proposition here through a loop product and coproduct computation for  lens spaces. The coproduct computation will be used in Section~\ref{sec:invariance0} to show that the coproduct is not homotopy invariant, following \cite{Nae21}. 

\medskip

Let $S^3$ be the $3$-sphere, considered as the unit sphere in $\mathbb{C}^2$. We will write elements of $S^3$ %either as pairs of complex numbers $(z_1,z_2)$, or
in spherical coordinates as tuples $(\un r,\un \theta)=((r_1,\theta_1),(r_2,\theta_2))$ with $\theta_i\in \R/\Z$ and $r_i\ge 0$, satisfying $r^2_1+r^2_2=1$. 
The lens space $\Le_{p,q}$, for $p,q$ coprime, is the quotient of $S^3$ by the relation 
%$$(z_1,z_2)\ \sim\  (e^{\frac{2\pi i}{7}}z_1,e^{\frac{2k\pi i}{7}}z_2),$$ 
%or in spherical notation
$$((r_1,\theta_1),(r_2,\theta_2)) \ \sim \ ((r_1,\theta_1+\frac{1}{p}),(r_2,\theta_2+\frac{q}{p})).$$
This relation comes from the action of the torus $S^1\x S^1$ on $S^3\subset \mathbb{C}^2$ rotating each coordinate, where we have picked a particular subgroup $\Z/p$ inside $S^1\x S^1$.  Note that there is a residual torus action on the lens space: 
$$\begin{array}{llcl}\alpha: &(S^1\x S^1) \x \Le_{p,q} &\to &\Le_{p,q},\\
                             & ((s,t),(\un r,\un \theta)) &\mapsto & %\alpha_{s,t}(\un r,\un \theta):=
                                                                     ((r_1,\theta_1+\frac{s}{p}),(r_2,\theta_2+\frac{sq}{p}+t)).
                             \end{array}$$
%taking $((s,t),(\un r,\un \theta))$ to
%$$\alpha_{s,t}(\un r,\un \theta):=((r_1,\theta_1+\frac{s}{p}),(r_2,\theta_2+\frac{sq}{p}+t)).$$ 
In particular, given a point $(\un r,\un \theta)$ on the lens space, and a pair of integers $(\ell,m)$, %and a rational direction $(\ell,m)$ in the torus,
we get a loop $t\mapsto (\ell t,m t)$ in the torus, and hence a loop $\ga^{\ell,m}_{\un r,\un \theta}$ in the lens space  based at  $(\un r,\un \theta)$ by composing with the action. Explicitly,   the loop $\ga^{\ell,m}_{\un r,\un \theta}: S^1\to \Le_{p,q}$ is defined by 
$$\gamma^{\ell,m}_{\un r,\un \theta}(t)= ((r_1,\theta_1+\frac{\ell t}{p}),(r_2,\theta_2+\frac{q\ell t}{p}+mt)).$$
As the action is continuous, varying the startpoint $(\un r,\un \theta)$ of the loop gives a family of loops $\rho_{\ell,m}\colon \Le_{p,q} \to L\Le_{p,q}$ parametrized by the lens space itself, and hence 
%defines the following family of classes in the homology of the loop space $L\Le_{p,q}$: 
 for each pair of integers $(\ell,m)$ a class
 $$\rho_{\ell,m}\in H_3(L\Le_{p,q}),$$
 where we use the same notation $\rho_{\ell,m}$. 
Note that each class $\rho_{\ell,m}$ is non-trivial as it maps to the fundamental class of $\Le_{p,q}$ under the evaluation map
$$\ev_0: H_3(L\Le_{p,q}) \rar H_3(\Le_{p,q}).$$
We will here compute the loop products and coproducts of these classes.

\smallskip

We consider first the product:
$$\wedge\colon H_3(L\Le_{p,q})\ot H_3(L\Le_{p,q}) \rar H_{3+3-3}(L\Le_{p,q})=H_3(L\Le_{p,q}).$$
%evaluated on the above classes. 

\begin{proposition}\label{prop:prodrho} %Let $\rho_{\ell,m}\in H_3(L\Le_{p,q})$ be as defined above, for $(\ell_i,m_i)$ pairs of coprime natural numbers.
  The Chas-Sullivan loop product of the classes $\rho_{\ell,m}\in H_3(L\Le_{p,q})$ defined above, is given by summing the indices: 
   $$\rho_{\ell_1,m_1}\wedge \rho_{\ell_2,m_2}=\rho_{\ell_1+\ell_2,m_1+m_2}.$$
\end{proposition}

\begin{proof}
The cycles $\rho_{\ell,m}:\Le_{p,q}\to L\Le_{pq}$ are smooth cycles parametrized  $\Le_{p,q}$.  To apply Proposition~\ref{prop:compute}, we need to check that the maps $$\Le_{p,q} \xrightarrow{\rho_{\ell_i,m_i}}L\Le_{p,q}\xrightarrow{\ \ev_0\ } \Le_{p,q}$$ are transverse. But for each $(\ell_i,m_i)$, this composition is the identity on the lens space, so the maps are certainly transverse,  and the locus of basepoint-intersections is the diagonal $\De\Le_{p,q}\subset \Le_{p,q}\x \Le_{p,q}$. The product is thus explicitly given by 
$$\rho_{\ell_1,m_1}\wedge \rho_{\ell_2,m_2}=(\rho_{\ell_1,m_1}\star \rho_{\ell_2,m_2})|_{\De\Le_{p,q}}\colon \Le_{p,q}\equiv \De \Le_{p,q}\rar L\Le_{p,q}$$
for $\star$ the concatenation of the loops in the image at their common basepoint. At each point $(\un r,\un \theta)$ in $\Le_{p,q}$, we are thus left to compute the concatenation
$\ga^{\ell_1,m_1}_{\un r,\un \theta}\star \ga^{\ell_2,m_2}_{\un r,\un \theta}$
%\colon t\mapsto \left\{\begin{array}{ll}((r_1,\theta_1+\frac{\ell_1 (2t)}{p}),(r_2,\theta_2+\frac{q\ell_1 (2t)}{p}+m_1(2t))) & 0\le t \le \frac{1}{2} \\
%                                                                                 ((r_1,\theta_1+\frac{\ell_2 (2t-1)}{p}),(r_2,\theta_2+\frac{q\ell_2 (2t-1)}{p}+m_2(2t-1))) & \frac{1}{2}\le t\le 1 
%                                                                               \end{array}\right.$$
which is exactly the image under the torus action of the concatenation of the loops $(\ell_1,m_1)$ and $(\ell_2,m_2)$ in the torus. This concatenation in the torus is homotopic to the loop  $(\ell_1+\ell_2,m_1+m_2)$ (corresponding to the fact that $\pi_1(S^1\x S^1)\cong \Z\x \Z$) and hence the above product is homotopic the loop $\ga^{\ell_1+\ell_2,m_1+m_2}_{\un r,\un \theta}$.  As this homotopy originates in the torus, it defines a continuous homotopy over the lens space. It follows that the Chas-Sullivan product of such classes is  as claimed. 
  \end{proof}

  The coproduct of homology classes of degree 3 in $L\Le_{p,q}$ is a map
  $$\vee: H_3(L\Le_{p,q},\Le_{p,q})\rar H_{1}(L\Le_{p,q}\x L\Le_{p,q},\Le_{p,q}\x L\Le_{p,q}\cup L\Le_{p,q}\x \Le_{p,q}).$$
 % given that $3+1-3=1$ for the degree of the target.
For the classes $\rho_{\ell,m}$, it will given in terms of $\beta$--classes in the target, that we describe now.

  Let $\la: S^1\to \Le_{p,q}$ be the loop defined by 
  $\lambda(t)=((1,\frac{t}{p}),0)$, tracing the points $(\un r,\un \theta)\in \Le_{p,q}$ with $r_2=0$. This is a generator of  $\pi_1\Le_{p,q}\cong \Z/p$.
  Note that $\la=\ga_{((1,0),0)}^{1,0}$ is the evaluation of the class $\rho_{1,0}$ at $((1,0),0)\in \Le_{p,q}$.
  In particular, it is freely homotopic to  $\ga_{(0,(1,0))}^{1,0}$, the evaluation of  $\rho_{1,0}$ at  $(0,(1,0))$, where we note that  $\ga_{(0,(1,0))}^{1,0}=(\la')^{\star q}$ for  $\la': S^1\to \Le_{p,q}$ defined by 
  $\lambda'(t)=(0,(1,\frac{t}{p}))$, the loop tracing the points $(\un r,\un \theta)$  with $r_1=0$.

%  An explicit homotopy $\la\simeq_{H}\la'$ can be obtained by setting $H(\tau,t)=((\tau,\frac{t}{p}),(1-\tau,\frac{qt}{p}))$.  \Fnote{the two maps are $\rho_{0,1}$ at two points, so this gives a homotopy, $\Le_{p,q}$ being path-connected.}
  
The coproduct of the classes $\rho_{\ell,m}$ will be given by applying the cut map  to families of figure eights 
  $$\be_{k,k'}: S^1\to L\Le_{p,q}\x_{\Le_{p,q}} L\Le_{p,q}\subset  L\Le_{p,q}\x L\Le_{p,q},$$
  based at the points of $\la$ and  defined by $\be_{k,k'}(t)= [s\mapsto ((1,\frac{t+ks}{p}),0)]\star[s\mapsto ((1,\frac{t+k's}{p}),0)]$.
  We denote likewise by $\beta_{k,k'}\in H_1(L\Le_{p,q}\x L\Le_{p,q})$
the associated homology class. 
 
%\Fnote{I would prefer to break up the following lemma, and the preceding definitions. First define $\beta$'s as classes in $H^1(LM \times LM)$ represented by beta. Then say they can also be represented by $\beta^\prime$'s. Then say that they are linearly independent over $F_p$ using evaluation at $0$ on each component)}

Just like the loop $\la$ is freely homotopic to $\la'^{\star q}$, the family of figure eights $\beta_{k,k'}$ is freely homotopic to a family loops with basepoints parametrized by $\la'^{\star q}$: 
  \begin{lemma}\label{lem:beta'}
 Let $\be'_{k,k'}: S^1\to L\Le_{p,q}\x_{\Le_{p,q}} L\Le_{p,q}\subset  L\Le_{p,q}\x L\Le_{p,q}$
 be the family of figure eights based at the points of $\la'$ defined by $\be'_{k,k'}(t)= [s\mapsto (0,(1,\frac{t+ks}{p}))]\star[s\mapsto (0,(1,\frac{t+k's}{p}))]$. 
 Then $$\beta_{k,k'}=q\beta'_{qk,qk'}\in H_1(L\Le_{p,q}\x L\Le_{p,q}).$$ 
\end{lemma}

\begin{proof}
 % We can lift the homotopy  $\la\simeq_{H}\la'$ to a
An explicit free homotopy $\beta_{k,k'}\simeq_{H}\beta'_{qk,qk'}$ between the families of loops is given by setting
$H(\tau,t)= [s\mapsto ((\tau,\frac{t+ks}{p}),(1-\tau,\frac{q(t+ks)}{p}))]\star[s\mapsto ((\tau,\frac{t+k's}{p}),(1-\tau,\frac{q(t+k's)}{p}))]$.
  \end{proof}

  \begin{lemma}\label{lem:beta}
 %   Each class $\be_{k,k'}\in H_1(L\Le_{p,q}\x L\Le_{p,q})$ is non-zero,
    We have that
\begin{enumerate}
\item   $\beta_{k,k'}=\beta_{h,h'}\in H_1(L\Le_{p,q}\x L\Le_{p,q})$ if and only if $k=h\!\mod p$ and $k'=h'\!\mod p$.
    %, and $\beta_{k,k'}=\beta'_{qk,qk'}$ in homology for all $k,k'$.
\item  The relative classes $$\{\be_{k,k'}\}_{\substack{ 0<k<p\\ 0<k'<p}}\in H_1(L\Le_{p,q}\x L\Le_{p,q},\Le_{p,q}\x L\Le_{p,q}\cup L\Le_{p,q}\x \Le_{p,q})$$
  are linearly independent over $\mathbb{Z}_p$. 
%\Fnote{It seems we did not assume $p$ to be prime. It's hopefully still clear if we just say linearly independent over $\Z_p$?}
  %is non-zero precisely when $k,k'\neq 0\!\mod p$.
%\item The classes $\{\beta_{k,k'}\}_{\substack{ 0\le k\le p-1\\ 0\le k'\le p-1}}$ are linearly independent in $H_1(L\Le_{p,q}\x L\Le_{p,q};\mathbb{F}_p)$. 
  \end{enumerate}
    \end{lemma}

    \begin{proof}
      The evaluation at 0 takes the family of figure eights $\beta_{k,k'}$ to the generator $\la$ of $\pi_1(\Le_{p,q})\cong \Z/p$. Hence the map $H_1(L\Le_{p.q}\x L\Le_{p,q})\to H_1(\Le_{p,q})$ projecting on the first component and evaluating at 0, takes $\be_{k,k'}$ to the generator of $H_1(\Le_{p,q})$. In particular, each class  $\beta_{k,k'}\in H_1(L\Le_{p,q}\x L\Le_{p,q})$ is non-trivial. 
      
Note now that $\beta_{k,k'}$ has image in the component $(k\!\mod p,k'\!\mod p)$ of the space $L\Le_{p,q}\x L\Le_{p,q}$, as each loop  $[s\mapsto ((1,\frac{t+ks}{p}),0)]$ is homotopic to  $\la^{\star k}$.
Given that the classes are non-zero, $\beta_{k,k'}=\beta_{h,h'}$ thus  necessarily requires that $k=h\!\mod p$ and $k'=h'\!\mod p$. 
The converse follows from the fact that any homotopy $\la^{\star p}\simeq *$ extends continuously over such a family of loops, using the residual torus action to push it along $\la$, proving that $\beta_{k,k'}=\beta_{k+np,k'+mp}$ in homology for any $n,m\in \NN$, which proves (1).

Finally,
%for the non-triviality statements, the class $\be_{k,k'}\in H_1(L\Le_{p,q}\x L\Le_{p,q})$
% is always non-trivial, because its evaluation at 0  $\ev_0\circ \beta_{k,k'}=\la$ is a generator of $H_1(\Le_{p,q})\cong \Z/7$. It is thus
by the above, $\beta_{k,k'}$ is 
 non-zero in relative homology precisely when $k$ and $k'$ are not equal to $0$ mod $p$, as $\beta_{0,k'}$ and $\beta_{k,0}$ being trivial in relative homology.  And the classes are linearly independent as they live in different components. 
\end{proof}

We are now ready to compute the coproduct of $\rho$--classes, where we will assume that $\ell$ and $m$ are positive for simplicity.

\begin{proposition}\label{prop:vrho}
  The coproduct of the class $\rho_{\ell,m}\in H_3(L\Le_{p,q},\Le_{p,q})$ with $\ell,m\ge 0$  is given by the formula
  $$\vee\rho_{\ell,m} =\sum_{\substack{0<t<\ell \\ t,
      (\ell-t)\neq 0\!\!\!\mod p}}\beta_{t,\ell-t} \ \ + \ \ q'\!\!\!\!\!\! \sum_{\substack{0<t<q\ell+pm \\ t,(\ell-tq')\neq 0\!\!\!\mod p}}\beta_{tq',\ell-tq'}$$
where $q'$ is the multiplicative inverse of $q$ mod $p$. 
    \end{proposition}
%\Fnote{The formula is only true for $l,m$ non-negative (or at least otherwise one has to read it carefully)}

    Using the previous lemma, one deduces that the coproduct of $\rho$--classes is non-trivial most of the time.
    %\Nnote{quick computation one could do: is it the case that the coproduct is trivial iff the classes are represented by simple loops?}

\begin{proof}

  To compute the coproduct $\vee \rho_{\ell,m}$ by geometric intersection applying Proposition~\ref{prop:compute}, we need the map
$$\Le_{p,q}\x (0,1)\xrightarrow{\rho_{\ell,m}\x \id} L\Le_{p,q}\x (0,1) \xrightarrow{\ e_I\ } \Le_{p,q}\x \Le_{p.q},$$
where $e_I$ evaluates the loops at 0 and $s\in (0,1)\subset I$, to be transverse to the diagonal embedding $\De: \Le_{p,q}\to \Le_{p,q}\x \Le_{p.q}$ after removing the locus of constant loops.
% Indeed, if we include the constant loops and the ``trivial times'' $s=0$ or $1$, such maps are essentially never transverse to the diagonal, but Proposition 3.7 of \cite{HinWah0} states that it is enough to have transversality away these ``singular'' loci.
In the present case, either  $(\ell,m)= (0,0)$ in which case all loops are constant, with $\rho_{(0,0)}=0$ in homology relative to the constant loops, or  $(\ell,m)\neq (0,0)$ and the cycle has no constant loop in its image. So we can assume $(\ell,m)\neq (0,0)$ and work with the parametrizing pair $(\Si,\Si_0)=(\Le_{p,q},\emptyset)$ for our relative cycle.

To achieve transversality, we represent the homology class of  $\rho_{\ell,m}$ by the homotopic family $\tilde\rho_{\ell,m}:\Le_{p,q}\to L\Le_{p,q}$ defined by $\tilde\rho_{\ell,m}(\un r,\un \theta)=\tilde\ga^{\ell,m}_{\un r,\un \theta}$ for
   $\tilde\ga^{\ell,m}_{\un r,\un \theta}:S^1\to \Le_{p,q}$ the loop based at  $(\un r,\un \theta)$ given by
   $$\tilde\gamma^{\ell,m}_{\un r,\un \theta}(t)= \big((\tilde r_1(t),\theta_1+\frac{\ell t}{p}),(\tilde r_2(t),\theta_2+\frac{q\ell +pm}{p}t)\big),$$
   where $(\tilde r_1(t),\tilde r_2(t))$ is a deformation of $(r_1,r_2)$ with   $(\tilde r_1(t),\tilde r_2(t))=(r_1,r_2)$ only when $r_1$ or $r_2=0$, or when $t=0$ or $1$.
   % This possible as $(r_1,r_2)$ are positive numbers satisfying that $r_1^2+r_2^2=1$, and hence live in an arc in the circle, which we can indentify with $I$.
   Such a deformation can be obtained by e.g., interpolating back and forth between the identity on $r_1$ at times $t=0$ and $1$  and $r_1^2$ at $t=\frac{1}{2}$, with $\tilde r_2(t)=\sqrt{1-\tilde r_1(t)^2}$.
 %  Note that the chosen deformation of the radii is maximal in the sense that loops $\ga$ are deformed to loops $\tilde \ga$ as the deformation is trivial at times $t=0,1$, and because there is no continuous deformation of $(\un r,\un \theta)$ over $\Le_{p,q}$ taking a point with $r_i=0$ to some $(\un{\tilde r},\un{ \tilde \theta})$ with $\tilde r_i>0$ as such a deformation requires choosing an accompanying $\theta_i$, which is undefined when $r_i=0$. 

   The map $e_I\circ (\tilde\rho_{\ell,m}\x \id)|_{\Le_{p,q}\x (0,1)}$ intersects the diagonal whenever a loop $\tilde\ga_{\un r,\un \theta}^{\ell,m}$ has a self-intersection $\tilde\ga_{\un r,\un \theta}^{\ell,m}(0)=\tilde\ga_{\un r,\un \theta}^{\ell,m}(t)$ for some $t\in (0,1)$.
   Such self-intersections can only happen when $r_1=0$ or $r_2=0$, as otherwise $\tilde r_i(t)\neq r_i(0)$, and when 
   $$\left\{\begin{array}{ll} 0<t=\frac{a}{\ell}<1 & \mathrm{if}\ r_2=0\\
                   0<t=\frac{b}{q\ell+pm}<1 & \mathrm{if}\ r_1=0
                   \end{array}\right.$$
                 for some $a,b\in \NN$. That is the locus of self-intersections of $\tilde\rho_{\ell,m}\x \id|_{\Le_{p,q}\x (0,1)}$ is
$$\Sigma_{\De}=\la\x I_1\cup \la'\x I_2\subset \Le_{p,q}\x (0,1)$$
for $I_1=\{\frac{1}{\ell},\dots,\frac{\ell-1}{\ell}\}$ and $I_2=\{\frac{1}{q\ell+pm},\dots,\frac{q\ell+pm-1}{q\ell+pm}\}$, 
 and $\la$, $\la'$  the loops  parametrizing the points with $r_2=0$ and $r_1=0$ respectively, as above. 
%Note that each of these intersections come in 1-parameter families, parametrized by the loop $\la'$ defined above if $r_1=0$, with $r_2=1$ and $\theta_2\in S^1$ the only free variable, and by the loop $\la$ if $r_2=0$, with $\theta_1$ the only variable. 

 We need to check that these self-intersections are transverse to the diagonal. This can be checked in local coordinates $(r_1,\theta_1,\theta_2,t)=(z,\theta_2,t)\in \mathbb{C}\x \mathbb{R}^2$ around points with $r_1=0$, and similarly with coordinates $(\theta_1,r_2,\theta_2,t)=(\theta_1,z,t)$ when $r_2=0$. In those coordinates, the function $e_I\circ (\tilde\rho_{\ell,m}\x \id)$ has the form
\begin{align*}
(z,\theta,t) \mapsto &((z,\theta), (e^{2\pi i\alpha t}r(t) z, \theta + \beta t)) %\\
%&= (z, \theta, e^{\alpha t} z r(t), \theta + \beta t),
\end{align*}
where $r(t)$ is a function so that $r(t) = 1$ only for $t = 0,1$. It is transversal to all the diagonals
\[
\Delta_k = (z,\theta, e^{2\pi i \frac{k}{p}}z, \theta + \frac{k q}{p}) \ \ \ \textrm{resp.} \ \ \ (z,\theta, e^{2\pi i \frac{k}{p}}z, \theta + \frac{k q'}{p}) 
\]
because the zeros of the functions
%\begin{align*}
$f_k (z, \theta, t) = ((e^{2\pi i\alpha t}  r(t)-e^{2\pi i \frac{k}{p}}) z, \beta t- \frac{kq^{(')}}{p})$
%\end{align*}
are transversal. Indeed, away from $t=0,1$ the factor $(e^{2\pi i\alpha t}  r(t)-e^{2\pi i\frac{k}{p}})$ is never zero,
so, up to translation, $f_k$ has the form  $f_k(z,\theta,t)=(a(t)z,\beta t)$ for $0\neq a(t)\in \mathbb{C}$ and $\beta>0$, either equal to $\frac{\ell}{p}$ or to $\frac{q\ell+pm}{p}$. 
%so the derivatives
%\[
%\partial_z f_k = (e^{2\pi i \frac{k}{p}} - e^{2 \pi i\alpha t}  r(t), 0) 
%\ \ \ \textrm{and} \ \ \ 
%\partial_t f_k = (\partial_t(e^{2\pi i \alpha t}  r(t)z), -\beta)
%\]
%span all of $\R^3$.

Applying  Proposition~\ref{prop:compute}, it now follows that the coproduct
%$\vee\tilde\rho_{\ell,m}$ by direct intersection of $\rho\x \id_{(0,1)}$ with the diagonal in $\Le_{p,q}\x \Le_{p,q}$ in the following sense: for $(\ell,m)\neq (0,0)$, the coproduct in homology is the degree $1=3+1-3$ homology class
 $$\vee \tilde\rho_{\ell,m}=[\operatorname{cut}\circ (\tilde \rho_{\ell,m}\x I)|_{\overline{\Sigma_{\De}}}]$$  %\in H_1(L\Le_{p,q}\x L\Le_{p,q},\Le_{p,q}\x L\Le_{p,q} \cup L\Le_{p,q}\x \Le_{p,q})$$    
 where $\overline{\Sigma_{\Delta}}$ is the closure inside $\Le_{p,q}\x I$ of $\Sigma_{\De}$, with $\Si_\De$  oriented so that the isomorphism
 $$T_{(\un r,\un \theta,t)}(\Le_{p,q}\x I)\cong N\De\Le_{p,q}\oplus T_{(\un r,\un \theta,t)}\Si_{\De},$$
 coming from transversality, is orientation preserving.\footnote{In our conventions,  $N\De M$ is oriented so that $\tau_M\cap [M\x M]=[M]$, for $\tau_M$ the corresponding Thom class.}
Our computation above shows that $\overline{\Sigma_{\Delta}}=\Sigma_\De$ is the disjoint union of circles $\la\x I_1\cup \la'\x I_2\subset \Le_{p,q}\x (0,1)$. Given that the sign depends on choices and conventions, we only give here the important part of the sign computation for us, namely that it is independent of $t\in I_1\cup I_2$, and independent of $\ell,m$. 

Orient $T_{(\un r,\un \theta,t)}(\Le_{p,q}\x I)$ around $r_1=0$ as
 $\R^4\lgl r_1,\theta_1,\theta_2,t \rgl$. Then we have $T_{(\un r,\un \theta,t)}(\Le_{p,q}\x I)\cong -\R^3\lgl r_1,\theta_1,t \rgl\oplus T\Si_\De\lgl \theta_2\rgl$ at the intersections with $r_1=0$.
 Around $r_2=0$, we then have  $T_{(\un r,\un \theta,t)}(\Le_{p,q}\x I)\cong\R^4\lgl r_2,\theta_2,\theta_1,t \rgl $ as $r_2=\sqrt{1-r_1^2}$ is orientation preserving, and hence likewise 
 $T_{(\un r,\un \theta,t)}(\Le_{p,q}\x I)\cong -\R^3\lgl r_2,\theta_2,t \rgl\oplus T\Si_\De\lgl \theta_1\rgl$.
 And in local coordinate $(z,\theta,t)$, the map considered has the form $(z,\theta,t)\mapsto ((z,\theta),(c(t)z,\theta+\beta))$, independently of the point of $\Si_\De$. 

 Finally, we have that
 $$\operatorname{cut}\circ (\tilde \rho_{\ell,m}\x I)|_{\la\x I_1\cup \la'\x I_2}= \big(\sum_{t=1}^{l-1} \beta_{t,\ell-t} +  \sum_{t=1}^{q\ell+pm-1}\beta'_{t,q\ell+pm-t)}\big)$$
 as a family of pairs of loops. The result thus follows from Lemmas~\ref{lem:beta'} and \ref{lem:beta}.
 %The sign is determined by the orientation of the intersection, following the convention given in \cite[Prop 3.7]{HinWah0}. 
\end{proof}

\subsection{Homotopy invariance}\label{sec:invariance0}

A diffeomorphism $f: M\xrightarrow{\cong} N$ induces an isomorphism $Lf_*: H_*(LM)\xrightarrow{\cong} H_*(LN)$, and likewise for relative homology, that preserves both the loop product and coproduct, as all their defining ingredients are identified by diffeomorphisms.
It is natural to ask whether only assuming that $f$ is a homotopy equivalence could be enough for the induced isomorphism $Lf_*$ to preserve the loop product and coproduct.
Note that if $f$ satisfies the even weaker assumption of being a degree 1 map, then $f_*:H_*(M)\to H_*(N)$ already preserves the intersection product, see eg.~\cite[VI, Proposition 14.2]{Bredon}.

The following two results show that the answer to the above question is yes for the product, and no for the coproduct. 

\begin{theorem}\cite{CKS}(see also \cite{Cra,GruSal,Felix-Thomas})\label{thm:producthtpy}
Let $f:M\to N$ be a degree 1 homotopy equivalence between two closed oriented manifolds. Then $Lf_*:H_*(LM)\to H_*(LN)$ is an isomorphism of algebras with respect to the Chas-Sullivan product. 
  \end{theorem}

The main ingredient of the proof of this theorem is sketched in Section~\ref{sec:invariance} (see Theorem~\ref{thm:invEE}), where we will revisit the question of invariance of the loop product and coproduct after going through a deeper analysis of their defining ingredients. 

In the meanwhile, as noted by the first author in \cite{Nae21}, the computations presented in Section~\ref{sec:computations} can already be used to show that the loop coproduct is not homotopy invariant:
%\Fnote{I think the theorem is true like that, but the proof only shows it for the homotopy equivalence that sends the preferred generator of $\pi_1$ to the square. I think there is also one that sends it to it's fifth power?}\Nnote{you are right. i was reading Mnev (Lem 6.9) wrongle. 2 and 5 are possible.}
\begin{theorem}\cite{Nae21}\label{thm:inv} Let $f: \Le_{7,1}\to \Le_{7,2}$ be a homotopy equivalence and $\rho_{1,0}\in H_3(L\Le_{7,1})$ be as in Section~\ref{sec:computations}. Then
  $$0=f(\vee(\rho_{1,0}))\neq \vee(f(\rho_{1,0}))\in H_1(L\Le_{7,2}\x L\Le_{7,2},\Le_{7,2}\x L\Le_{7,2} \cup L\Le_{7,2}\x \Le_{7,2}).$$ 
  In particular, the loop coproduct $\vee$ is not preserved by $f$. 
\end{theorem}

The lens spaces $\Le_{7,1}$ and $\Le_{7,2}$ are the simplest examples of lens spaces that are homotopy equivalent, but not simple homotopy equivalent.
They were also used in \cite{LonSal} to prove that the configuration space of two points in a manifold is not a homotopy invariant of the manifold. In Section~\ref{sec:42}, we will see that the same  configuration of two points plays an important role in the definition of the loop coproduct.

\begin{proof}
  The class $\rho_{1,0}\in H_3(L\Le_{7,1})$ has trivial coproduct by Proposition~\ref{prop:vrho} as $\ell=q=1$ and $m=0$. (This also follows, using \cite[Theorem 3.10]{HinWah0}, from the fact that $\rho_{1,0}$ is a family of simple loops whenever $q=1$.).

  We need to compute the coproduct of the image $f(\rho_{1,0})$. The free loop space $L\Le_{7,q}$ has 7 components, and each component $L_{\ell}\Le_{7,q}$ has $H_3(L_{\ell}\Le_{7,q})\cong \Z\oplus \Z/7$ (see \cite[Sec 2.1]{Nae21}).  From Lemma~\ref{lem:beta}  and Proposition~\ref{prop:vrho}, one can deduce that e.g. the classes $\rho_{\ell,0}$ and $\rho_{\ell,1}$ generate $H_3(L_{\ell}\Le_{p,q})$. 
 Now \cite[Lemma 6.9]{Mne14} tells us that, because $f$ is a homotopy equivalence,  $\rho_{1,0}$ has image in $L_{\ell}\Le_{7,2}$ for $\ell=2$ or $5$. 
 %for any $m\neq m' \mod 7$ coprime to $\ell$. \Nnote{check that coprime to $\ell$ is needed?}
  Hence  $f(\rho_{0,1})=a\rho_{\ell,0}+(1-a)\rho_{\ell,1}$ for some $a \in \{0, \dots, p-1\}$, with $\ell=2$ or $5$. Now 
  Proposition~\ref{prop:vrho} for $\ell=2$ shows that $\vee\rho_{2,0}=5\beta_{1,1}+4(\beta_{4,5}+\beta_{5,4})$
  while $\vee\rho_{2,1}=2\beta_{1,1}+\beta_{4,5}+\beta_{5,4} + 4(\beta_{3,6}+\beta_{6,3})$. And one checks readily that there is no such $a$ such that $\vee(a\rho_{\ell,0}+(1-a)\rho_{\ell,1})=0$.
  A similar computation rules out the possibility in the case $\ell=5$. 
  % But $f(\rho_{0,1})=\rho_{2,3}\in H_3(L\Le_{7,2})$ (see \cite[Lem 9]{Nae21} for a specific choice of $f$...) and $\vee\rho_{2,3}=... \neq 0$ by Lemma~\ref{lem:beta}. 
\end{proof}

Combining the invariance of the corresponding (co)product in algebra (see Theorem~\ref{thm2}), with the fact that the algebraic model indeed models the loop coproduct  (see Theorem~\ref{thm:alggeo}), it follows that, when working over real coefficients and with simply connected manifolds, the coproduct is homotopy invariant, as stated in Theorem~\ref{thmE}. 
On the other hand, the above  computation can be extended to show the following:  %would allow to determine whether the loop coproduct detects whether lens spaces are homeomorphic. 

\begin{theorem}\label{thm:lens}
  A degree 1 homotopy equivalence $f:\Le_{p,q_1}\to \Le_{p,q_2}$ between two 3-dimensional spaces such that $Lf_*:H_*(L\Le_{p,q_1},\Le_{p,q_1})\to H_*(L\Le_{p,q_2},\Le_{p,q_2})$ preserves the loop coproduct of degree 3 classes is homotopic to a homeomorphism. 
  %necessarily a homeomorphism. \Mnote{change to ``$f$ is homotopic to a homeomorphism'' or to ``lens spaces are homeomorphic", cite Gadgil. even if it is not directly related, it has the same flavor.}
  \end{theorem}

  \begin{proof}
    Suppose $f$ is such a homotopy equivalence. Let $\rho_{1,0}\in H_3(L\Le_{p_1,q_1})$ be as above. We will just as above compare $f(\vee(\rho_{1,0}))$ with $\vee(f(\rho_{1,0}))$.

    The class $f(\rho_{1,0})$ lies in $H_3(L_{\ell}\Le_{p,q_1})$ for some $\ell$ satisfying $q_1=\ell^2q_2\mod p$, because $f$ is a degree 1 homotopy equivalence, with $f$ inducing multiplication by $\ell$ on $\pi_1$, see e.g.~\cite[Theorem 6.11]{Mne14}, where  $0<q_1,q_2,\ell<p$. 
We want to show that $f$ is homotopic to a homeomorphism. By \cite[Lemma 6.8]{Mne14}, it is enough to check that the 
  two lens spaces are homeomorphic, which happens precisely 
 %   \Mnote{change to ``The two lens spaces are homeomorphic", then refer to Lemma 6.8 of Mnev (or M. Cohen's book p.92 as referenced there) to conclude that $f$ is homotopic to a homeomorphism}
    if either $q_1q_2=\pm 1\mod p$ or $q_1=\pm q_2\mod p$, see \cite[Theorem 1.3]{Mne14}.  We may assume without loss of generality that $q_2\neq 1$. 

    To avoid confusion, denote by $\bar\rho_{\ell,0},\bar\rho_{\ell,1}\in H_3(L\Le_{p,q_2})$ the classes in the second lens space, and likewise for the $\beta$-classes. As argued above for $\Le_{7,2}$, we have that $\bar\rho_{\ell,0},\bar\rho_{\ell,1}$ generate $H_3(L_{\ell}\Le_{p,q_2})$, so we know that $f(\rho_{1,0})=(1-a)\bar\rho_{\ell,0}+a\bar\rho_{\ell,1}$ for some $a\in \{0,\dots,p-1\}$.

  The answer of the computations  $f(\vee(\rho_{1,0}))$ and $\vee(f(\rho_{1,0}))$ will be in terms of the classes $\bar\beta_{k,\ell-k}\in H_1(L\Le_{p,q_2}\x L\Le_{p,q_2})$ (or the corresponding relative homology group). As these classes only depends on the parameter $k$, we will denote them by $[k]$ below, for better readability. Note also that $f(\beta_{k,1-k})=\ell \bar\beta_{\ell k,\ell-\ell k}$ as $f$ is multiplication by $\ell$ on $\pi_1$. 

    From our computation above, we have that
    $$\vee \rho_{1,0}=\ \ q_1'\sum_{\substack{0<t<q_1}}\beta_{tq_1',1-tq_1'}$$
   so in the above notation,
    $$f(\vee \rho_{1,0})=\ \ \ell q_1'\sum_{\substack{0<t<q_1}}[t\ell q_1']=\ \  \ell' q_2'\sum_{\substack{0<t<q_1}}[t\ell' q_2']$$ 
using that $\ell q_1'=\ell'q_2$ for the second equality.  On the other hand, 
\begin{align*} \vee((1-a)\bar\rho_{\ell,0}+a\bar\rho_{\ell,1}) &=\sum_{\substack{0<t<\ell}}[t] \ \ + \ \  q_2'\!\!\!\sum_{\substack{0<t<q_2\ell \\ t,(\ell-tq_2')\neq 0\!\!\!\mod p}}[tq_2'] 
  + \ \ a q_2'\!\!\!\!\!\! \sum_{\substack{q_2\ell\le t<q_2\ell+p \\ t,(\ell-tq_2')\neq 0\!\!\!\mod p}}[tq_2'] \\
  &= \sum_{\substack{0<t<\ell}}[t] \ \ + \ \ q_2' \sum_{\substack{0<t<d}}[tq_2'] 
    + \ \ (a+c) q_2' \sum_{0<t<p}[t]%{\substack{q_2\ell\le t<q_2\ell+p \\ t,(\ell-tq_2')\neq 0\!\!\!\mod p}} [tq_2'] 
\end{align*}
where $q_2\ell=cp+d$, for some $0<d<p$. 

The equality $\vee(f(\rho_{1,0}))-f(\vee(\rho_{1,0}))=0$ holds precisely if all possible terms $[s]$ appear with coefficient a multiple of $p$. A necessary condition for that is that the terms $[s]$ all appear in 
$$ \sum_{\substack{0<t<\ell}}[t] \ \ + \ \  q_2' \sum_{\substack{0<t<d}}[tq_2'] \ \ -\ \  \ell' q_2'\sum_{\substack{0<t<q_1}}[t\ell' q_2']$$
with the same total coefficient.
Consider the sets 
\begin{align*}
    A &= \{ t \ | \ 0 < t < \ell \} \\
    B &= \{ tq_2' \ | \ 0 < t < d  \} \\
    C &= \{ t\ell'q_2' \ | \ 0 <t< q_1 \}.
\end{align*}

\noindent
{\em Case 1: $A=\emptyset$, or equivalently $\ell=1$.} Then $q_1=\ell^2 q_2=q_2\mod p$ and $f$ is a homeomorphism. 
%The only possibilities are then  $B=C$, in which case $q_2=d=q_1$, or $B=C^c$, requiring $q_1+q_2=p$. In both cases, $f$ is a homeomorphism. 
% and also $q_2=-q_2$

\noindent
{\em Case 2: $B=\emptyset$.} Then $\ell q_2=1\mod p$ (as  $d=1$), so that $q_1=\ell^2 q_2=\ell$. But then 
$q_1q_2=1\mod p$, which also gives that $f$ is a homeomorphism.

\noindent
{\em Case 3: $C=\emptyset$ with $A,B\neq \emptyset$.} So $q_1=1$, and either $A=B$ or $A=B^c$. 
If $A=B$, then $\ell=d=\ell q_2\mod p$, giving $q_2=1$. If $A=B^c$, we would need $q_2'=1$ for the coefficients to agree. In both cases, this contradicts our assumption that $q_2\neq 1$. 

\noindent
{\em Case 4: $A,B,C\neq \emptyset$}. If the sets are disjoint, we need the three coefficients to be equal, giving in particular $q_2'=1$ which contradicting again  $q_2\neq 1$.  
The case  $A=B$ is ruled out above. If $A=C$, then $\ell=q_1=\ell^2q_2\mod p$, giving $\ell q_2=1\mod p$, i.e. $d=1$ contradicting that $B$ is non-empty. And if $B=C$, $q_2\ell=q_1=\ell^2q_2\mod p$, giving $\ell=1$, contradicting that $A\neq \emptyset$.
We are now left with the case when all three sets intersect, but none are equal. In that case, we need all sums of coefficients to agree: $1+q_2'=1-\ell'q_2'=q_2'-\ell'q_2'$ modulo $p$, implying in particular $q_2'=1\mod p$, again a contradiction. 
    \end{proof}

%\Fnote{The following is extremely sloppy and should not appear in a final version}
%\begin{proof}
%Let $(l,m)$ be such that $f_*(\rho_{(l,m)}) = \rho_{(1,0)}$. And let $l'$ be the multiplicative inverse of $l$. If $f$ did preserve the coproduct then the two expressions
%\[
%l' \left(\sum_{0 < t < l} \beta_{tl',1-l't} + q_1' \sum_{0 < t < q_1 l + pm} \beta_{tq_1'l', 1-tq_1'l'} \right)
%\]
%and
%\[
%q_2' \sum_{0 < t < q_2 } \beta_{tq_2', 1 - tq_2'}
%\]
%would coincide. Let us write linear combinations of $\beta_{t, 1-t}$ as functions $\Z_p \setminus \{ 0 , 1 \} \to \Z_p$ and note that the dependence on $m$ is by adding a constant function. We then obtain the equation
%\[
%l' \chi_A + l'q_1' \chi_B = q_2' \chi_C - q_1'm
%\]
%where
%  \begin{align*}
%    A &= \{ tl' \ | \ 0 < t < l \} \\
%    B &= \{ tq_1'l' \ | \ 0 < t < q_1 l \} \\
%    C &= \{ tq_2' \ | \ 0 < q_2 \}.
%\end{align*}
%There are lots of cases: If $A = \varnothing$ then $l = 1$ and either $B = C$ or $B = C^c$ in the first case we get $q_1 = q_2$ by looking at the sizes of $B$ in the other case we get $q_1' = - q_2'$ by looking at the coefficients. \Fnote{not sure if this is orientation reversing}
%We can now assume that $A \neq \varnothing$, but then either $A \subset B = \Z_p \setminus \{ 0 ,1 \}$ or $\varnothing = B \subset A$. And thus either $A = C$ or $A = C^c$ and we conclude similar to the other case.
%\end{proof}

\subsection{The good and the bad coproduct}\label{sec:badcop}

The coproduct we have described looks for self-intersections of the form $\ga(t)=\ga(0)$ in families of loops $\ga$ where $t\in I$ is any time along the interval. One could instead define a coproduct $\vee_{\frac{1}{2}}$ that only looks for self-intersections at time $t=\frac{1}{2}$. This leads to a rather trivial coproduct, as noted by Tamanoi in \cite{Tam}. Indeed, the coproduct $\vee_{\frac{1}{2}}$ is homotopic to the coproduct $\vee_0$ that looks for (``left-trivial'') self-intersections at $t=0$, i.e.~of the form $\ga(0)=\ga(0)$, or likewise to the coproduct $\vee_1$ looking for (``right-trivial'') self-intersection at $t=1$ only. Combining these two facts can be used to show that the coproduct $\vee_{\frac{1}{2}}$ is only non-trivial on the fundamental class of $[M]$, considered as a family of constant loops, and only when $\chi(M)\neq 0$, with $\vee_{\frac{1}{2}}[M]=(-1)^n\chi(M)[\{*\}\x\{*\}]\in H_0(LM\x LM)$ (see e.g., \cite[Lemma 4.5]{HinWah0}). In fact, the ``good'' coproduct
$\vee$ that we have worked with here can be thought of as a secondary operation, coming from these two reasons that $\vee_{\frac{1}{2}}$ is trivial, homotoping it to its $t=0$ or $t=1$ versions.

\medskip

One way to formulate this relationship between the two coproduct is as follows:
%\Fnote{To me this seems like the same relationship and not "another" one. How about: "This can be made concrete as follows:"}
 the coproduct $\vee$ can be defined as a {\em relative version} of the coproduct  $\vee_{\frac{1}{2}}$, as we explain now. This form of definition first appeared in \cite[Section 9]{GorHin}, in the definition of the dual cohomology product.

The coproduct $\vee_{\frac{1}{2}}$ is defined just like $\vee$ but replacing  the evaluation $e_I$ by the map $\ev_{0,\frac{1}{2}}=(\ev_0,\ev_{\frac{1}{2}}):LM\to M\x M$.  Denoting  $\Figeight=\ev_{0,\frac{1}{2}}^{-1}(\De M) \subset LM$ is the space of ``figure eights'', i.e.~loops $\ga$ with a self-intersection $\ga(0)=\ga(\frac{1}{2})$, and  $U_\eps(\Figeight)=\ev_{0,\frac{1}{2}}^{-1}(U_M)$ its $\eps$--neighborhood, we have
  $$\vee_{\frac{1}{2}}: H_*(LM) 
  \xrightarrow{(\ev_{0,\frac{1}{2}})^*\tau_M\cap }  H_*(U_\eps(\Figeight)  \xrightarrow{R_{\frac{1}{2}}}  H_*(\Figeight)\xrightarrow{cut}  H_*(LM\x LM), $$
  for $R_{\frac{1}{2}}$ a retraction map defined just like the retraction map $R$ used for $\vee$.  

Let $J:LM\x I\to LM$ be the reparametrizing map defined by $J(\ga,s)=\ga\circ \theta_{\frac{1}{2}\to s}$ where $ \theta_{\frac{1}{2}\to s}:[0,1]\to [0,1]$ is the piecewise linear map that fixes $0$ and $1$ and takes $\frac{1}{2}$ to $s$. Note that $J$ restricts on the boundary to a map $J: LM\x \del I \to \calR$ for
\begin{equation}\label{equ:R}
  \calR:=\{\ga\in LM\ |\ \ga|_{[0,\frac{1}{2}]} \ \textrm{or}\ \ga|_{[\frac{1}{2},1]} \ \textrm{is constant}\}
\end{equation}
the subspace of $LM$ of \textit{half-constant loops.}

\begin{proposition}\label{prop:goodandbad}
The loop coproduct $\vee$ can equivalently be defined as the composition of the following sequence of maps:
\begin{multline}\label{equ:cop2}
  H_*(LM,M) \xrightarrow{\x I} H_{*+1}(LM\x I,LM\x \del I\cup M\x \del I) \xrightarrow{\;J\;} H_*(LM,\calR) \\
  \xrightarrow{(\ev_{0,\frac{1}{2}})^*\tau_M\cap }  H_*(U_\eps(\Figeight),\calR)  \xrightarrow{R_{\frac{1}{2}}}  H_*(\Figeight,\calR)\xrightarrow{cut}  H_*(LM\x LM,M\x LM\cup LM\x M). 
\end{multline}
\end{proposition}

See \cite[Theorem 2.13]{HinWah0}  for a proof that this new definition is equivalent to the one of Section~\ref{sec:defproco}. Note that the last three maps in the statement indeed compose to a relative version of the coproduct $\vee_{\frac{1}{2}}$.

\section{String topology via Hochschild complexes}\label{sec:HH}

In this section we define a product on the Tate-Hochschild complex of any connected dg Frobenius algebra $A$. The Tate-Hochschild complex is an amalgam of the Hochschild chains and cochains, that model, by results of Jones and Chen, the cohomology and homology of the free loop space of simply connected manifolds, respectively. We will see below and in Section~\ref{sec:conf} that, in this case,  the product on the Tate-Hochschild complex relates to both the Chas-Sullivan product, when restricted to the Hochschild cochains, and the Goresky-Hingston coproduct, when restricted to the Hochschild chains. 

\subsection{Differential graded algebras} 
%We start by recalling some classical notions, setting up at the same time some notation for this section.
Let $\mathbb{K}$ be a commutative ring with unit. Recall that a {\it dg $\mathbb{K}$-module}, or \textit{chain complex}, is a graded $\mathbb K$-module $V = \bigoplus_{j\in \mathbb Z} V^j$ equipped with a %$\mathbb K$-linear map
differential  $d_V \colon V \to V$; in this section, all differentials will have degree $+1$. 
% of degree $+1$ such that $d_V\circ d_V = 0.$ %For any element $a \in V^j$ we write $|a| = j$.
The {\it dual} of  $(V, d_V)$ is  the dg $\mathbb{K}$-module  $(V^{\vee}, d_{V^{\vee}})$ with $(V^{\vee})^{-j} = \Hom_{\mathbb K}(V^{j}, \mathbb K)$ and the differential defined by  $d_{V^{\vee}}(\alpha)(x) = -(-1)^{|\alpha|} \alpha (d_V(x))$ on homogeneous elements $\alpha \in V^{\vee}$, where $|\alpha|$ denotes the degree of $\alpha$.  %and $x \in V$. 

A {\it dg $\mathbb{K}$-algebra} $A=(A, d,  \mu)$, or {\it dg-algebra} for short,  is a dg $\mathbb{K}$-module $(A,d)$ equipped with an associative product $\mu\colon A \otimes A \to A$ of degree zero satisfying the Leibniz rule $$\mu \circ (d \otimes \text{id} + \text{id} \otimes d)= d \circ \mu.$$
% We will refer to dg $\mathbb{K}$-algebras by \textit{dg algebras} when the underlying ring is understood.
We write $\mu(a \otimes b)=ab$. The multiplication  is (graded) commutative if $ab=(-1)^{|a||b|}ba$, and {\it unital} 
if there is a map $u\colon \mathbb{K} \to A$ such that the image of $1\in \mathbb K$ is a unit for the multiplication of $A$. %, we denote also by $1$ the unit of $A$. 
% $$\mu \circ (u \otimes \text{id}) = \text{id} = \mu \circ (\text{id} \otimes u);$$ we will write $1\in A$ for the image of $1\in \mathbb K$.
 
  % A \textit{morphism of dg algebras} is a $\mathbb{K}$-linear map $f: A \to A'$ between dg algebras which preserves the structure, i.e. it is simultaneously a map of graded $\mathbb{K}$-algebras and of dg $\mathbb{K}$-modules.
The \textit{cohomology}  $H^*(A)$ of a dg algebra $A=(A,d,\mu)$ becomes a graded $\mathbb{K}$-algebra with product $H^*(A) \otimes H^*(A) \to H_*(A)$ induced by $\mu: A \otimes A \to A$. 
  % whose underlying graded $\mathbb{K}$-module is given by $H^n(A)= \frac{\text{ker} (d: A^n \to A^{n+1}) }{\text{im} (d: A^{n-1} \to A^n)}$ and with product $H^*(A) \otimes H^*(A) \to H_*(A)$ induced by $\mu: A \otimes A \to A$.
  A morphism of dg algebras $f: A \to A'$ is a \textit{quasi-isomorphism} if it induces an isomorphism of graded algebras $H^*(f): H^*(A) \xrightarrow{\cong} H^*(A')$.

\begin{example}  
The following examples  are particularly relevant to our discussion:
\begin{enumerate}
\item The singular cochains on a topological space $X$ equipped with the simplicial differential and cup product define a dg algebra $(C^*(X; \mathbb{K}), d, \smile)$. The cup product is associative and homotopy commutative. 
\item When $\mathbb{K}= \mathbb{Q}$ the dg algebra $(C^*(X; \mathbb{Q}), d, \smile)$ is quasi-isomorphic to a \textit{commutative} dg algebra $(\Apl(X), d, \wedge)$ of $\mathbb{Q}$-polynomial differential forms, as shown by Sullivan. 
\end{enumerate}
\end{example}

One of the main theorems discussed in this note, Theorem \ref{thm1}, involves the weaker notion of an $A_{\infty}$-algebra. Recall that an $A_{\infty}$-\textit{algebra} is a graded $\mathbb{K}$-module  $A$ equipped with linear maps $\{m_n: A^{\otimes n} \to A\}_{n \in \mathbb{Z}_{>0}}$, where each $m_n$ is of degree $2-n$,  satisfying the following relations:
\begin{itemize}
\item $m_1 \circ m_1=0$, in other words, $(A, m_1)$ is a dg $\mathbb{K}$-module,
\item $m_1 \circ m_2= m_2 \circ (m_1 \otimes \text{id}_A + \text{id}_A \otimes m_1)$, in other words, the product $m_2$ satisfies Leibniz rule with respect to $m_1,$
\item more generally, for each positive integer $n$ we have
$$ \sum (-1)^{p+qr} m_{p+1+r}\circ ( \text{id}_A^{\otimes p} \otimes m_q \otimes \text{id}_A^{\otimes r})=0,$$
where the sum runs over all triples of positive integers $(p,q,r)$ such that $n=p+q+r.$
\end{itemize}

In particular, the last equation implies that $m_3: A^{\otimes 3} \to A$ is a chain homotopy for the associativity of $m_2$. Hence, for any $A_{\infty}$-algebra $A$, the cohomology $H^*(A,m_1)$ has an induced graded associative algebra structure. 
\subsection{Differential graded Frobenius algebras} 
The notion of a symmetric dg Frobenius algebra consists of a dg algebra equipped with a non-degenerate symmetric bilinear pairing compatible with the product structure. Our interest in symmetric dg Frobenius algebras is motivated by Poincar\'e duality. 
\begin{definition}\label{def:symfrob} %Let $k$ be a positive integer.
  A {\it dg Frobenius $\mathbb{K}$-algebra of dimension $n$} is a non-negatively graded unital dg $\mathbb{K}$-algebra $(A, d, \mu)$ equipped with a pairing $\langle -, -\rangle\colon A \otimes A \to \mathbb{K}$ such that
\begin{enumerate}
\item $\langle -, -\rangle$ is of degree $- n$, i.e.~non-zero only on $A^i \otimes A^{n-i}$ for $i=0,\dotsb,n$
\item $\langle -, -\rangle$ is non-degenerate, namely, the induced map 
$$\rho\colon A \to A^{\vee}, \quad a \mapsto (b \mapsto \langle a, b\rangle) 
$$
is an isomorphism of degree $-n$
\item $\langle ab ,c\rangle = \langle a,bc\rangle$ \quad for any $a,b,c \in A$
\item $\langle d(a),b\rangle  = -(-1)^{|a|}\langle a,d(b)\rangle $ \quad  for any $a,b \in A$.
\end{enumerate}
\end{definition}

Conditions (3) and (4) imply that $\rho\colon A \to A^{\vee}$ is a map of dg $A$-$A$-bimodules of degree $-n$, where the $A$-$A$-bimodule structure on $A^{\vee}$ is given by $$ 
(a\otimes b)\cdot \beta (c) = (-1)^{|\beta| (|a| + |b|) + |a| (|b| + |c|)} \beta(bca), \quad \text{for any $\beta \in A^{\vee}$ and $a, b, c \in A$.}
$$ 
A dg Frobenius algebra $A$ is said to be \textit{symmetric} if $\langle a,b\rangle = (-1)^{|a| |b|}\langle b,a\rangle$ for any $a,b \in A.$ 

Note that the isomorphism $\rho: A \to A^{\vee}$ gives rise to a degree $n$ product on $A^\vee$:
$$A^{\vee} \otimes A^{\vee} \xrightarrow{\rho^{-1} \otimes \rho^{-1}} A \otimes A \xrightarrow{\mu} A \xrightarrow{\rho}  A^{\vee}.$$
When $A$ is a finitely generated free $\mathbb{K}$-module, e.g.~when $\mathbb{K}$ is a field,
%will write $\Delta\colon A \to A \otimes A$ the coproduct of degree $k$ defined by the composition
the linear dual of that product becomes a coproduct on $A$: 
\begin{equation}\label{equ:copfob}
  \Delta\colon A  \xrightarrow{ \rho} A^\vee \xrightarrow{\mu^{\vee}} (A\otimes A)^{\vee} \cong A^{\vee} \otimes A^{\vee} \xrightarrow{\rho^{-1} \otimes \rho^{-1}} A \otimes A
  \end{equation}
%denote the dual (degree $k$) coproduct. 
This coproduct is a map of $A$-$A$-bimodules.

%\begin{remark}[2-dimensional field theories] \Nnote{To be fixed: This need the coproduct, so some assumption on the algebra, on $\mathbb K$ or a different def of Frobenius I guess, as in Lauda-Pfeifer} While commutative Frobenius algebras classify $2$-dimensional (closed) topological field theories, symmetric Frobenius algebras classify open topological field theories, and non-commutative Frobenius algebras classify planar open topological field theories,  see \cite{Kock} and \cite[Cor 4.5-7]{Lau08}. We do not require commutativity for our algebras. 
%\end{remark}

\begin{example}[Poincar{\'e} duality and relationship to the intersection product]\label{ex:PDint}
  Let $M$ be a closed manifold of dimension $n$. The graded cohomology ring $(H^*(M;\mathbb{K}), \smile)$ with coefficients in the commutative ring $\mathbb{K}$ is an example of a symmetric dg Frobenius algebra of dimension $n$ with trivial differential $d=0$ and pairing given by Poincar\'e duality.
  When $\mathbb{K}$ is a field, the coproduct $\Delta: H^*(M;\mathbb{K}) \to H^*(M;\mathbb{K}) \otimes  H^*(M;\mathbb{K})$ is the dual of the intersection product on homology. 
%The element $\Delta(1)=\sum_{i } e_i \otimes f_i \in H^*(M;\mathbb{K}) \otimes H^*(M;\mathbb{K})$ then corresponds to the Thom class of the diagonal embedding $M \hookrightarrow M \times M$.
  Indeed, the cup product is induced by the diagonal $\De_M^*:H^*(M\x M) \to H^*(M)$, so, following \eqref{equ:copfob}, 
  the coproduct $\Delta: H^*(M;\mathbb{K}) \to H^*(M;\mathbb{K}) \otimes H^*(M;\mathbb{K})$ is the composition 
$$\xymatrix{
  H^k(M;\mathbb{K}) \ar[d]_{[M]\cap }^\cong\ar@{-->}[rr]^-\Delta & & \bigoplus_{i+j=n-k} H^{n-i}(M;\mathbb{K}) \otimes H^{n-j}(M;\mathbb{K}) \ar[d]^{[M]\cap\ot [M]\cap}_{\cong} \\
  H_{n-k}(M;\mathbb{K}) \ar[r]^-{(\De_M)_*} &  H_{n-k}(M\x M;\mathbb{K})  &\ar[l]_-\cong \bigoplus_{i+j=n-k}H_i(M;\mathbb{K}) \otimes H_j(M;\mathbb{K}). %\ar[u]_{P^{-1} \otimes P^{-1}}  
}$$
%Here $AW$ denotes the Alexander-Withney map. 
% and $P: H^k(M;\mathbb{K}) \to H_{n-k}(M;\mathbb{K})$ is the Poincar\'e duality $H^*(M;\mathbb{K})$-bimodule isomorphism defined by $P(\alpha) = [M] \cap \alpha$, for $[M] \in H_n(M;\mathbb{K})$  the fundamental class of $M$ and $\cap$ is the classical cap product.
And the intersection product can been defined as the Poincar\'e dual of the cup product, i.e.~precisely the linear dual of the cohomology coproduct. (See e.g., \cite[App B]{HinWah0} for the relationship between that definition of the intersection product and the one given in Section~\ref{sec:int0}.) 

Applying the above composition of maps  to $1 \in H^0(M;\mathbb{K}) \cong \mathbb{K}$ we get a class $\Delta(1) \in \bigoplus_{i+j=n} H_i(M;\mathbb{K}) \otimes H_j(M;\mathbb{K}) \cong H_n(M \times M; \mathbb{K})$ that is the uniquely such that $[M\x M]\cap \De(1)=(\De_M)_*[M]$. 
%determined Poincar\'e dual in $M\x M$ of the homology class defined by the diagonal embedding $M \hookrightarrow M \times M$.
Hence $\Delta(1)$ maps to the 
Thom class $\tau_M\in H^n(M\x M,M\x M\minus M)$ of Section~\ref{sec:int0}  in relative cohomology,  
as the Thom class is determined by this same relation. 
\end{example}

A dg algebra $A$ is \textit{simply connected} if it is non-negatively graded, $A^0=\mathbb{K}$, and $A^1=0$.
The following result of Lambrecths and Stanley shows that, when $\mathbb K$ is a field and $A$ is commutative and simply connected, a Frobenius structure on $H^*(A)$ can be ``lifted'' to $A$. 
%shows that, for a simply connected closed oriented manifold $M$ and over a field $\mathbb{K}$, there is a commutative symmetric dg Frobenius algebra $A_M$ which ``lifts" the graded Frobenius structure of $H^*(M;\mathbb{K})$ to the cochain level. 

\begin{theorem}\cite[Theorem 1.1]{LamSta08} \label{lastthm} 
Let $\mathbb{K}$ be any field and $\mathcal{A}$ be a simply connected commutative dg $\mathbb{K}$-algebra equipped with a pairing $\langle -, -\rangle_{\mathcal{A}}: \mathcal{A} \otimes \mathcal{A} \to \mathbb{K}$ which induces a graded Frobenius algebra structure of dimension $k$ on its cohomology $H^*(\mathcal{A})$. Then there exists a simply connected commutative symmetric dg Frobenius $\mathbb{K}$-algebra $A$ and a zig-zag of quasi-isomorphisms of commutative dg algebras between $A$ and $\mathcal{A}$ inducing an isomorphism  $H^*(A) \cong H^*(\mathcal{A})$ of graded Frobenius algebras. 
\end{theorem}

\begin{example}[Frobenius models of manifolds]\label{ex:frobmodel}
  Let $M$ be a simply-connected oriented closed manifold and assume  $\mathbb{K}=\mathbb{Q}$. 
  % such that $C^*(M,\mathbb{K})$ can be modeled by a strictly commutative algebra; this is for example always possible when $\mathbb{K}=\mathbb{Q}$, using
  Then the polynomial forms
$\Apl(M)\simeq C^*(M,\mathbb{Q})$ are a strictly commutative, simply connected model of the cochains. The above theorem then yields a commutative dg Frobenius algebra $A_M \simeq C^*(M,\mathbb{Q})$, that ``lifts" the graded Frobenius structure of $H^*(M;\mathbb{Q})$ to the cochain level.  
\end{example}

\subsection{Hochschild chains and cochains}

We recall here the definition of the Hochschild chain and cochain complexes and their relevance in homological algebra and topology.
We will work with the normalized version of the Hochschild complex, assuming that the algebra 
is unital. Let $\overline{A}$  denotes the cokernel of the unit map $\mathbb{K} \to A$. 

For any dg $\mathbb{K}$-module $(V, d)$ we denote by $(s^iV, s^id)$ the $i$-th shifted module given by $(s^iV)^j = V^{i+j}$  and $s^id(v)=(-1)^i d (s^iv)$ for any $v\in V$. The definition of the Hochschild complex will use the suspension $s\overline A$. For simplicity, we write $\overline{a} $ for the element $sa\in s\overline{A} $ where $a \in \overline A$.

\begin{definition} \label{definition:hochschildchain} Let %$(A=\mathbb{K}.1 \oplus \overline{A}, d, \mu)$
  $A$ be an unital dg algebra. The \textit{Hochschild chain complex of $A$} is the complex $(C_*(A,A), \partial= \partial_v + \partial_h)$
where 
$$C_*(A,A) = %\bigoplus_{n \in \mathbb{Z}}C_n(A,A) = \bigoplus_{n \in \mathbb{Z}}
\bigoplus_{m \geq 0} (s\overline{A})^{\otimes m} \otimes A$$
 and 
where $\partial_v$, the \textit{vertical} differential, is given by
\begin{equation*}
\begin{split}
\partial_v(\overline{ a_1}\otimes \cdots \otimes \overline{a_m}\otimes a_{m+1})=&-\sum_{i=1}^m(-1)^{\epsilon_{i-1}} \overline{a_1}\otimes \cdots \otimes \overline{a_{ i-1}} \otimes
\overline{d(a_i)} \otimes \overline{a_{i+1}}\otimes \cdots \otimes a_{m+1} \\
& +(-1)^{\epsilon_m} \overline{a_1}\otimes\cdots \otimes \overline{a_m} \otimes d(a_{m+1})
\end{split}
\end{equation*}
 and $\partial_h$, the \textit{horizontal} differential, is given by
\begin{equation*}
\begin{split}
\partial_h(\overline{a_{1}}\otimes \cdots \otimes \overline{a_m}\otimes a_{m+1})=&\sum_{i=1}^{m-1} (-1)^{\epsilon_i} \overline{ a_1} \otimes \cdots \otimes \overline{a_{i-1}}\otimes \overline{a_{i}a_{i+1}}\otimes \overline{a_{i+2}} \otimes\cdots \otimes 
a_{m+1}\\
&-(-1)^{\epsilon_{m-1}} \overline{a_1}\otimes \cdots \otimes \overline{a_{m-1}}\otimes a_ma_{m+1}\\
&+ (-1)^{(|a_2|+
\dotsb+ |a_{m+1}| - m +1)|a_1| } \overline{a_2}\otimes \cdots \otimes \overline{a_m}\otimes a_{m+1}a_1.
\end{split}
\end{equation*}
Here we denote $\epsilon_i= |a_1| + \dotsb + |a_i| - i$ and  $\epsilon_0= 0$.

We will denote by $C_{-m,k}(A,A) = ((s\overline{A})^{\otimes m} \otimes A)^{k}$ the elements in $(s\overline{A})^{\otimes m} \otimes A$ of total degree $k$. In particular,  $C_k(A,A) = \bigoplus_{m \in \mathbb{Z}_{\geq 0}} C_{-m,k}(A,A)$
\end{definition}

The Hochschild homology of $A$ is defined to be the homology of $(C_*(A,A), \partial= \partial_v + \partial_h)$ and it is denoted by $\HH_*(A,A)$. Hochschild homology is functorial with respect to maps of unital dg algebras. Furthermore, a quasi-isomorphism $f: A \to A'$ between unital dg algebras that are flat as $\mathbb{K}$-modules induces an isomorphism $$HH_*(f): HH_*(A,A) \to HH_*(A',A').$$

\begin{remark}[The Hochschild complex in algebra and topology]\label{rem:HHalgtop}
  The Hochschild chain complex originates in the context of homological algebra. When $A$ is a dg algebra which is projective as a $\mathbb{K}$-module, $C_*(A,A)$ is model for $A \otimes^{\mathbb{L}}_{A \otimes A^{op}} A$, the derived tensor product of $A$ with itself in the category of $A$-$A$-bimodules. %This model may be obtained by resolving $A$ as an $A$-$A$-bimodule using the bar resolution.
  Hence, $\HH_*(A,A)= \text{Tor}_*^{A \otimes A^{op}}(A,A)$. 

  In topology, when $\mathbb{K}=\mathbb{F}$ is a field, %the Hochschild complex models the free loop space in the following ways: If
 and $A\simeq C^*(X;\mathbb{F})$ is a dg algebra cochain model for the singular cochains of a simply connected space $X$, then there is a quasi-isomorphism  $C_*(A,A)\simeq C^*(LX;\mathbb{F})$ between the Hochschild chains of $A$ and  the singular cochains of the free loop space of $X$. This relationship may be deduced over the reals using Chen iterated integrals (as introduced by Chen in \cite{Chen73}, see also \cite{GeJoPe91, Mer}), or over any field using a cosimplicial model for the free loop space (as done by  Jones in \cite{Jon87}). A dual version of the result, in terms of the coHochschild complex of the singular chains coalgebra, that works for coefficients in an arbitrary ring $\mathbb{K}$ may be found in \cite{RiSa19}. 

  Goodwillie gave in \cite{Goo85} the following ``Koszul dual" version of this model of the free loop space that does not assume simple connectivity. Let $\mathbb{K}$ be any commutative ring and  assume $X$ is a path-connected and set instead $A=C_*(\Omega X; \mathbb{K})$, the singular chains on the space of (Moore) loops in $X$, equipped with the concatenation product. Then there is a quasi-isomorphism $C_*(A,A) \simeq C_*(LX;\mathbb{K})$. 
\end{remark}

%\begin{remark}
%Note that an element $
%\overline{a_1}\otimes \cdots \otimes \overline{a_m}\otimes a_{m+1} \in (\sA)^{\otimes m} \otimes A$ belongs to $C_n(A, A)$ if and only if $|a_1| + |a_2| + \dotsb + |a_{m+1}| - m = n$. The differential $\partial$ on $C_*(A,A)$ is of degree $+1$. We use lower index notation in $C_*(A,A)$ to distinguish from the Hochschild cochain complex defined below.
%\end{remark}

\begin{definition}
  \label{definition-cohomology} Let $A$ be a unital dg algebra. 
  The \textit{Hochschild cochain complex of $A$} is the complex $(C^*(A,A), \delta= \delta^v + \delta^h)$ where
   $$C^*(A,A) = \prod_{m\ge 0} \text{Hom}_{\mathbb{K}}( (s\overline{A})^{\otimes m}, A)$$
 and where $\delta^v$ is given by
\begin{equation*}
\begin{split}
\delta^v(f)(\overline{a_{1}}\otimes \cdots\otimes \overline{a_m})={} &d(f(\overline{a_1}\otimes \cdots \otimes \overline{a_ m})) +\sum_{i=1}^m (-1)^{|f|+\epsilon_{i-1}} f(\overline{a_{1}}\otimes \cdots \otimes \overline{d(a_i)} \otimes  \cdots \otimes \overline{a_m}), 
\end{split}
\end{equation*}
and  $\delta^h$  by
\begin{equation*}
\begin{split}
\delta^h(f)(\overline{a_{1}} \otimes \cdots \otimes \overline{a_{m+1}})=& -(-1)^{(|a_1|-1)|f|} a_1 f(\overline{a_{2}}\otimes \cdots \otimes\overline{a_{ m+1}})  \\
&- \sum_{i=1}^m (-1)^{|f|+\epsilon_i}f(\overline{a_{1}}\otimes \cdots \otimes \overline{a_{ i-1}}\otimes\overline{ a_ia_{i+1} }\otimes \overline{a_{i}}\otimes \cdots \otimes \overline{a_{ m+1}}) \\
&+(-1)^{|f|+\epsilon_{m}} f(\overline{a_{1}}\otimes \cdots \otimes \overline{a_ m})a_{m+1},
\end{split}
\end{equation*}
with $\epsilon_i =  |a_1| + \dotsb + |a_i| - i$ and $\epsilon_0 =0$ as before.

Denoting by $C^{m,k}(A,A)= \text{Hom}^n_{\mathbb{K}}( (s\overline{A})^{\otimes m}, A)$ the submodule of $\mathbb{K}$-linear maps  of degree $k\in \mathbb Z$, we have 
$C^k(A,A) = \prod_{m \geq 0}  C^{m, k}(A, A)$. 
\end{definition}
The Hochschild cohomology $\HH^*(A,A)$ of $A$ is defined to be the cohomology of $(C^*(A,A), \delta= \delta_v + \delta_h)$. The Hochschild cochain complex is not at such natural in maps of dg algebras, but if $f: A \to A'$ is a quasi-isomorphism of unital dg algebras that are flat as $\mathbb{K}$-modules, then there is an isomorphism $\HH^*(A,A) \cong \HH^*(A',A')$. We will see in the next section that the product structure of Hochschild cohomology is also invariant under quasi-isomorphisms.  %this isomorphism preserves the algebra structure that will be described in Definition \ref{cup}. 

\begin{remark}\label{rem:coHHalg}
  When $A$  is projective as a $\mathbb{K}$-module, the complex $C^*(A,A)$ is a model for $\mathbb{R}\text{Hom}_{A \otimes A^{op}}(A,A)$, the derived hom from $A$ to itself in the category of $A$-$A$-bimodules. %, obtained by resolving $A$ as an $A$-$A$-bimodule using the bar resolution.
  Hence, $HH^*(A,A)= \text{Ext}^*_{A \otimes A^{op}}(A,A)$. One may also model the Yoneda product $\text{Ext}^*_{A \otimes A^{op}}(A,A)$ via the chain level cup product $\cup$ on $C^*(A,A)$ of Definition \ref{cup}. The graded algebra $(HH^*(A,A), \cup)$ may be also equipped with a Lie bracket of degree $-1$ which is compatible with the cup product. The resulting algebraic structure is known as a Gerstenhaber algebra and was described in \cite{Ger63}. The Gerstenhaber algebra structure on $\HH^*(A,A)$ may be lifted to an $E_2$-algebra structure at the cochain level on $C^*(A,A)$. This statement is known as the \textit{Deligne conjecture} and was solved in \cite{McSm03}. 
\end{remark}

\begin{remark}[Duality]\label{rem:HHdual}  For any dg algebra $A$ the graded hom-tensor adjunction provides an isomorphism $$C_{-m,*}(A,A)^{\vee} \cong C^{m,*}(A, A^{\vee}).$$ If $A$ is a symmetric dg Frobenius algebra which is a finitely generated free $\mathbb{K}$-module then the isomorphism of $A$-$A$-bimodules $A \cong A^{\vee}$ induces an isomorphism of graded $\mathbb{K}$-modules $$C_{-m,*}(A,A)^{\vee} \cong C^{m,*}(A, A^{\vee}) \cong C^{m,*}(A,A).$$

In particular,  if $A$ is a symmetric dg Frobenius algebra model over a field $\mathbb{F}$ for a simply connected closed manifold $M$, e.g.~as provided by Theorem \ref{lastthm}, combining this duality with Remark~\ref{rem:HHalgtop} gives an isomorpism $\HH^*(A,A)\cong H_*(LM;\mathbb{F})$. In Section~\ref{sec:conf} we discuss how the Gerstenhaber algebra structure of $\HH^*(A,A)$ corresponds to the Chas-Sullivan product of Section~\ref{sec:inter} and a loop bracket that in addition uses the circle action, see also \cite{FTV}.
\end{remark}

\subsection{Tate-Hochschild complex}\label{sec:Tate}

In the presence of a Frobenius structure on an algebra $A$ we may combine Hochschild chains and cochains of $A$ into a single unbounded complex through a construction reminiscent of the Tate cohomology of a finite group.

\begin{definition}\cite{RivWan19}  Let $A$ a symmetric dg Frobenius $\mathbb{K}$-algebra of dimension $n >0$. %that is free as a $\mathbb{K}$-module.% 
 Write $\Delta(1) =\sum_{i } e_i\otimes f_i \in A \otimes A$. %for the Casimir element of $A$.
  The \textit{Tate-Hochschild complex } $(\calD^*(A, A), \delta)$ of $A$ is the totalization of the double complex 
$$
\calD^{*,*}(A,A) = \cdots \xrightarrow{\partial_h} s^{1-n}C_{-1,*}(A,A) \xrightarrow{\partial_h} s^{1-n}C_{0,*}(A,A) \xrightarrow{\gamma} C^{0,*}(A,A) \xrightarrow{\delta_h} C^{1,*}(A,A) \xrightarrow{\delta_h} \cdots
$$
%where $\gamma$ is the composition )%$$
%s^{1-k}C_{0,*}(A,A) = s^{1-k}A \xrightarrow{\Delta} s(A \otimes A) \xrightarrow{sT} s(A \otimes A) \xrightarrow{s\mu} sA \xrightarrow{s^{-1}} A = C^{0,*}(A,A),
%$$ 
%and $T(x \otimes y)= (-1)^{|x||y|}y \otimes x$. 
where $\gamma: C_{0,*}(A,A) \cong A \to A \cong C^{0,*}(A,A)$ is given by  $$
\gamma(a) = \sum_{i }(-1)^{|f_i||a|} e_i a f_i, \quad \text{for any $a \in A$.}$$  
The fact that $\partial_h \circ \gamma=0 = \gamma \circ \delta^h$ follows from (4) Definition~\ref{def:symfrob}. 
% so that the horizontal maps above square to zero.  
Here totalization means the direct sum totalization in the Hochschild chains direction and the direct product totalization in the Hochschild cochains direction: 
\begin{equation*}
\begin{split}
\calD^k(A, A)&= \prod_{p\geq 0} \Hom_{\mathbb{K}}((\sA)^{\otimes p}, A)^k \oplus \bigoplus_{p\in \geq 0}( (\sA)^{\otimes p}\otimes A)^{k-n+1}\\
&=C^k(A, A)\oplus C_{k-n+1}(A, A). 
\end{split}
\end{equation*}
% We denote the differential of the totalization by $\delta: \mathcal{D}^*(A,A) \to \mathcal{D}^{*+1} (A,A)$.
\end{definition}

One can equivalently define the Tate-Hochschild complex $\calD^*(A, A)$ as the mapping cone of the chain map
\begin{align}\label{mappingcone}
\widetilde{\gamma} \colon s^{1-n} C_*(A, A) \to C^*(A, A)
\end{align}
defined by $\widetilde{\gamma} (\alpha) = 0$ if $\alpha \in C_{-m, *}(A, A)$ for $m\neq 0$ and $\widetilde{\gamma} (\alpha) =  \sum_{i }(-1)^{|f_i||a|} e_i a f_i$ if $\alpha \in A=C_{0, *}(A, A)$.
%Here  $\Delta(1):= \sum_{i } e_i \otimes f_i \in A \otimes A$. Note $|e_i| + |f_i|=k$ for all $i$.

\begin{definition} \label{Tatepairing} Let $A$ be a dg Frobenius algebra with pairing $\langle -, -\rangle_A: A \otimes A \to \mathbb{K}$. Define a paring $$\langle -, -\rangle_{\mathcal{D}}: \mathcal{D}^*(A,A) \otimes \mathcal{D}^*(A,A) \to \mathbb{K}$$  by
$$\langle f, \alpha\rangle_{\mathcal{D}} := \langle f(\overline{a_1} \otimes \cdots  \otimes \overline{a_m}), a_{m+1} \rangle_A$$ for any $\alpha= \overline{a_1}\otimes \cdots \otimes \overline{a_m} \otimes a_{m+1} \in C_{-m, *}(A, A)$ and $f\in C^{m, *}(A, A)$, and $0$ otherwise.
\end{definition}

The above pairing is compatible with the Tate-Hochschild differential, i.e. it satisfies $$\langle \delta x,y\rangle_{\mathcal{D}}= (-1)^{|x|} \langle x, \delta y\rangle_{\mathcal{D}}.$$ Consequently, we obtain an induced pairing $H^*(\mathcal{D}^*(A,A)) \otimes H^*(\mathcal{D}^*(A,A)) \to \mathbb{K}$.

\begin{remark}[The Tate complex in algebra and topology]\label{rem:Tatealgtop} Let $\mathbb{K}$ be a field and $A$ a symmetric dg Frobenius $\mathbb{K}$-algebra $A$. Then $H^*(\mathcal{D}^*(A,A))$ is isomorphic to the graded $\mathbb{K}$-vector space of morphisms from $A$ to itself in the {\it singularity category} $$\DD_{\sg}(A\otimes A^{\op})= \DD^b(A \otimes A^{\op}) / \Perf(A\otimes A^{\op}),$$ i.e.\ the Verdier quotient of the bounded derived category of finitely generated dg $A$-$A$-bimodules by the full subcategory of perfect dg $A$-$A$-bimodules. This statement was originally proven in Proposition 6.9 of \cite{Wan21} when $A$ is a (non-graded) symmetric Frobenius algebra and extended in Proposition 3.11 of \cite{RivWan19} to the case when $A$ is a symmetric dg Frobenius algebra.  

The singularity category was used in \cite{Orl06} to study singularities of algebraic varieties.

  % Based on the topological meaning of $C_*(A,A)$, $C^*(A,A)$ and the map $\gamma: A \to A$ when $A$ is the symmetric dg Frobenius algebra model over a field $\mathbb{K}$ for a simply connected manifold $M$, we may
  
  In topology, when $A$ is a commutative symmetric dg Frobenius model for $C^*(M,\mathbb{K})$ for a simply connected manifold $M$, Remarks~\ref{rem:HHalgtop} and \ref{rem:HHdual}, we can  think of $\mathcal{D}^*(A,A)$ as a way of connecting the singular chains and cochains on $LM$ into a single unbounded complex via the Euler characteristic of $M$. Indeed, the map $\gamma: A \to A$ in that case takes the product with the element $\sum_ie_if_i$, that identifies with the Euler class of $M$. In other words, the map $\gamma$ is determined by taking a representative of the Poincare dual of the fundamental class $[M]$ to the Euler characteristic $\chi(M)$ thought of as a top dimensional cochain on $M$ by using a representative of the volume form. On cohomology this is just multiplication by $\chi(M)$ thought of as a map $\mathbb{K}\cong H^0(A) \to H^n(A)\cong \mathbb{K}$.
%The pairing $\langle -, -\rangle_{\mathcal{D}}$ may be interpreted as a kind of Poincar\'e duality for $LM$ deduced from the finite dimensional approximations for $LM$ given by a cosimplicial model for $LM$ in terms of the Cartesian products of manifolds $M^n= M \times \dots \times M$.\Nnote{ref for where this is discussed?}%
 A symplectic version of the Tate-Hochschild construction has been described and studied in \cite{CiFrOa10, CieOan20} by combining symplectic homology and cohomology via a ``V-shaped" Hamiltonian.
  \end{remark}

\subsection{Two operations on Hochschild complexes}\label{sec:algop}
We recall the classical cup product on the Hochschild cochains of a dg algebra, and define afterwards a form of dual operation on the Hochschild chains.

\begin{definition}\label{cup} Let $A$ be a dg $\mathbb{K}$-algebra. The \textit{cup product}
$$ \cup: C^{m,*}(A,A) \otimes C^{n,*}(A,A) \to C^{m+n,*}(A,A)$$ is defined on any $f\in C^{m, *}(A, A), g\in C^{n, *}(A, A)$ by the formula
$$f \cup g (\overline{a_1} \otimes \cdots  \otimes \overline{a_{m+n}})= (-1)^{|g|\epsilon_m}f(\overline{a_1} \otimes \cdots  \otimes \overline{a_m})g(\overline{a_{m+1}} \otimes \cdots  \otimes \overline{a_{m+n}}),$$ where $\epsilon_m=\sum_{i=1}^m |a_i| -m$.
\end{definition}
The cup product gives rise to an associative product of degree $0$ on $C^*(A,A)$ that satisfies the graded Leibniz identity with respect to the Hochschild cochains differential $\delta$. Therefore $(C^*(A,A), \delta, \cup)$ is a dg algebra and, consequently, the induced product on $\HH^*(A,A)$ defines a graded associative algebra structure. This computes the endomorphism graded algebra $\text{Ext}^*_{A \otimes A^{op}}(A,A)$ with the categorical Yoneda product. 

We now describe a product on the Hochschild chains of a symmetric dg Frobenius algebra that behaves as a ``dual" to this cup product, following \cite[Section 2.3]{RivWan19}. This product has also appear in a slight variation in e.g.~\cite[Section 6]{Abb15} and \cite[Example 2.12]{Kla13B}.

A dg algebra $A$ is \textit{connected} if it is non-negatively graded and $A^0=\mathbb{K}$. When $A$ is a Frobenius of dimension $n$,  finitely generated free as a $\mathbb{K}$-module,  this implies that also $A^n \cong \mathbb{K}$.

\begin{definition}\label{defn:algebraic coproduct}
Suppose $A$ is a connected symmetric dg Frobenius $\mathbb{K}$ algebra of dimension $n>0$.
The \textit{algebraic Goresky-Hingston product} $$* \colon  C_*(A,A) \otimes C_*(A,A) \to C_*(A,A)$$ is defined on any  $\alpha = \overline{a_1} \otimes \dotsb \otimes \overline{a_p} \otimes a_{p+1}$ and $\beta = \overline{b_1} \otimes \dotsb \otimes \overline{b_q} \otimes b_{q+1}$ by the formula 
$$\alpha * \beta  =\sum_i (-1)^{\eta_i}\overline{b_1}\otimes \cdots \otimes \overline{b_{q+1}e_i}\otimes \overline{a_1}\otimes \cdots \otimes \overline{a_p} \otimes a_{p+1}f_i,$$
%\Fnote{I "lengthened" the bar from $\overline{b_{q+1}}e_i$ to $\overline{b_{q+1}e_i}$}
where $\eta_i = |\alpha| |f_i| +|b_{q+1}| + (|\alpha|+n-1) (|\beta| +n-1)$. 
The product $*$ induces a degree zero product on the $(1-n)$-shifted graded $\mathbb{K}$-module $s^{1-n}C_*(A,A)$.
\end{definition}
Note that $*$ does \textit{not} satisfy the Leibniz rule with respect to the Hochschild chains differential $\partial$. In fact, the product $*$ may be understood as a secondary operation, or a chain homotopy, between two operations. If $p>0$ and $q>0$ we do have
\begin{align}\label{leibniz-rule}
\partial( \alpha * \beta) -\partial(\alpha) * \beta - (-1)^{|\alpha|+k-1}\alpha * \partial( \beta)=0.
\end{align}
However, if $p=0$, so that $\alpha= a_1 \in C_{0, *}(A, A) =A$, we may compute
$$\partial(\alpha * \beta) -\partial(\alpha) * \beta - (-1)^{|\alpha|+n-1}\alpha * \partial(\beta)= \sum_i (-1)^{\eta_i + |\beta|-1-|b_{q+1}|} \overline{b_1} \otimes \dotsb \otimes \overline{b_q} \otimes b_{q+1}e_ia_1f_i.$$
% An analogous computation yields that, if $q=0$, there is a similar obstruction for $*$ to satisfy the Leibniz rule.
The case $q=0$ is analogous. 

Note that, for degree reasons, $e_ia_1f_i$ is only non-zero if $a_1 \in A^0 \cong \mathbb{K}$ and, in such case, $e_ia_1f_i \in A^n\cong \mathbb{K}$.
It follows that $*$ induces a well-defined chain map on the complement of $C_{0,0}(A,A)=A^0\cong \mathbb{K} \subset C_*(A,A)$, which we call the reduced Hochschild complex.

\begin{definition}\label{def:redHH} The \textit{reduced Hochschild chain complex} $\overline{C}_*(A,A)$  of a connected dg algebra $A$ is the subcomplex  $\overline{C}_{*,*}(A,A) \subset C_{*,*}(A,A)$ given by  $\overline{C}_{0,0}(A,A)=0$ and $\overline{C}_{i,j}(A,A) = C_{i,j}(A,A)$ for all pairs of integers $(i,j) \neq (0,0)$. We denote by $\overline{\HH}_*(A, A)$ its homology.
  %of the complex $\overline C_*(A,A)$ with the differential induced by the Hochschild chains differential. We call  $\overline{\HH}_*(A, A)$ the \textit{reduced Hochschild chain complex}. 
\end{definition}

The algebraic Goresky-Hingston product $*$ gives rise to an associative product of degree $0$ 
$$*: s^{1-n}\overline C_*(A,A) \otimes s^{1-n}\overline C_*(A,A) \to s^{1-n}\overline C_*(A,A)$$
 that satisfies the graded Leibniz identity with respect to the reduced Hochschild chains differential. The elements that lead to obstructions for the Leibniz rule on $C_*(A,A)$ to be satisfied are now removed in the sub-complex $\overline C_*(A,A)$. Consequently, the induced product on  $s^{1-n}\overline{\HH}_*(A, A)$ defines a graded associative algebra structure.

% \begin{remark}\label{rem:comm}
%\Nnote{switching and commutativity}
%   \end{remark}

\subsection{Cyclic $A_{\infty}$-algebra on the Tate-Hochschild complex}\label{sec:Tate2}
The following natural questions now arise:
\begin{enumerate}[label=(\subscript{Q}{{\arabic*}})]
\item In what sense are the products $\cup$ and $*$ dual to each other?
\item What is the compatibility between $\cup$ and $*$ and what is the general algebraic structure they are part of?
\item Do $\cup$ and $*$ satisfy a form of homotopy invariance?
\item Is there a homological interpretation for the product $*$ similar to the interpretation of $\cup$ as the endomorphism algebra in the derived category of $A$-$A$-bimodules?
\item What is the precise relationship between the geometrically defined Chas-Sullivan and Goresky-Hingston operations and $\cup$ and $*$?
\end{enumerate}
Question $(Q_5)$ will be discussed in Section~\ref{sec:conf}, following \cite{NaeWil19}. The following two statements adress the remaining questions $(Q_1)$---$(Q_4)$, saying in particular that $\cup$ and $*$ naturally combine to a single product on the Tate-Hochschild complex.
% and opens up further research directions. The main idea is to combine both operations $\cup$ and $*$ as part of a larger structure in the Tate-Hochschild complex.

\begin{theorem} \cite[Theorem 6.3, Proposition 6.5]{RivWan19} \label{thm1}
Let $\mathbb{K}$ be a field and $A$ be a connected symmetric dg Frobenius $\mathbb{K}$-algebra of dimension $n$. There exists a (strictly unital) $A_{\infty}$-algebra structure $\{m_1,m_2, m_3, \cdots\}$ on $\mathcal{D}^*(A,A)=s^{1-n}C_*(A,A) \oplus C^*(A,A)$ such that
\begin{enumerate}
\item $m_1= \delta$ is the Tate-Hochschild complex differential, $m_2$ extends both $*$ and $\cup$ (i.e. $m_2|_{s^{1-n}C_*(A,A)}= *$, $m_2|_{C^*(A,A)}= \cup),$ and $m_i=0$ for $i>3$.
\item The $A_{\infty}$-algebra  is cyclically compatible with the pairing $\langle -, -\rangle_{\mathcal{D}}$:
  $$\langle m_p(\alpha_0\otimes\cdots \otimes \alpha_{p-1}), \alpha_p\rangle_{\mathcal{D}}=(-1)^{|\alpha_0| (|\alpha_1|+ \cdots  + |\alpha_p|} )\langle m_p(\alpha_1\otimes\cdots\otimes \alpha_p), \alpha_0\rangle_{\mathcal{D}}.
$$
\item The induced homology product is (graded) commutative, and there is an isomorphism of graded algebras
  $$H^*(\mathcal{D}^*(A,A)) \cong \HH^*_{sg}(A,A),$$ where the latter is the endomorphism algebra from $A$ to itself in the singularity category of $A$-$A$-bimodules.  
\item Connes' operator $B: C_*(A,A) \to C_{*-1}(A,A)$ extends to an operator $B_{\mathcal{D}}: \mathcal{D}^*(A,A) \to \mathcal{D}^{*-1}(A,A)$ satisfying $B_{\mathcal{D}} \circ \delta+ \delta \circ B_{\mathcal{D}}=0$, $B_{\mathcal{D}} \circ B_{\mathcal{D}}=0$, and making $H^*(\mathcal{D}^*(A,A))$ into a BV-algebra. 
\end{enumerate}
\end{theorem}
% We now discuss the significance of the above result and some consequences.

Statement $(3)$ in Theorem \ref{thm1} provides a homological algebra interpretation for the graded associative algebra structure on $H^*(\mathcal{D}^*(A,A))$, thus giving an answer to $(Q_4)$, while an answer to questions $(Q_1)$ and $(Q_2)$ is given by $(1)$ and $(2)$.
%after noticing that $(s^{1-k}\overline{\HH}^*(A,A), *)$ sits inside $H^*(\mathcal{D}^*(A,A))$ as a sub-algebra. 

\begin{remark}[Manin triples]\label{rem:Manin}
Using constructions and language originated in the theory of quantum groups, we can say a little more about $(Q_1)$ and $(Q_2)$.  Denote the associative product on $H^*(\mathcal{D}^*(A,A))$ by $$\star: H^*(\mathcal{D}^*(A,A)) \otimes H^*(\mathcal{D}^*(A,A)) \to H^*(\mathcal{D}^*(A,A)).$$ Observe that there is an isomorphism $$H^*(\mathcal{D}^*(A,A)) \cong H^*(\text{ker}(\tilde\gamma))  \oplus H^*(\text{coker}(\tilde\gamma)),$$
where $\tilde\gamma: s^{1-n}C_{*}(A,A)\to C^*(A,A)$ is the map $\gamma$ considered as map of chain complexes that is mostly zero. 
In this language, the above result imply the existence of a commutative product $\star$ on the direct sum $H^*(\text{ker}(\gamma))  \oplus H^*(\text{coker}(\gamma))$,  together with a pairing $\langle-,-\rangle_{\mathcal{D}}$, satisfying the following properties:
\begin{enumerate}[label=(\roman*)]

\item The pairing $\langle-,-\rangle_{\mathcal{D}}$ of Definition \ref{Tatepairing} is non-degenerate with respect to the ``monomial length'' chain level filtration on  $\mathcal{D}^{*,*}(A,A)= s^{1-n}C_{*,*}(A,A) \oplus C^{*,*}(A,A)$. More precisely, it induces an isomorphism of graded vector spaces $$C_{-m,*}(A,A) \xrightarrow{\cong} C^{m,*}(A,A)^{\vee}.$$

\item  For any $x,y,z \in H^*(\text{ker}(\gamma))  \oplus H^*(\text{coker}(\gamma))$ we have $\langle x \star y, z \rangle_{\mathcal{D}} = \langle x, y \star z \rangle_{\mathcal{D}}$.

\item Both $(H^*(\text{coker}(\gamma)), \cup)$ and $(H^*(\text{ker}(\gamma)), *) $ are isotropic sub-algebras of  $$(H^*(\text{ker}(\gamma))  \oplus H^*(\text{coker}(\gamma)), \star)$$ with respect to the pairing $\langle -,- \rangle_{\mathcal{D}}$.
\end{enumerate}

The algebraic structure just described is reminiscent of a \textit{Manin triple},  a notion originally introduced in the context of quantum groups. A Manin triple consists of a triple of Lie algebras $(\mathfrak{g}, \mathfrak{g}_+, \mathfrak{g}_{-})$ over a field $\mathbb{K}$ such that $\mathfrak{g}= \mathfrak{g}_+ \oplus  \mathfrak{g}_-$ as vector spaces and $\mathfrak{g}$ is equipped with a symmetric bilinear pairing $\langle- ,-\rangle_{\mathfrak{g}}: \mathfrak{g} \otimes \mathfrak{g} \to \mathbb{K}$ satisfying $\langle [x,y],z\rangle_{\mathfrak{g}} = \langle x,[y,z]\rangle_{\mathfrak{g}}$, inducing an isomorphism $\mathfrak{g}_+ \cong \mathfrak{g}_-^{\vee}$, and for which $\mathfrak{g}_+$ and $\mathfrak{g}_-$ are isotropic Lie sub-algebras. If $\mathfrak{h}$ is a finite dimensional Lie algebra then there is a $1$-$1$ correspondence between Manin triples with $\mathfrak{g}_+=\mathfrak{h}$ and Lie bialgebra structures on $\mathfrak{h}$. In particular, if $\mathfrak{g}$ is a Lie bialgebra then one can describe a canonical Lie bialgebra structure on $\mathfrak{g} \oplus \mathfrak{g}^{\vee}$ called the \textit{Drinfeld double of $\mathfrak{g}$}. Drinfeld showed this construction yields a quasi-triangular Lie bialgebra. A complete reference for these notions and results is \cite{ChPr}.

Define analogously a \textit{graded commutative Manin triple} to be a triple of graded commutative $\mathbb{K}$-algebras $(V, V_+,V_-)$ over a field $\mathbb{K}$ such that
\begin{enumerate}[label=(\roman*)]
 \item $V=V_+ \oplus V_-$ as a vector space and $V$ is equipped with a symmetric bilinear pairing $\langle-,-\rangle_V: V\otimes V \to \mathbb{K}$ inducing an isomorphism $V_+ \cong V_-^{\vee}$,

 \item for any $a,b,c \in V$, we have $\langle ab,c\rangle_V=\langle a,bc\rangle_V$, and 
 
 \item  both $V_+$ and $V_-$ are isotropic sub-algebras of $V$.
 \end{enumerate}
 As in the Lie case,  one can use the duality given by the pairing to reformulate the defining equations of this structure in terms of a type of bialgebra structure on $V$. More precisely, if $W$ is a finite dimensional graded commutative algebra, there is a $1$-$1$ correspondence between graded commutative Manin triples with $V_+=W$ and graded commutative cocommutative \textit{infinitesimal bialgebra} structures on $W$, as introduced by Joni and Rota in \cite{JonRot82}. 
 % Infinitesimal bialgebras were introduced by Joni and Rota in \cite{JonRot82} and have been studied in more detail by Aguiar \cite{Agu00}. The notion of (non-graded) commutative Manin triple has been studied by Bai and Ni \cite{BaiNi13}.
 The data of a graded infinitesimal bialgebra structure on $W$ consists of a product $\cdot: W \otimes W \to W$ of degree $0$ and coproduct $\Delta: W \to W \otimes W$ of degree $k$ such that $\Delta$ is a derivation of the product, namely
 $$\Delta( a \cdot b ) = \Delta(a) \cdot b + (-1)^{|a|k} a \cdot \Delta(b),$$ where we define $(a' \otimes a'')\cdot b:= a' \otimes (a''\cdot b)$ and $a \cdot (b' \otimes b'') := (a\cdot b') \otimes b''$. See  \cite{Agu00} for more about infinitesimal biaglebras.  See \cite{BaiNi13} for (a non-graded version of) the correspondence between commutative cocommutative infinitemsial bialgebras with Manin triples of commutative algebras and, more generally, between Poisson bialgebras and Manin triples of Poisson algebras.
 %Hence, the algebraic information about the compatibility between the two products $\cup$ and $*$  is encoded in a graded commutative Manin triple structure with $V_+= H^*(\text{ker}(\gamma))$, $V_-=H^*(\text{coker}(\gamma))$, and bilinear form $\langle - ,- \rangle_\mathcal{D}$. 
\end{remark}

The following result provides an answer to question $(Q_3)$.

\begin{theorem}\cite[Theorem 1.1]{RivWan21} \label{thm2} Let $\mathbb{K}$ be a field and $(A, \langle -, -\rangle_A)$ and $(B, \langle -, -\rangle_B)$ be two simply connected symmetric dg Frobenius $\mathbb{K}$-algebras of dimension $n$. Suppose that there is a zig-zag of quasi-isomorphisms of dg algebras $$A \xleftarrow{\simeq} \bullet \xrightarrow{\simeq} \dotsb \xleftarrow{\simeq} \bullet \xrightarrow{\simeq} B.$$ Then there is an isomorphism of algebras $$(H^*(\mathcal{D}^*(A,A)), \star) \cong (H^*(\mathcal{D}^*(B,B)), \star)$$ restricting to an isomorphism of subalgebras $$(s^{1-n}\overline{\HH}_*(A,A), *)  \cong (s^{1-n}\overline{\HH}_*(B,B), *).$$
\end{theorem} 
The proof the above theorem relies on the homological interpretation of the Tate-Hochschild cohomology algebra as the endomorphism algebra in the singularity category of $A$-$A$-bimodules (see Remark~\ref{rem:Tatealgtop}). %also known as \textit{singular Hochschild cohomology}.
The isomorphism class of the latter, just like for the Hochschild cohomology algebra with cup product, is an invariant of the quasi-isomorphism type of the underlying dg algebra. A  careful analysis of the relationship between Tate-Hochschild cohomology and singular Hochschild cohomology allows to conclude that the isomorphism $(H^*(\mathcal{D}^*(A,A)), \star) \cong (H^*(\mathcal{D}^*(B,B)), \star)$ restricts to an isomorphism $(s^{1-n}\overline{\HH}_*(A,A), *)  \cong (s^{1-n}\overline{\HH}_*(B,B), *)$ in the simply connected case. We refer to \cite{RivWan21} for further details. 

As a direct consequence of Theorem \ref{thm2} is  the following.

\begin{corollary}  \cite[Corollary 1.2]{RivWan21}
\label{corollary1}
\begin{enumerate} 
\item Let $M$  be a simply connected oriented closed manifold of dimension $n$ and $A$ a Poincar\'e duality model for the cdga of rational polynomial forms $\Apl(M,\mathbb{Q})$, as provided by Theorem \ref{lastthm}. The isomorphism class of the graded algebra structure on $s^{1-n}\overline{\mathrm H}^*(LM;\mathbb{Q})$ induced by the product $*: s^{1-n}\overline{\HH}_*(A,A)^{\otimes 2} \to s^{1-n}\overline{\HH}_*(A,A)$ through the isomorphism $\overline{\mathrm H}^*(LM;\mathbb{Q} ) \cong \overline{\HH}_*(A,A)$ is independent of the choice of Poincar\'e duality model $A\simeq\Apl(M,\mathbb{Q})$. 
\item If $M$ and $M'$ are homotopy equivalent simply connected oriented closed manifolds of dimension $n$, then the algebra structures on $s^{1-n} \overline{\mathrm H}^*(LM;\mathbb{Q})$ and $s^{1-n}\overline{\mathrm H}^*(LM';\mathbb{Q})$ are isomorphic. 
\end{enumerate}
\end{corollary}

\subsection{Final remarks} 
One would like to understand the complete algebraic chain level structure of the Tate-Hochschild complex of a symmetric dg Frobenius algebra. The type of cyclic $A_{\infty}$-algebra described in Theorem \ref{thm1} is a finite type version of a notion discussed in \cite{IyuKonVla21} under the name of \textit{Pre Calabi-Yau} algebra. It is explain there how the associator $m_3$ of a Pre-Calabi Yau algebra gives rise to a \textit{double Poisson bracket}. A precise formula for the map $m_3$ on the Tate-Hochschild complex may be found in Remark 6.4 of \cite{RivWan19}. 

This is only the tip of the iceberg of a very rich algebraic structure on the Tate-Hochschild complex. Part (4) of Theorem \ref{thm1} tells us that $B_{\mathcal{D}}$ and the product $\star$ define a $BV$-\textit{algebra} structure on $H^*(\mathcal{D}^*(A,A))= H^*(\text{coker}(\gamma)) \oplus H^*(\text{ker}(\gamma))$.  By definition, a BV-algebra consists of a triple $(V, \star, B)$ where $(V, \star)$ is a graded commutative algebra, $B: V \to V$ is a degree $-1$ operator satisfying $B \circ B = 0$, and the operation $$\{ x, y \} := B(x \star y) -B(x) \star y - (-1)^{|x|} x \star B(y)$$ is a Lie bracket pf degree $-1$ which is a derivation of $\star$ on each variable, i.e. $\{ -, - \}$ is Poisson compatible with $\star$.

The $BV$-algebra structure on Tate-Hochschild cohomology extends the $BV$-algebra structure of the Hochschild cohomology of a symmetric dg Frobenius algebra. Furthermore, in \cite{KauRivWan} we lift the $BV$-algebra structure of Tate-Hochschild cohomology to the chain level, building upon the framework of \cite{Kau07, Kau08}, solving a cyclic Deligne conjecture for the Tate-Hochschild complex. 
The Lie bracket associated to the $BV$-algebra structure on Tate-Hochschild cohomology gives rise to a compatible (Lie) graded Manin triple structure on $(H^*(\mathcal{D}^*(A,A)), H^*(\text{coker}(\gamma)), H^*(\text{ker}(\gamma)))$ extending the classical Gerstenhaber algebra structure on Hochschild cohomology. This Lie algebra structure on $H^*(\mathcal{D}^*(A,A))$  was also lifted to a cyclic $L_{\infty}$-algebra structure on $\mathcal{D}^*(A,A)$ in \cite{RivWan19}. After dualizing and completing the tensor product appropriately, we obtain on $H^*(\mathcal{D}^*(A,A))$ a graded commutative cocommutative infinitesimal bialgebra equipped with a Gerstenhaber bracket and a Gerstenhaber cobracket that are Lie bialgebra compatible.  Furthermore, the Gerstenhaber bracket and the cocommutative coproduct, as well as the Gerstenhaber cobracket and the commutative product, satisfy additional second order compatibility equations. This algebraic structure, which may be called a \textit{Gerstenhaber bialgebra},  is a graded version of a \textit{Poisson bialgebra}, defined and studied in \cite{BaiNi13}. 

Gerstenhaber bialgebras are reminiscent of similar structures appearing in the theory of quantum groups, where associated to a Lie bialgebra $\mathfrak{g}$, such as the structure induced on the tangent Lie algebra of a Poisson-Lie group, one may consider the commutative cocommutative Hopf algebra $S(\mathfrak{g})$, the symmetric algebra on the vector space $\mathfrak{g}$, with the Poisson bracket and Poisson cobracket induced by the Lie bialgebra structure on $\mathfrak{g}$. Then one proceeds to deform the product to obtain the non-commutative cocommutative universal enveloping algebra $U(\mathfrak{g})$ and then deforms the coproduct in the Poisson cobracket direction to obtain a non-commutative non-cocommutative Hopf algebra $U_h(\mathfrak{g})$. Motivated by the above discussion and by the question of constructing examples of non-commutative non-cocommutative infinitesimal bialgebras one can replace the notion of Hopf algebra by infinitesimal bialgebra. More precisely, one could ask if given a Poission bialgebra $A$ there exists a deformation to a (possibly non-commutative non-cocommutative) infinitesimal bialgebra $A[[h]]$ in the direction of the Poisson bracket and cobracket. One may also study analogous questions in the graded setting for Gerstenhaber bialgebras. 
 
Lie bialgebras also appear in $S^1$-equivariant string topology. In fact, the Chas-Sullivan loop product and the Goresky-Hingston loop coproduct induce a Lie bialgebra structure once we pass to the reduced $S^1$-equivariant homology of the free loop space of a manifold. This structure generalizes previous constructions of Goldman and Turaev from surfaces to manifolds of arbitrary dimension \cite{Sul04}, \cite{Gol86}, \cite{Tur91}.  In the algebraic context, this construction is modeled by a dg Lie bialgebra structure on the reduced cyclic chain complex of a dg Frobenius algebra (\cite{CheEshGan11}, \cite{NaeWil19}, \cite{CieFukLat20}), a construction foreshadowed by Ginzburg's necklace Lie bialgebra \cite{Gin01}. Turaev described the quantization of the Lie bialgebra structure on the zeroth $S^1$-equivariant homology of the free loop space of a surface in terms of skein invariants of links in $3$-manifolds. This quantization has also been studied from an algebraic perspective: in \cite{Sch05} a quantization of Ginzburg's necklace Lie bialgebra of a quiver is constructed and this is generalized in \cite{CheEshGan11} where a quantization of the Lie bialgebra on the cyclic homology of a Frobenius algebra is constructed. We expect that the functorial theory of quantization of Lie bialgebras described by Etingof and Kazhdan in \cite{EtKa96} may be adapted to quantize infinitesimal bialgebras in the direction of a compatible bracket and cobracket. This theory should give rise to explicit and interesting examples of non-commutative non-cocommutative infinitesimal bialgebras associated to dg Frobenius algebras by quantizing the infinitesimal bialgebra structure of $H^*(\mathcal{D}^*(A,A))$ in the direction of the Gerstenhaber bracket and cobracket.

\section{String topology and configuration spaces}\label{sec:conf}
%In this section $\mathbb{K} = \mathbb{R}$ and all (co)homology is taken with that coefficient unless stated otherwise.
In this section we compare the geometrically defined string topology operations of Section~\ref{sec:inter} with the ones defined
algebraically using a dg Frobenius model as in Section~\ref{sec:HH},
under the assumption  that the coefficients $\mathbb{K} = \mathbb{R}$ are the real numbers. The main ingredient is an algebraic model for the Fulton-McPherson compactification of $M \times M \setminus M$, the configuration space of two points in $M$.

Let $M$ be a simply connected oriented closed manifold. By a theorem of Lambrechts and Stanley (stated here as Theorem \ref{lastthm}), applied to the case $\mathbb{K} = \mathbb{R}$, there exists a commutative symmetric dg Frobenius algebra $A$ quasi-isomorphic to real cochains $C^*(M,\R)$. As discussed in Remark \ref{rem:HHalgtop}, we have isomorphisms
\begin{equation}\label{equ:jones}
HH_{*}(A,A) \cong HH_{*}(C^*(M;\R),C^*(M;\R)) \cong H^*(LM;\R).
\end{equation}
%\Fnote{I agree that, technically, I should find a reference that proves that all the those people define the same (homotopic) quasi-isomorphisms, so that later I can just talk about "the" isomorphism. But I don't think such a reference exists.}\Nnote{I'll live with it... }
% 
% Jones' isomorphism \cite{Jon87} \Nnote{or Chen \cite{Chen73,GeJoPe91}? }.
\begin{definition} \label{def:relHH} Define the \textit{relative Hochschild complex} by 
\[
\underline{C}_{*}(A,A) = \bigoplus_{m \geq 1} (s\overline{A})^{\otimes m} \otimes A
\]
Because $A$ is commutative,  $\underline{C}_{*}(A,A)$ is a sub chain complex of $C_*(A,A)$. 
\end{definition} 
The chain complex $\underline{C}_{*}(A,A)$ may also be regarded as the kernel of the natural chain map
$C_{*}(A,A) \to A$, which models the map
$\operatorname{cst} \colon M \to LM$ (see Example~\ref{ex:LM}). Hence \eqref{equ:jones} restricts to an isomorphism
\begin{equation}\label{equ:reljones}
\underline{HH}_*(A,A) \cong H^*(LM, M;\R). 
\end{equation}
%\Fnote{I slightly rephrased things here.}
% As $A$ is commutative, we have  a splitting $HH_*(A,A)\cong H_*(A)\oplus \underline{HH}_*(A,A)$ where $\underline{HH}_*(A,A)$ is the homology of the positive Hochschild complex $\oplus_{m \geq 1} s \overline{A}^{\otimes m} \otimes A$. Under the above isomorphism, this splitting corresponds to the splitting $H^*(LM)\cong H^*(M)\oplus H^*(LM,M)$, 
% so that the \eqref{equ:jones} restricts to an isomorphism
% \[
% \underline{HH}_*(A,A) \cong H^*(LM, M;\R). 
% \]
The algebraic Goresky-Hingston product given in Definition \ref{defn:algebraic coproduct} restricts to a product on 
this relative version of the Hochschild chain complex
(see also e.g., \cite[Sec 6]{Abb15}). The purpose of this section is to sketch a proof  of the following result: 
\begin{theorem}\cite[Theorem 1.3]{NaeWil19}\label{thm:alggeo}
Let $M$ be a simply-connected oriented closed manifold with commutative dg Frobenius algebra model $A\simeq C^*(M;\R)$. Then the isomorphism (\ref{equ:reljones}) %\Nnote{requires that we say which iso we mean above...}
\[
\underline{HH}_*(A,A) \cong H^*(LM, M;\R),
\]
intertwines the algebraic with the topological Goresky-Hingston product of Definitions~\ref{def:GeoCo} (dualised) and \ref{defn:algebraic coproduct}.
\end{theorem}
We will follow the line of argument of \cite{NaeWil19}. 
A similar argument to the one presented here gives the equivalence between the algebraic and topological Chas-Sullivan products of Definitions~\ref{def:GeoPro} and \ref{cup} (dualized), giving an alternative proof of \cite[Theorem 11]{Felix-Thomas}. Here we focus on the Goresky-Hingston product.

We will use the definition of the coproduct given in Section~\ref{sec:badcop}. Before embarquing into the proof of the theorem in Section~\ref{sec:real}, we will take a closer look at the crucial step in the definition of the coproduct, namely the intersection map, defining a general notion of {\em intersection products} (see Sections~\ref{sec:41} and \ref{sec:42}). Section~\ref{sec:invariance} then analyses invariance properties of such  intersection products. 

\begin{remark}[Dependence on the manifold $M$]
  Note that we can take $A$ {\em any} commutative dg Frobenius model of $C^*(M;\R)$ in the statement. As the right hand side in the theorem is model-independent, it follows that the algebraic Goresky-Hingston product on $H_*(A,A)$ does not depend on the particular model $A$. This partially recovers Corollary \ref{corollary1}.
  
%  the theorem is true for any dg Frobenius model which is guaranteed to exist by the Theorem of Lambrechts-Stanley).
%Note that the left-hand side only depends on the rational/real homotopy type  This is in line with the second part of Corollary \ref{corollary1} that the Goresky-Hingston coproduct is homotopy invariant in the 1-connected case.

 We saw in Section~\ref{sec:computations} through a lens space example  that  the coproduct on $H_*(LM)$  is in general not a homotopy invariant of $M$,
  at least with integral coefficients, see  Theorem~\ref{thm:inv}. 
% However, as the example of lens spaces shows, this is not true in the non-simply connected case.
In the proof of Theorem~\ref{thm:alggeo}, the topology of $M$ will enter through the homotopy type of the complement of the diagonal $M \times M \setminus M$. This last space identifies with the configuration space of 2 points, a space known to depend in general on more than the homotopy type from the same lens space example, see \cite{LonSal}.
We will use a recent result by Campos-Willwacher and Idrissi \cite{Campos-Willwacher, Idrissi} to obtain an algebraic model for this space in the case of simply-connected manifolds (together with some compatibility datum). 
\end{remark}

To simplify presentation and notation, we will show the corresponding statement for the operation
\[
H_{* + n -1}(LM;\R) \to H_*(LM,M;\R)\xrightarrow{\vee} H_*(LM,M;\R)^{\otimes 2},
\]
that is the pre-composition with the canonical map $H_*(LM;\R) \to H_*(LM,M; \R)$. %where $\vee$ itself and more details, we refer to \cite{NaeWil19}. 

\subsection{Intersection products}\label{sec:41}

Recall from Section~\ref{sec:badcop} that the loop coproduct can be defined as a relative version of the trivial coproduct $\vee_{\frac{1}{2}}$, intersecting with the figure eights space $\Figeight\subset LM$. The crucial step in this definition of the coproduct is the composition
\begin{equation}\label{equ:cap12}
 R_{\frac{1}{2}}\circ ((\ev_{0,\frac{1}{2}})^*\tau_M\cap) \colon H_*(LM, \RR)  \xrightarrow{\ \ }  H_{* - n}(\Figeight, \RR),
  \end{equation}
  see \eqref{equ:cop2}. Here $\RR$ is the subspace of half-constant loops, $\ev_{0,\frac{1}{2}}=(\ev_0, \ev_{\tfrac{1}{2}}) \colon LM \to M \times M$ is the evalutation at $0$ and $\frac{1}{2}$, the cochain $\tau_M \in C^n(M\x M, M\x M \setminus M)$ is a representative of the Thom class of the normal bundle of the diagonal $M\to M\x M$, and $R_{\frac{1}{2}}$ is a retraction map. In Sections~\ref{sec:41}--\ref{sec:invariance}, homology can be taken with integral coefficients.
% \Nnote{ok?}
  
Note that $\Figeight$ is the pullback of $\ev_{0,\frac{1}{2}}$ along the diagonal
$$\xymatrix{\Figeight \ar[r]\ar[d]_{\ev_0} & LM \ar[d]^{\ev_{0,\frac{1}{2}}} \\ M\ar[r]^-\Delta & M\x M,}$$
and one can show that, just like the evaluation map $\ev_0$,  the map $\ev_{0,\frac{1}{2}}$ is a fibration. 
The map \eqref{equ:cap12} is the lift along $\ev_{0,\frac{1}{2}}$ of the intersection product $H_*(M\x M)\xrightarrow{\bullet} H_{*-n}(M)$, taken relative to $\RR$.  We will think of it as a ``relative intersection product'' and will now abstract what is needed to define it. 

\subsubsection{Relative intersection products}
The definition of the relative intersection product \eqref{equ:cap12} immediately generalizes to the following situation. 
Suppose  $p_\EE:\EE \to M \times M$ is a fibration, and $\RR$ is a space equipped with maps $p_\RR:\RR \to M$ and $f: \RR\to \EE$ such that the diagram 
\begin{equation}\label{equ:ER}
\xymatrix{
\RR \ar[r]^-f \ar[d]_{p_\RR} & \EE \ar[d]^{p_\EE} \\
M \ar[r]^-\Delta & M \times M}
\end{equation}
commutes.
%We then wish to define a map
%\[
%H_\bullet(\EE, \RR) \to H_{\bullet - n}( \EE|_{\Delta}, \RR).
%\]
%Moreover, setting $\FF = \varnothing$ we obtain an absolute intersection product
%\[
%H_\bullet(\EE) \to H_{\bullet -n}(\EE|_{\Delta}),
%\]
%which is refined by the relative intersection product in the sense that
%\[
%\begin{tikzcd}
%H_\bullet(\EE) \ar[r] \ar[d] & H_{\bullet - n}(\EE|_\Delta) \ar[d] \\
%H_\bullet(\EE, \RR) \ar[r] & H_{\bullet - n}(\EE|_\Delta, \RR)
%\end{tikzcd}
%\]
%commutes.
%
% One can define such an operation just as above in section \ref{sec:defproco} by
% choosing neighborhoods of the diagonal and retractions of
% $\EE|_{U_\epsilon}$ that are compatible with the inclusion of $\RR \to
% \EE|_{U_\epsilon})$ and do a Thom-Pontrijagin type construction as
% above. Instead we use the following modification, which we think of as
% choosing an infinitesimal neighborhood of $M \subset M \times M$ as
% opposed to an $\epsilon$-neighborhood.
%
%
%The definition of the relative intersection product is given by expanding the definition of the capping map in \eqref{equ:cop2} as in \eqref{equ:cap}. To that extent, we choose a tubular neighborhood of the diagonal $M \subset M \times M$ and obtain the following diagram
%\begin{equation}\label{diag:geom intersection context}
%\begin{tikzcd}
%& TM \setminus M \ar[r] \ar[d] & M \times M \setminus M \ar[d] \\
%M \ar[r, "\sim"] & TM \ar[r] & M \times M
%\end{tikzcd}
%\end{equation}
%and choose a representative of the Thom class $\tau \in C^n(TM, TM \setminus M)$.
%
From this data, we can define the following zig-zag of chain maps: 
\[
\begin{tikzcd}
C_*(\EE) \ar[r]& C_*(\EE, \EE|_{M \times M \setminus M}) & \ar[l, "\sim"'] C_*(\EE|_{U_M}, \EE|_{U_M \setminus M}) \ar[r, "\cap p_\EE^* \tau_M"] & C_{* -n}(\EE|_{U_M}) & \ar[l, "\sim"'] C_{* - n}(\EE|_M),
\end{tikzcd}
\]
where $TM\cong U_M\subset M\x M$ is a tubular neighborhood of the diagonal as in Section~\ref{sec:int0}.  Both wrong-way maps are quasi-isomorphisms: the first one by excision and the second one since we are pulling back a fibration along the homotopy equivalence $M\arsim U_M$. Thus we get a map in homology
\begin{equation}\label{equ:absint}
H_*(\EE) \xrightarrow{\intp_M} H_{* -n}(\EE|_M),
\end{equation}
which we call the  {\em (absolute) intersection product} associated to the fibration $p_\EE$.
To refine this operation to a relative version, we note that the following diagram commutes.
\begin{equation}\label{diag:def of rel int}
\begin{tikzcd}
C_*(\EE) \ar[r]& C_*(\EE, \EE|_{M \times M \setminus M}) & \ar[l, "\sim"] C_*(\EE|_{U_M}, \EE|_{U_M \setminus M}) \ar[r, "\cap p_\EE^*\tau"] & C_{* -n}(\EE|_{U_M}) & \ar[l, "\sim"] C_{* - n}(\EE|_M) \\
C_*(\RR) \ar[u,"f"]\ar[r, equal]& C_*(\RR) \ar[u,"f"] \ar[r, equal]  & C_*(\RR) \ar[u,"f"] \ar[r, "\cap f^* p_\EE^* \tau"] & C_{* -n}(\RR) \ar[u,"f"] & \ar[l, equal] C_{* - n}(\RR)  \ar[u,"f"]
\end{tikzcd}
\end{equation}
Taking vertical mapping cones, this again defines a zig-zag of complexes such that the wrong-way maps are quasi-isomorphisms and thus we obtain a map in homology
\begin{equation}\label{equ:relint}
H_*(\EE, \RR) \xrightarrow{\intp_M}  H_{* - n}(\EE|_M, \RR)
\end{equation}
which we call the {\em relative intersection product} associated to the diagram \eqref{equ:ER}.

\begin{proposition}
For $\EE = LM$ with $p_\EE=\ev_{0,\frac{1}{2}}=(\ev_0, \ev_{\frac{1}{2}}) \colon LM \to M \times M$ and $\RR\hookrightarrow LM$ the space of half-constant loops,  the  operation %\eqref{equ:relint} 
\[
\intp_M\colon H_*(LM,\RR) \rar H_{* -n}(\Figeight,\RR)
\]
coincides with the corresponding map in the definition \eqref{equ:cop2} of the loop coproduct. 
\end{proposition}
\begin{proof}
  The first three commuting squares in \eqref{diag:def of rel int} are simply spelling out the details in \eqref{equ:cop2} (as in \eqref{equ:cap}), with the only difference that a homotopy inverse to excision was chosen in \eqref{equ:cop2}. 
  % and identifying $TM$ with its image $U_\epsilon$ in $M \times M$.
The last step follows from the fact  that the retraction map $R_{\tfrac{1}{2}}$ in \eqref{equ:cop2} is a homotopy inverse to the inclusion $\Figeight \hookrightarrow LM|_{U_M}$ (this is essentially \cite[Lemma 2.11]{HinWah0}), thus inducing an inverse to the map $H_*( \Figeight, \RR)\to H_*( LM|_{U_M}, \RR )$ in relative homology. 
%$H_*( LM|_{TM}, \RR ) \longleftarrow H_*( \Figeight, \RR)$.
\end{proof}

Similarly, we obtain the loop product as an examle of the (non-relative) intersection product: 
\begin{proposition}\label{prop:loop product from intersection}
For $\EE = LM \times LM$ with $p_\EE = (\ev_0, \ev_0) \colon LM \times LM \to M \times M$ and $\RR = \varnothing$, the operation %\eqref{equ:relint}
\[
\intp_M\colon H_*(LM \times LM) \to H_{*-n}(\Figeight)
\]
coincides with the corresponding map in the definition \eqref{equ:CS} of the loop product. 
\end{proposition}

The following properties of the relative intersection product follow directly from the definitions. 
\begin{proposition}
The relative intersection product \eqref{equ:relint} is natural in diagrams \eqref{equ:ER} over a fixed manifold $M$ and refines the absolute intersection product \eqref{equ:absint} in the sense that
\[\begin{tikzcd}
H_*(\EE, \RR) \ar[r,"\intp_M"] & H_{* - n}(\EE|_{M}, \RR) \\
H_*(\EE) \ar[r,"\intp_M"] \ar[u] & H_{* - n}(\EE|_{M}) \ar[u]
\end{tikzcd}\]
commutes. The absolute intersection product is natural in fibrations $p_\EE$, and identifies with the classical intersection product of Section~\ref{sec:int0} in the case $\EE = M \times M$ with $p_\EE=\id$: 
\[
H_*(M \times M) \xrightarrow{\intp_M=\bullet} H_{* -n}(M).
\] 
\end{proposition}

%\begin{remark} 
%One actually gets a map between the long exact sequences of the pairs $(\EE, \RR)$ and $(\EE|_M, \RR)$ where the map $H_\bullet(\RR) \to H_{\bullet -n}(\RR)$ is given by capping with the pullback of the Euler class along $\RR \to M$.
%\end{remark}

\subsection{Intersection contexts}\label{sec:42}
The definition of the relative intersection product uses the following data from the manifold:  the diagram
\begin{equation}\label{diag:geom intersection context}
\xymatrix{
& U_M \setminus M \ar@{^(->}[r] \ar[d] & M \times M \setminus M \ar[d] \\
M \ar[r]^-\sim & U_M \ar@{^(->}[r] & M \times M
}
\end{equation}
and the class $\tau_M \in H^n(M\x M, M\x M \setminus M)\cong H^n(U_M, U_M \setminus M)$. This is also the data used to define the classical intersection product.  
We note that the spaces $U_M$, $U_M \setminus M$ and $M \times M \setminus M$ only appear in the intermediate steps of the definition. 

We now describe a slight generalization of a relative intersection product 
\[
H_*( \EE, \RR) \to H_{* -n}(\EE|_M, \RR).
\]
Such a construction may be defined from the data of a diagram \eqref{equ:ER} as before together with the ``manifold data'' recorded by any homotopy pushout diagram of the shape
\begin{equation}\label{diag:extended intersection context}
\begin{tikzcd}
& A \ar[r] \ar[d] & B \ar[d] \\
M \ar[r, "\sim"] & C \ar[r] & M \times M
\end{tikzcd}
\end{equation}
equipped with a class $\tau \in H^n(C,A)$. Indeed, if we denote $\EE|_A,\EE|_B,\EE|_C$ the pull-back of $\EE$ along the maps $A,B,C\to M\x M$, to construct the relative intersection using the corresponding zig-zag \eqref{diag:def of rel int}, all we need is that the maps 
\[
C_*(\EE, \EE|_B) \longleftarrow C_*(\EE|_C, \EE|_A)
\]
and
\[
C_*(\EE|_C) \longleftarrow C_*(\EE|_M)
\] are quasi-isomorphisms. For the second one, this follows as before from our assumption that $M \to C$ is a homotopy equivalence, given that $p_\EE$ is a fibration. For the first, it follows from the assumption that \eqref{diag:extended intersection context} is a homotopy pushout, using Mather's second cube theorem \cite[Theorem 25]{Mather} applied to the pullback of the square along the fibration $p_\EE$, as a replacement of excision.

\begin{definition}\label{def:intcont}
  We call a homotopy pushout diagram of the shape \eqref{diag:extended intersection context} an {\em intersection context}, 
  % if the square is a homotopy pushout diagram and the map $M \to C$ is a homotopy equivalence. We call
 and a cohomology class $\tau \in H^n(C,A)$ an \em{$n$-orientation}.
\end{definition}

We define two oriented intersection contexts to be equivalent if there is a zig-zag of diagrams that is a pointwise homotopy-equivalence, compatible with the orientations.
A  diagram chase gives the following.
\begin{proposition}\label{prop:invint}
Two equivalent oriented intersection contexts associate the same relative intersection map
\[
 \intp_M\colon H_*(\EE, \RR) \to H_{* - n}(\EE|_M, \RR)
\]
to a tuple $(\EE,\RR,p_\EE,p_\RR,f)$ as in diagram \eqref{equ:ER}. 
\end{proposition}

%Finally, we can construct a slightly simpler yet equivalent oriented intersection context than \eqref{diag:geom intersection context}.

% \subsubsection*{The intersection product revisited}
% Let us first explain the construction in case where $\EE = M \times M$ and $\FF = \varnothing$ i.e. how one defines the classical intersection product
% \[
% H_\bullet(M \times M) \to H_{\bullet -n}(M).
% \]
% One way to define it is by the following sequence of maps
% \[
% H_\bullet(M \times M) \to H_\bullet(M \times M, M \times M \setminus M) \overset{\sim}{\leftarrow} H_\bullet(TM, TM \setminus M) \overset{\cap \tau}{\to} H_{\bullet -n}(TM) \to H_{\bullet - n}(M).
% \]
% The two main ingredients are a Thom class $\tau \in H^n(TM, TM \setminus M)$ and excision. Excision is equivalent to the fact that
% \[
% \begin{tikzcd}
% TM \setminus M \ar[r] \ar[d] & M \times M \setminus M \ar[d] \\
% TM \ar[r] & M \times M,
% \end{tikzcd}
% \]
% is a homotopy pushout diagram. 
% We can replace it by the following equivalent diagram.
The intersection context we will be using in our proof of Theorem~\ref{thm:alggeo} is the following. 
Let $FM_2$ denote the Fulton-McPherson compactification of the configurations space of two points. It is obtained as the real oriented blowup of $M \times M$ along the diagonal. That is $FM_2$ is a manifold with boundary whose interior is $M \times M \setminus M$ and with boundary the unit tangent bundle $UTM$ of $M$. In particular, it fits into the following commuting square
\begin{equation}
\label{diag:FM_2}
\begin{tikzcd}
& UTM \ar[r] \ar[d] & FM_2 \ar[d] \\
M \ar[r,equal] & M \ar[r] & M \times M.
\end{tikzcd}
\end{equation}
%Given an orientation on $M$ we also obtain a Thom class $\tau \in H^n(M, UTM)$.

\begin{proposition}
 Together with the  class  $\tau_M \in H^n(M, UTM)\cong H^n(U_M,U_M\backslash M)$, Diagram \eqref{diag:FM_2}  defines an  oriented intersection context equivalent to \eqref{diag:geom intersection context}.
\end{proposition}
\begin{proof}
There is a zig-zag of equivalences between the two diagrams coming from the pair of zig-zag $UTM \rar \overline U_M\minus M\lar U_M\minus M$ and $FM_2=FM_2\lar M\x M\minus M$, for $\overline U_M\minus M$ an epsilon neighborhood of $UTM$  in $FM_2$.
\end{proof}

\subsection{Invariance of intersection products}\label{sec:invariance} 

Suppose $f:M\to N$ is a smooth map, and that $M$ comes equipped with an intersection context, for example one of the form \eqref{diag:FM_2}. 
% with chosen intersection contexts and a homotopy equivalence $f \colon M \to N$. Let us for simplicity assume that the intersection contexts are of the form as in \eqref{4.9}. By
Composing with $f$, we obtain an intersection context for $N$ from that of $M$. We denote the corresponding relative intersection product by $f_* \intp_M$. By construction we have the following naturality property:
\begin{lemma}\label{lem:fint} For 
\[
\begin{tikzcd}\RR\ar[r]\ar[d] & \EE \ar[d]\\ N\ar[r]& N\x N\end{tikzcd}
\]
 as in \eqref{equ:ER}, the square
%\vspace{-6mm}
% $$\hspace{8cm}  
\[
\begin{tikzcd}
    H_*(f^*\EE, f^*\RR) \ar[r, "\intp_M"] \ar[d] & H_{* - n}(f^*\EE|_M, f^*\RR) \ar[d] \\
 H_*(\EE, \RR) \ar[r, "f_* \intp_M"] & H_{* - n} (\EE|_N, \RR)
 \end{tikzcd}
\]
commutes, 
where $f^*\EE$ and $f^*\RR$ are the homotopy pullback of $\EE$ and $\RR$ along $f \times f \colon M \times M \to N \times N$ and  $f \colon M \to N$. Note that the vertical maps are isomorphisms if $f$ is a homotopy equivalence.
\end{lemma}

We are interested in the case $\EE = LN \to N \times N$ with $\RR_N \to N$ the space of half-constant loops, as defined in \ref{sec:badcop}. In that case, $f$ also induces compatible natural maps $LM \to LN$ and $\RR_M \to \RR_N$ giving a commuting diagram
 \[
 \begin{tikzcd}
 H_*(LM, \RR_M) \ar[r, "\intp_M"] \ar[d] & H_{* - n}(\Figeight_M, \RR_M) \ar[d] \\
 H_*(f^*LN, f^*\RR_N) \ar[r, "\intp_M"] \ar[d] & H_{* - n}(f^*LN|_M, f^*\RR_N) \ar[d] \\
 H_*(LN, \RR_N) \ar[r, "f_* \intp_M"] & H_{* - n} (\Figeight_N, \RR_N),
 \end{tikzcd}
 \]
where again the vertical arrows are all isomorphisms if $f$ is a homotopy equivalence. 
Hence comparing the loop coproduct for two manifolds $M$ and $N$
%relative intersection products $\intp_M$ and $\intp_N$ coming from intersection contexts of the form \eqref{diag:FM_2} for each manifolds
is equivalent to comparing the relative intersection products $f_*\intp_M$ and $\intp_N$ on the pair $(LN,\RR_N)$.
In general, these are not equal. Otherwise, since the loop coproduct may be described in terms of the above intersection products (as in Proposition~\ref{prop:goodandbad}), this would yield a proof for homotopy invariance of the loop coproduct, contradicting Theorem \ref{thm:inv}.

\medskip

In contrast, the loop product is known to satisfy homotopy invariance (see Theorem~\ref{thm:producthtpy}), and the (failed) line of argument suggested above for the coproduct does go  through for the product. 
The essential difference is that the loop product only uses the non-relative intersection product (see Proposition~\ref{prop:loop product from intersection}). Its homotopy invariance follows from the following result.

%In view of the previous section, we have reduced the question about homotopy invariance of the loop product/coproduct to the question about homotopy invariance of the (relative) intersection product. 

%Suppose that we have two manifolds $M$ and $N$ with chosen intersection contexts, and fibrations $\EE_M\to M\x M$ and $\EE_N\to N\x N$. 
%Given a map $f:M\to N$,  we get a  pulled back fibration $f^*\EE_N\to M\x M$, and we can try to compare the intersection products of $\EE_M,f^*\EE_N$ and $\EE_N$, and similarly in the relative case. 
%
%When $\EE_M=LM$ and $\EE_N=LN$, the map $f$ also induces a map $Lf:\EE_M\to \EE_N$ that factors through $f^*\EE_N$. If $\RR$ is in each case taken to be the subspace of half-constant loops, Proposition~\ref{prop:invint} and its non-relative analogue gives that
%$\intp_M\EE_M\cong\intp_Mf^*\EE_N$...  \Nnote{fix!}
%but ... 
%
%
%
%It turns out that the absolute intersection product is indeed homotopy invariant.

\begin{theorem}\label{thm:invEE}
Let $f \colon M \to N$ be an orientation-preserving homotopy equivalence of manifolds, each equipped with its intersection context of the form \eqref{diag:FM_2}. Then   for any fibration $\EE \to N \times N$ the intersection product
\[
\intp_N  \colon H_*(\EE) \rar H_{*-n}(\EE|_N)
\]
% associated to the intersection context \eqref{diag:FM_2}
coincides with the transferred intersection product $f_* \intp_M$. 
%\Fnote{Changed $\EE$ to be over $N$ and removed statement about pullback. Everything is supposed to be happening over the target manifold. It seems easier to push forward those intersection contexts than pulling them back.}
\end{theorem}

\begin{proof}[Sketch proof]
  The above theorem is proved in the papers \cite{CKS,Cra,GruSal,Felix-Thomas} in the context of string topology, i.e.~in the special case when $\EE=LN\to N\x N$,  as the crucial ingredient in the homotopy invariance of the loop product, and the proofs generalise to our context. 
  The proof of Gruher-Salvatore in  \cite{GruSal} is closest to our langage, so we follow that paper.
  Translating to our notation, Theorem 8 in that paper defines a
  product preserving map $\theta_f$ in homology from $(\EE,\intp_N)$ to $(f^*\EE,\intp_M)$. This map can be composed by the product-preserving map
  $(f^*\EE,\intp_M)\to (\EE,f_*\intp_M)$ given by the non-relative version of Lemma~\ref{lem:fint}. As both maps preserve the product, it is enough to show that they compose to the identity on $\EE$. This statement corresponds to the last display in the proof of Proposition 23 in \cite{GruSal}. This last computation is only stated in the case of the loop space in that paper, but it comes from an analysis of the maps using Thom isomorphisms that only use what the maps do on the underlying manifolds. 
% Gruher-Salvatore work with fiberwise monoids over manifold $E\to M$. In their notation, $E^{-TM}$ is the ring spectrum with product, in homology, coming from the intersection product $\intp_M$ for the fibraion $\EE=E\x E\to M\x M$. (The monoid part of the fiber is only used for concatenation at the end of the definition of the product.) To show that a map $f:M\to N$ induces an isomorphism of the associated ring spectra, they frist consider the maps $LM^{-TM}\to f^*LN^{-TM}\leftarrow LN^{TN}$. Both maps are ring maps by their Theorem 8, that is stated in the more general set-up of fiberwise monoids. 
  \end{proof}

%\begin{theorem}
%Let $f \colon N \to M$ be an orientation-preserving homotopy equivalence of manifolds. Then
%\[\begin{tikzcd}
%H_*(f^*\EE) \ar[r, "{\text{int}_N}"] \ar[d] & H_{* - n}(f^*\EE|_\Delta) \ar[d] \\
%H_*(\EE) \ar[r, "{\text{int}_M}"] & H_{* - n}(\EE|_\Delta),
%\end{tikzcd}
%\]
%where $\text{int}_N$ and $\text{int}_M$ are the intersection products in $N$ and $M$, respectively,
%commutes for any fibration $\EE \to M \times M$.
%\end{theorem}

%in \cite{CKS}, it is proved 
% The homotopy invariance of the string topology loop product and string bracket, Cohen-Klein-Sullivan
%by showing that the configuration space of two points becomes homotopy invariant after a certain stabilization, while in  \cite{Cra} and \cite{GruSal}  a version of Atiyah duality in the topological context is used. See also \cite{Felix-Thomas} for a purely algebraic version of a similar idea, in the context of rational homotopy theory. The generalised statement written here is obtained by replacing $LM$ by any fibration $\EE\to M\x M$ in those proves.
An alternative approach to the above statement is to use paramertrized homotopy theory as in \cite{May-Sigurdsson}, identifying the intersection product considered here with the evaluation map of the Costenoble-Waner duality for $M$.

% \Fnote{mention: String topology on Gorenstein spaces, Felix-Thomas,
%                 GENERALIZED STRING TOPOLOGY OPERATIONS, Gruher-Salvatore,
%                 LOOP HOMOLOGY AS FIBREWISE HOMOLOGY, Crabb,
%                 The homotopy invariance of the string topology loop product
% and string bracket, Cohen-Klein-Sullivan}

% \Fnote{Also follows from May-Sigurdsson - 18.6. Parametrized Atiyah duality for closed manifolds}

\begin{remark}
As the example of lens spaces shows (Theorem \ref{thm:inv}), the above theorem does not generalize to the relative intersection product. The above argument fails in that the composition $\iota\circ \theta_f$ may fail to be equal to the identity in relative homology. This is equivalent to the lack of a Thom isomorphism type map in the computation to be an isomorphism in relative homology, relating to the issue discussed in  \cite[Sec 3.8]{HinWah1}. 
\end{remark}

\subsection{Equivalence between algebraic and geometric models for the loop coproduct}\label{sec:real} 
We will now give a sketch of the proof of Theorem \ref{thm:alggeo}. We first describe real models (in the sense of rational homotopy theory) for each of the steps in the definition of the loop coproduct and compare the final result with the description in \ref{defn:algebraic coproduct}. More precisely, up to crossing with an interval,  we can write the geometric coproduct \eqref{equ:cop2} as the composition of the following three maps:  
\begin{equation}
\label{diag:defn of coproduct}
\begin{tikzcd}
C_{{*} + n}(LM \times I, LM \times \partial I) \xrightarrow{J} C_{{*} + n}(LM, \mathcal{R}) \xrightarrow{\text{int}} C_{{*}}(\Figeight, \mathcal{R}) \xrightarrow{\text{cut}} C_{*}(LM,M)^{\otimes 2},
\end{tikzcd}
\end{equation}
where the middel map is the intersection product discussed in Sections~\ref{sec:41} and \ref{sec:42}.
We will give models for each of these three maps. Most of what we do
in this section can be done with rational coefficients; real
coefficients will only be needed at the very end of the section, when
picking a particular model of the configuration space $FM_2$. For simplicity, we will ignore sign issues in this section. 

A major ingredient will be the Eilenberg-Moore theorem, that we will use to give rational models of homotopy pull-backs. We will apply it to the functorial rational model of polynomial forms $\Apl$, with $\Apl(X)\simeq C^*(X;\mathbb Q)$:  
%For that let us first recall the following theorem that says that the functor $\Apl(-)$ \Nnote{why $\Apl$ and not plain de Rham form as we are over the real? also your sentence gives the impression that you are using $\Apl$ because it has that property but  McCleary uses singular cochains anyway.}\Fnote{Because $\Apl$ or $\Apl \otimes \R$ is defined for any topological space and not just manifolds. In general Eilenberg-Moore tells you how to get the complex of cochains on the fiber product, but not what the algebra structure is. We will actually never use the algebra structure but we need to at least know that cochains of the fiber product are modules over cochains over the factors.} sends homotopy pullbacks to homotopy fiber products of algebras at least if the base is simply connected. 
\begin{theorem}[Eilenberg-Moore; see for instance
Theorem 7.14 in \cite{McCleary}]\label{thm:Eilenberg-Moore} 
% "https://www.math.uni-hamburg.de/home/holstein/lehre/RHTnotes.pdf" Theorem 7.37 and Lemma 9.13
% \Fnote{Find better reference, maybe Theorem 7.14 in McCleary "User's guide to spectral sequences"}
Suppose that
\[
\begin{tikzcd}
W \ar[r,"f"] \ar[d, "g"'] & \ar[d] X \\
Y \ar[r] & Z
\end{tikzcd}
\]
is a homotopy pullback of spaces, such that $Z$ is simply-connected and either $X$ or $Y$ are connected. Then the natural map 
\[
\Apl(X) \otimes^L_{\Apl(Z)} \Apl(Y) \rar \Apl(X) \otimes_{\Apl(Z)} \Apl(Y) \rar \Apl(W)
\]
induced by $f^* \colon \Apl(X) \to \Apl(W)$ and $g^* \colon \Apl(Y) \to \Apl(W)$ is a quasi-isomorphism. Here $\Apl(X) \otimes^L_{\Apl(Z)} \Apl(Y)$ denotes the derived tensor product.
%that can be computed using the standard bar resolution.
\end{theorem}

In the following we  use the bar construction model for the derived tensor product: 
\[
\Apl(X) \otimes^L_{\Apl(Z)} \Apl(Y) \ =\  \bigoplus_{p \geq 0} \Apl(X) \otimes s\overline{\Apl(Z)}^{\otimes p} \otimes \Apl(Y),
\]
with differential analogous to that of the Hochschild complex (see Definition~\ref{definition:hochschildchain}). Note that with this definition $\Apl(X) \otimes^L_{\Apl(Z)} \Apl(Y)$ is a quasi-free (i.e.~free after forgetting the differential) $\Apl(X) \otimes \Apl(Y)$-module. Moreover, there is an $\Apl(X) \otimes \Apl(Y)$-module map
\[
\Apl(X) \otimes \Apl(Y) \rar \Apl(X) \otimes^L_{\Apl(Z)} \Apl(Y)
\]
given by inclusion of the $(p=0)$--summand.

The map
\[
\Apl(X) \otimes^L_{\Apl(Z)} \Apl(Y)  \rar \Apl(W)
\]
in the theorem is then given by projecting onto the $\Apl(X) \otimes \Apl(Y)$ summand on which the map is $f^* \cup g^*$. 
We obtain the following commutative diagram
\begin{equation}\label{equ:pointed}
\begin{tikzcd}
\Apl(X) \otimes^L_{\Apl(Z)} \Apl(Y) \ar[r, "\sim"] & \Apl(W) \\
\Apl(X) \otimes \Apl(Y) \ar[u] \ar[ur] & 
\end{tikzcd}
\end{equation}
of $\Apl(X) \otimes \Apl(Y)$--modules. In other words, both $\Apl(X) \otimes^L_{\Apl(Z)} \Apl(Y)$ and $\Apl(W)$ come with a distinguished element, called the \emph{pointing} and the equivalence respects that distinguished element. That is, we have the following

% or in other words, the top map in the diagram is an equivalence of {\em pointed $\Apl(X) \otimes \Apl(Y)$-modules}. 

%Let us summarize the discussion as follows. The algebra map $\Apl(X) \otimes \Apl(Y) \to \Apl(W)$ turns $\Apl(W)$ into a pointed $\Apl(X) \otimes \Apl(Y)$-module, that is $\Apl(W)$ is an $\Apl(X) \otimes \Apl(Y)$-module and the map $\Apl(X) \otimes \Apl(Y) \to \Apl(W)$ is a map of $\Apl(X) \otimes \Apl(Y)$-modules.

\begin{corollary}
The map 
$\Apl(X) \otimes^L_{\Apl(Z)} \Apl(Y) \overset{\sim}{\rightarrow} \Apl(W)$ of Theorem~\ref{thm:Eilenberg-Moore} 
is a quasi-isomorphism of pointed $\Apl(X) \otimes \Apl(Y)$--modules.
\end{corollary}

\begin{example}[The Hochschild complex as a model for $LM$]\label{ex:LM}
%[Chen-Jones-Goodwillie]
The loop space $LM$ can be defined as a pullback 
\[\xymatrix{
LM \ar[rr]^-{\ev_0}\ar[d] &&  M \ar[d]^{\De_M} \\
PM \ar[rr]^-{\ev_0\x \ev_1} & & M \times M.
}
\]
We then obtain the following zig-zag of pointed $\Apl(M) \otimes \Apl(M)$--modules
\[
\Apl(M) \otimes^L_{\Apl(M) \otimes \Apl(M)} \Apl(M) \overset{\sim}{\longleftarrow} \Apl(PM) \otimes^L_{\Apl(M) \otimes \Apl(M)} \Apl(M) \overset{\sim}{\longrightarrow} \Apl(LM),
\]
where the first arrow is the quasi-isomorphism induced by  $\Apl(PM) \simeq \Apl(M)$ and the second one comes from Theorem \ref{thm:Eilenberg-Moore}. The above zig-zag thus exhibits $\Apl(M) \otimes^L_{\Apl(M) \otimes \Apl(M)} \Apl(M)$ as a model for $C^*(LM;\mathbb Q)$. Additionally, we obtain a model for the map $\ev_0:LM \to M$ as follows: 
\[
\begin{tikzcd}
\Apl(M) \otimes^L_{\Apl(M) \otimes \Apl(M)} \Apl(M) &\ar[l, "\sim"] \Apl(PM) \otimes^L_{\Apl(M) \otimes \Apl(M)} \Apl(M) \ar[r, "\sim"] & \Apl(LM) \\
\Apl(M) \ar[u, "1 \otimes \operatorname{id}"] \ar[r, equal] & \Apl(M) \ar[u, "1 \otimes \operatorname{id}"] \ar[r, equal] & \Apl(M) \ar[u, "\ev_0"]. 
\end{tikzcd}
\]
Finally, let $$B\Apl(M) = \oplus_{p \geq 0} \Apl(M) \otimes \overline{s\Apl(M)}^{\otimes p} \otimes\Apl(M)$$
denote the two-sided bar construction computing $\Apl(M) \otimes^L_{\Apl(M)} \Apl(M)$. There is a quasi-isomorphism of pointed $\Apl(M) \otimes \Apl(M)$--modules
\[
B\Apl(M) \arsim \Apl(M).
\]
Since $B\Apl(M)$ is a quasi-free $\Apl(M) \otimes \Apl(M)$-module we obtain quasi-isomorphisms
\[\begin{tikzcd}
B\Apl(M) \bigotimes_{\Apl(M) \otimes \Apl(M)} \Apl(M)& \ar[l, "\sim"] B\Apl(M) \otimes^L_{\Apl(M) \otimes \Apl(M)} \Apl(M) \ar[d, "\sim"]  \\
\Apl(M) \ar[u, "1 \otimes \operatorname{id}"] \ar[ur, "1 \otimes \operatorname{id}"] \ar[r, "1 \otimes \operatorname{id}"]   &  \Apl(M) \otimes^L_{\Apl(M) \otimes \Apl(M)} \Apl(M). 
\end{tikzcd}
\]
%\[\begin{tikzcd}
%B\Apl(M) \bigotimes_{\Apl(M) \otimes \Apl(M)} \Apl(M)& \ar[l, "\sim"] B\Apl(M) \otimes^L_{\Apl(M) \otimes \Apl(M)} \Apl(M) \ar[r, "\sim"] & \Apl(M) \otimes^L_{\Apl(M) \otimes \Apl(M)} \Apl(M)  \\
%\Apl(M) \ar[u, "1 \otimes \operatorname{id}"] \ar[r, equal]  & \Apl(M) \ar[u, "1 \otimes \operatorname{id}"] \ar[r, equal] & \Apl(M) \ar[u, "1 \otimes \operatorname{id}"]
%\end{tikzcd}
%\]
The left hand side is now exactly the definition of the Hochschild complex: 
\[
C_{*}(\Apl(M), \Apl(M)) \equiv B\Apl(M) \otimes_{\Apl(M) \otimes \Apl(M)} \Apl(M).
\]
This shows that the Hochchild complex, as a pointed $\Apl(M)$-module, is a model for $\ev_0 \colon LM \to M$,
giving a  proof of the isomorphism \eqref{equ:jones} in the rational case. (See also \cite[Proposition 1]{FTrational}.)

Note that the above computations also shows that the map $\Apl(M) \otimes \Apl(M) \to B\Apl(M)$  is a model for the fibration $PM \to M \times M$.
\end{example}

Let $A\simeq \Apl(M)$ be any commutative dg algebra model of $M$. We can now replace $\Apl(M)$ by $A$ in the models for $LM$ and $PM$ that we have obtained. 
%In what follows,  we let $A$ %\Nnote{should we use $A_M$ like in section 3?}\Fnote{I would prefer a shorter symbol, the formulas are already barely readable}\Nnote{agreed}
As above,  let
\[
BA = \bigoplus_{p \geq 0} A \otimes (s\overline{A})^{\otimes p} \otimes A
\]
be the two sided bar-resolution, considered as a pointed $A \otimes A$-module.
%As demonstrated in the example, since this is quasi-isomorphic to $A$ as a pointed $A \otimes A$-module we can use this as a pointed $A \otimes A$-module model for the path fibration.
We then have models 
\[
A\otimes A \rar BA\ \ \ \textrm{and}\ \ \  A \rar  C_{*}(A, A) 
\]
for the fibrations  $\ev_0\x \ev_1: PM \to M$ and $\ev_0 \colon LM \to M$. The latter map admits a section $\operatorname{cst} \colon M \to LM$, which, under the identification $LM\cong PM \times_{M \times M} M$, is given by the diagonal embedding. Analysing the zig-zags in Example~\ref{ex:LM} we find that it is modelled by
\[
C_{*}(A,A) = BA \otimes_{A^{\otimes 2}} A \rar A \otimes_{A^{\otimes 2}} A \overset{m}{\rar} A
\]
where $m \colon A \otimes A \to A$ is the multiplication map of $A$.

%\begin{equation}
%\begin{tikzcd}

%\end{tikzcd}
% \end{equation}

%\begin{remark}\Nnote{too difficult to read. Write more and/or give a ref or drop}
%Using Chen's iterated integrals one can define an explicit map from the bar-complex computing the derived relative tensor product to the standard model of the homotopy fiber product given as $B \times_D PD \times_D C$ (i.e. a point in $B$ and a point in $C$ together with a path in $D$ connecting their images in $D$), where $PD$ is the path space of $D$.
%\end{remark}

%Inspired by the previous example we can now write down a model for the three maps appearing in \eqref{diag:defn of coproduct}.

\subsubsection{Reparametrization map $J$}\label{sec:J}
In this section, we will give a model of the reparametrization map
\[
J\colon sC_{*}(LM) \cong C_{*}(LM \times I, LM \times \partial I) \rar C_{*}(LM, \RR).
\]
%\Nnote{I think it is confusing that $LM$ should suddently be called $\EE$} \Fnote{The reason was not to confuse $LM \to M$ and $LM \to M \times M$. On the algebraic side we never silently identify them, so I thought we also shouldn't on the geometric side. But I see your point!}
We have so far seen that the Hochschild complex $C_*(A,A)$ can be used to model the loop space together with the evaluation and inclusion maps $LM\leftrightarrows M$. This model however 
does not come with a convenient description of the map
$\ev_{0,\frac{1}{2}}=(\ev_0, \ev_\frac{1}{2}):LM\to M\x M$. 
We start the section by giving a model of $LM$ that is more convenient to describe that map.  
%indeed
%we know from Section~\ref{sec:41} that this will be necessary to describe the coproduct.
%and the Hochschild model $C_\bullet(A,A)$ 
% Using theorem \eqref{thm:Eilenberg-Moore} we can find models for $\EE$ as follows. Recall that $\EE = LM$ but thought of as a space over $M \times M$, that is we could take $C_\bullet(A,A)$ as a model for $\EE$ but then the map $\EE \to M \times M$ would be more complicated.

\begin{lemma}\label{lem:LM2}
  The fibration $\ev_{0,\frac{1}{2}}: LM\to M\x M$ admits the following pointed $A \otimes A$-module model:
  \[
    A\otimes A \rar %BA \otimes_{A \otimes A} BA:=
    A^{\otimes 2}\otimes_{A^{\otimes 4}} (BA)^{ \otimes 2} 
\]
where $A^{\otimes 2}$ is an $A^{\otimes 4}$ module via the map $(x,y,z,w) \to (xz,yw)$. %and the map is the map $A\to BA$.
%and we will write elements as $c \otimes \overline{\bf a} \otimes \overline{\bf b} \otimes d = c \otimes (\overline{a_1} \otimes \dotsb \otimes \overline{a_p}) \otimes (\overline{b_1} \otimes \dotsb \otimes \overline{b_q}) \otimes d$.\Nnote{...} 
As a vector space
\[
 % BA \otimes_{A \otimes A} BA
   A^{\otimes 2}\otimes_{A^{\otimes 4}} (BA)^{ \otimes 2} = \bigoplus_{p,q\ge 0}  (s\bar{A})^{\otimes p} \otimes A \otimes (s\bar{A})^{\otimes q} \otimes A
\]
and the map is the inclusion into the summand with $p,q=0$.
%\Nnote{removed the alternative notation and moved an A in the middle in the vect space identification}
  \end{lemma}

\begin{proof}
The map $\ev_{0,\frac{1}{2}}=(\ev_0, \ev_\frac{1}{2}) \colon LM \to M \times M$ is the product over $M \times M$ of two copies of the path fibration:
$$\xymatrix{LM \ar[rr]\ar[d] && PM \x PM \ar[d] \\
M\x M  \ar[rr]^-{\De_M\x \De_M}  && (M\x M) \x (M\x M).} $$
As in the Example~\ref{ex:LM} we use $A\otimes A \to BA$ as a model for $PM \to M \times M$.
%Recall from Example~\ref{ex:PM} that $PM$ was modelled by $BA$. 
Applying Theorem~\ref{thm:Eilenberg-Moore}, we get a model for $\ev_{0,\frac{1}{2}}\colon LM \to M \times M$ as
\[
\Apl(LM) \xleftarrow{\sim} (A\otimes A)\otimes^L_{A^{\ot 4}} (BA\otimes BA) \simeq   A^{\otimes 2}\otimes_{A^{\otimes 4}} (BA)^{ \otimes 2}   %= BA \otimes_{A \otimes A} BA, 
\]
where one checks that  $a\otimes b\in A\otimes A$ is mapped to the right-hand side as claimed in the statement. 
\end{proof}

\begin{lemma}\label{lem:Fig8}%\Fnote{I hope it is okay to just silently identify the two models, they are isomorphic (not just quasi-isom) and the natural map that identifies the two is also a model for the map that identifies the two ways of saying figure eight homotopically, as the proof suggests.}
The fibration $\Figeight \to M$  admits the following pointed $A$-module model
\[
A \to C_{*}(A,A) \otimes_A C_{*}(A,A) \cong A\otimes_{A^{\otimes 4}} (BA \otimes BA) 
\]
with the cut map and inclusions  $ LM\x LM \xleftarrow{cut}\Figeight \inc LM$ given by the quotient maps
$$C_{*}(A,A) \otimes C_{*}(A,A) \rar  C_{*}(A,A) \otimes_A C_{*}(A,A) \cong A\otimes_{A^{\otimes 4}} (BA \otimes BA)  \lar  BA \otimes_{A^{\otimes 2}} BA.$$
%and the inclusion $\Figeight \to LM$ is given  by quotient map
%\[
%BA \otimes_{A^{\otimes 2}} BA \to (BA \otimes BA) \otimes_{A^{\otimes 4}} A.
%\]
\end{lemma}

\begin{proof}
  The space $\Figeight$ can be seen to be the pullback of $PM \times PM \to M^2 \times M^2=  M^4 $ along the diagonal $M \to M^4$,
  which gives us a model for $\Figeight \to M$ as $A\otimes_{A^{\otimes 4}} (BA)^{\otimes 2}$.
Consider the two factorizations of the diagonal as $M \to M^2 \to M^4$, where the second map $M^2 \to M^4$ is either $(x,y) \mapsto (x,x,y,y)$ or $(x,y) \mapsto (x,y,x,y)$. The first version exhibits $\Figeight$ as the pullback of $LM \times LM \to M \times M$ along the diagonal and gives the description of the cut map. The second version exhibits $\Figeight$ as the pullback of $\ev_{0,\frac{1}{2}} \colon LM \to M \times M$ giving the description of the inclusion $\Figeight \to LM$.
%
%The first part of the statement and the description of the cut map is obtained by seeing $\Figeight$  as the pull-back of $\ev_{0,\frac{1}{2}}:LM \to M \times M$ along the diagonal, and applying Theorem~\ref{thm:Eilenberg-Moore}. For the second part, we recall that $LM$ is itself a pull-back of $PM \to M \times M$ along the diagonal, and thus 
%
%, while the second is 
%obtained by instead considering $\Figeight$  as the pull-back of $\ev_0\x \ev_0:LM \x LM \to M \times M$ along the diagonal. The theorem also give a model of the cut map, as it is precisely the map $\Figeight \to LM\x LM$ occuring in the latter diagram. 
  \end{proof}

%\Fnote{It is the same model. They are both the pullback of $PM \times PM$ along the small diagonal in $M^4$}
%\begin{lemma}\label{lem:Fig8}
%The figure eight  $\Figeight$ space admits the following models
%\[
%(BA \otimes BA) \otimes_{A \otimes A} A \arsim \Apl(\Figeight) \xleftarrow{\sim} C_\bullet(A,A) \otimes_A C_\bullet(A,A). 
%\]
%Moreover, the cut map $cut:\Figeight \to LM\x LM$ in the right-hand model is given by the quotient map
%$$C_\bullet(A,A) \otimes C_\bullet(A,A) \rar  C_\bullet(A,A) \otimes_A C_\bullet(A,A). $$
%\end{lemma}
%
%\begin{proof}
%The first model  is obtained  by seeing $\Figeight$  as the pull-back of $\ev_{0,\frac{1}{2}}:LM \to M \times M$ along the diagonal, and applying Theorem~\ref{thm:Eilenberg-Moore}, while the second is 
%obtained by instead considering $\Figeight$  as the pull-back of $\ev_0\x \ev_0:LM \x LM \to M \times M$ along the diagonal. The theorem also give a model of the cut map, as it is precisely the map $\Figeight \to LM\x LM$ occuring in the latter diagram. 
%  \end{proof}

  The above description of the figure eight space, allows us now to give a model for the map $\RR \to LM$.
  %Namely, we have two natural maps $C_\bullet(A,A) \to C_\bullet(A,A) \otimes_A C_\bullet(A,A)$ that coincide when restricting to $A \to C_\bullet(A,A)$ and hence a natural map
%\[
%C_\bullet(A,A) \otimes_A C_\bullet(A,A) \to C_\bullet(A,A) \oplus_A C_\bullet(A,A),
%\]
%where
  Let $$C_{*}(A,A) \oplus_A C_{*}(A,A) =\cone\big(C_{*}(A,A) \oplus C_{*}(A,A) \to A)$$ where
  the map is the composition $C_{*}(A,A) \oplus C_{*}(A,A) \to A \oplus A \to A$, with the second map being the difference.
  Here, as we are working with cochain complexes, by ``cone'' we mean the following construction: 
  $$\cone(A\xrightarrow{f} B)=(A\oplus sB,d_A+d_B+f).$$
  
\begin{lemma}\label{lem:R8}
  The map $\RR \to \Figeight$ is modelled by the map 
\[
\begin{tikzcd}[row sep = tiny]
  C_{*}(A,A) \otimes_A C_{*}(A,A)\ar[r] & C_{*}(A,A) \oplus_A C_{*}(A,A) \\
  \overline{\bf a} \otimes \overline{\bf b} \otimes c \ar[r, mapsto] & \epsilon(\overline{\bf a}) (\overline{\bf b} \otimes c) \, \oplus\, \epsilon(\overline{\bf b}) (\overline{\bf a} \otimes c),
\end{tikzcd}
\]
%\[
%\begin{tikzcd}[row sep = tiny]
%    BA \otimes_{A \otimes A} BA \ar[r] & (BA \otimes BA) \otimes_{A \otimes A} A \simeq  C_\bullet(A,A) \otimes_A C_\bullet(A,A)\ar[r] & C_\bullet(A,A) \oplus_A C_\bullet(A,A) \\
%    c \otimes \overline{\bf a} \otimes \overline{\bf b} \otimes \ar[r, mapsto] & \overline{\bf a} \otimes \overline{\bf b} \otimes cd \ar[r, mapsto] & \epsilon(\overline{\bf a}) \overline{\bf b} \otimes cd \oplus \epsil%on(\overline{\bf b}) \overline{\bf a} \otimes cd,
%\end{tikzcd}
%\]
where $\epsilon(\overline{a_1} \otimes \dots \otimes \overline{a_p}) = 0$ if $p \geq 1$ and $1$ if $p = 0$. 
%Henc
% giving us the quasi-isomorphism
%\[
%\Apl(LM,\RR) \simeq \cone\big(BA \otimes_{A \otimes A} BA \to C_\bullet(A,A) \oplus_A C_\bullet(A,A)\big)
%\]
\end{lemma}

\begin{proof}
%Let us first understand the "topological origin" of the space $\RR$ and the map $\RR \to \Figeight$. To that extend, let $M \to LM$ be the inclusion of the constant loops. Taking product with itself we obtain the
Consider the commuting diagram
\[\begin{tikzcd}
LM \times LM & LM \times M \ar[l] \\
M \times LM \ar[u] & M \times M \ar[l]\ar[u]
\end{tikzcd}\]
of spaces over $M \times M$. By pulling back along the diagonal we obtain
\[\begin{tikzcd}
\Figeight & LM \ar[l] \\
LM \ar[u] & M \ar[l]\ar[u]
\end{tikzcd}
\]
% Writing $\RR$ as the homotopy pushout
% \[
% \begin{tikzcd}
% M \ar[r] \ar[d] & LM \ar[d] \\
% LM \ar[r] & \RR
% \end{tikzcd}
% \]
%
% we obtain
% %\begin{align*}
%   $ C_\bullet(A,A) \oplus_A C_\bullet(A,A)\simeq \Apl(\RR)$. % := \cone( C_\bullet(A,A) \oplus C_\bullet(A,A) \to A) \simeq \Apl(\RR)$. 
% %\end{align*}
% We have already given a description of the two maps $LM \to \Figeight$. To define a map $\RR \to \Figeight$ is the same as defining two maps $LM \to \Figeight$ that coincide when restricted to $M \to LM$. But the reason the two maps coincide in our case is because we have a map $LM \times_M LM \to \Figeight$ and $M \to LM$ is a section of $LM \to M$. 
%
% More concretely, from the commuting diagram
% \[\begin{tikzcd}
% LM \times LM & LM \times M \ar[l] \\
% M \times LM \ar[u] & M \times M \ar[l]\ar[u]
% \end{tikzcd}\]
% of spaces over $M \times M$ be obtain by pulling back along the diagonal
% \[\begin{tikzcd}
% \Figeight & LM \ar[l] \\
% LM \ar[u] & M \ar[l]\ar[u]
% \end{tikzcd}
% \]
and $\RR$ is the pushout  of the lower right triangle; the diagram thus encodes the inclusion map $\RR\to \Figeight$. Hence we can get a model for that commuting square in algebra, by pulling-back in the same way the previous square and using the naturally part of Theorem~\ref{thm:Eilenberg-Moore}.  
% $\Figeight$ by pulling back $LM \times LM \to M \times M$ along the diagonal. Moreover, the model we have for $M \to LM$ is a model of pointed $A$-modules and we can thus apply Theorem~\ref{thm:Eilenberg-Moore} (the naturality part) to obtain a model for last commuting square as
This becomes
\[\begin{tikzcd}
C_{*}(A,A) \otimes_A C_{*}(A,A) \ar[r, "\operatorname{id} \otimes \epsilon"] \ar[d, "\epsilon \otimes \operatorname{id}"] & C_{*}(A,A) \ar[d, "\epsilon"] \\
C_{*}(A,A) \ar[r, "\epsilon"] & A, 
\end{tikzcd}
\]
from which one can read off the map given in the statement. 
%We already have models for those two maps and thus obtain the description as stated.
%
%
%Lastly, the composition of maps $\RR \cong LM\x_M M \cup M\x_MLM \to LM \times_M LM \cong \Figeight \to LM$ is modelled by  \Nnote{not sure where the computation comes from, but it really should come from switching between these two models of the figure eight space that I added in, but I don't know if that helps}
\end{proof}

We now assemble the models of $LM$, $\Figeight$ and $\RR$ just obtained to give a model of reparametrization map: 

\begin{proposition}\label{prop:J}
In our models, the reparametrization map $J^*: \Apl(LM,\mathcal{R}) \to s\Apl(LM)$ 
\[
 \cone\big( A^{\otimes 2}\otimes_{A^{\otimes 4}} (BA)^{ \otimes 2}  \rar C_{*}(A,A) \oplus_A C_{*}(A,A)\big) \ \ \xrightarrow{J^*} \ \ sC_{*}(A,A) 
\]
takes $\alpha = (\overline{a_1} \otimes \dotsb \otimes \overline{a_p}) \otimes c \otimes (\overline{b_1} \otimes \dotsb \otimes \overline{b_q}) \otimes  d$
of the subcomplex $ A^{\otimes 2}\otimes_{A^{\otimes 4}} (BA)^{ \otimes 2} $ of the source to
\[
B(\alpha) := \pm (\overline{a_1} \otimes \dotsb \otimes \overline{a_p} \otimes \overline{c} \otimes \overline{b_1} \otimes \dotsb \otimes \overline{b_q}) \otimes d 
\]
in the target and maps $\beta \oplus \gamma \in C_{*}(A,A) \oplus_A C_{*}(A,A)$ to $\beta - \gamma$.
\end{proposition}

One can give a proof of the above proposition using Chen's iterated integrals, see \cite[Section 4.2]{NaeWil19}. We give here an alternative proof. 

\begin{proof}
We split the reparametrization into two maps
\[
(LM \times I, LM \times \partial I) \xrightarrow{\simeq} (LM, LM \sqcup LM) \to (LM, \RR)
\]
where  $LM \sqcup LM \to LM$ maps the two copies of $LM$ to the left (resp.~right) half-constant loops.
% Note that the first map applies a homotopy equivalence to each component of the pair. Moreover,
Now there is an equivalence of pairs (an equivalence of the corresponding cones, to be precise)
$$(LM \times I, LM \times \partial I) \xrightarrow{\simeq} (\text{pt}, LM \sqcup \text{pt})$$
via the map that sends one of the $LM$ factors to $\{ \text{pt} \}$. We can thus think of the reparametrization map as the zig-zag
\[
(\text{pt}, LM \sqcup \text{pt}) \xleftarrow{\simeq} (LM,LM \sqcup LM) \to (LM, \RR).
\]
In our rational model, this becomes a map  
\[
sC_{*}(A,A) \overset{\sim}{\longrightarrow} 
    {\operatorname{cone}\left(
    \begin{tikzcd}
     A^{\otimes 2}\otimes_{A^{\otimes 4}} (BA)^{ \otimes 2}\ar[d] \\ C_{*}(A,A) \oplus C_{*}(A,A)
    \end{tikzcd}
    \right)}
\longleftarrow
    {\operatorname{cone}\left(
    \begin{tikzcd}
     A^{\otimes 2}\otimes_{A^{\otimes 4}} (BA)^{ \otimes 2}\ar[d] \\ C_{*}(A,A) \oplus_A C_{*}(A,A)
    \end{tikzcd}
    \right)}
\]
where the first map is the inclusion $C_{*}(A,A) \to C_{*}(A,A) \oplus C_{*}(A,A)$ in the first summand, and the second map is the natural projection. It remains to give a left-inverse to the first map.
%Namely, we define
%\begin{align*}
%    B \colon BA \otimes_{A \otimes A} BA &\longrightarrow C_\bullet(A,A) \\
%    (\overline{a_1} \otimes \dotsb \otimes \overline{a_p}) \otimes (\overline{b_1} \otimes \dotsb \otimes \overline{b_q}) \otimes c \otimes d &\longmapsto  (\overline{a_1} \otimes \dotsb \otimes \overline{a_p} \otimes c \otimes \overline{b_1} \otimes \dotsb \otimes \overline{b_q}) \otimes d.
%\end{align*}
%This is not a chain map. However, we can define
One can check that 
%\[
%    {\operatorname{cone}\left(
%    \begin{tikzcd}
%    BA \otimes_{A \otimes A} BA \ar[d] \\ C_\bullet(A,A) \oplus C_\bullet(A,A)
%    \end{tikzcd}
%    \right)}
%\longrightarrow C_\bullet(A,A)[1],
%\]
 sending $(\alpha, \beta \oplus \gamma) \to B(\alpha) + \beta - \gamma$  defines such a chain model for such a homotopy inverse. The result follows. 
\end{proof}

%\begin{proposition}
%\label{prop:reparam without reparam}
%The map $(\alpha, \beta \oplus \gamma) \to B(\alpha) + \beta - \gamma$ defines a homotopy inverse to 
%\begin{equation}\label{diag:alg reparam}
%C_\bullet(A,A)[1] \overset{\sim}{\longrightarrow} 
%    {\operatorname{cone}\left(
%    \begin{tikzcd}
%    BA \otimes_{A \otimes A} BA \ar[d] \\ C_\bullet(A,A) \oplus C_\bullet(A,A)
%    \end{tikzcd}
%    \right)}
%\end{equation}
%and gives hence a rational description of the reparametrization map.
%\end{proposition}

% \newcommand{\Tot}[1]{
% \def \Tot #1{
% \operatorname{Tot}\left(
% \begin{tikzcd}
% #1
% \end{tikzcd}
% \right)}

% \begin{equation*}
%     C_\bullet(A,A) \longleftarrow \Tot{BA \otimes_{A \otimes A} BA \\ \ar[u] C_\bullet(A,A) \oplus_A C_\bullet(A,A)} \longrightarrow \Tot{BA \otimes_{A \otimes A} BA & \ar[l] BA \otimes BA \otimes_{A \otimes A} \F \\
%     C_\bullet(A,A) \oplus_A C_\bullet(A,A) \ar[u] & \ar[l]\ar[u] C_\bulilet(A,A) \oplus_A C_\bullet(A,A) \otimes \U
%     }
% \end{equation*}

% \begin{equation}
%     C_\bullet(A,A) \longleftarrow 
%     \operatorname{Tot}\left(
%     \begin{tikzcd}
%     BA \otimes_{A \otimes A} BA \\ \ar[u] C_\bullet(A,A) \oplus_A C_\bullet(A,A)
%     \end{tikzcd}
%     \right)
%     \longrightarrow
%     \operatorname{Tot}\left(
%     \begin{tikzcd}[column sep=small]
%     BA \otimes_{A \otimes A} BA \ar[r] &  BA \otimes BA \otimes_{A \otimes A} \F \\
%     C_\bullet(A,A) \oplus_A C_\bullet(A,A) \ar[r]\ar[u] & \ar[u] C_\bullet(A,A) \oplus_A C_\bullet(A,A) \otimes \U
%     \end{tikzcd}
%     \right)
% \end{equation}

\subsubsection{Cut map}

We give now a model for the cut map used in the definition of the coproduct. Its target is $C^* (LM \times LM,M\x LM\cup LM\x M)\simeq
C^*(LM,M)^{\otimes 2}$. Recall that the relative Hochschild chain complex $\underline{C}_{*}(A,A)$ is the kernel of the (surjective) map $C_{*}(A,A) \to A$ and hence a model for $C^*(LM,M)$.

% , where $C^*(LM,M)$ can be modeled by the
% {\em positive} Hochschild chains, the subcomplex 
% $$\underline{C}_\bullet(A,A):=\bigoplus_{n\ge 1}A^{\otimes n+1}$$
% of the Hochschild complex. \Nnote{not defined earlier?}

\begin{proposition}\label{prop:cut}
The cut map $(\Figeight,R) \to (LM \times LM,M\x LM\cup LM\x M)$ can be modelled as the map 
  \begin{align}\label{diag:alg cut map}
    {\operatorname{cone}\left(
    \begin{tikzcd}[column sep=small, ampersand replacement = \&]
    C_{*}(A,A) \otimes_A C_{*}(A,A) \ar[d] \\
    C_{*}(A,A) \oplus_A C_{*}(A,A)
    \end{tikzcd}
    \right)} \xleftarrow{cut}
     {\operatorname{cone}\left(
    \begin{tikzcd}[column sep=small, ampersand replacement = \&]
    C_{*}(A,A) \ar[d] \\
    A
    \end{tikzcd}
    \right)^{\otimes 2}} \xleftarrow{\sim} \ \underline{C}_{*}(A,A)^{\otimes 2}.
\end{align}
defined by
$$cut\big((\overline{a_1} \otimes \dotsb \otimes \overline{a_p} \otimes a_{p+1})\otimes  (\overline{b_1} \otimes \dotsb \otimes \overline{b_p} \otimes b_{p+1})\big)
=\pm (\overline{a_1} \otimes \dotsb \otimes \overline{a_p}) \otimes (\overline{b_1} \otimes \dotsb \otimes \overline{b_q}) \otimes a_{p+1}b_{q+1}$$
sitting in the subcomplex $  C_{*}(A,A) \otimes_A C_{*}(A,A) $ of the target. 
\end{proposition}

\begin{proof}
We have already seen in Lemma~\ref{lem:Fig8} that the cut map $\Figeight \to LM \times LM$ can be described as the quotient map
%\[
%(LM \times LM) \times_{M \times M} M \to (LM \times LM) \times_{M \times M} (M \times M),
%\]
%that is, it is described by naturality of the pullback along the diagonal map $M \to M \times M$. Using Theorem \ref{thm:Eilenberg-Moore} we obtain a cochain description of that map as
\begin{align*}
    C_{*}(A,A) \otimes_A C_{*}(A,A) \longleftarrow & C_{*}(A,A) \otimes C_{*}(A,A). 
%    \alpha \otimes_A \beta \mapsfrom &\alpha \otimes \beta.
\end{align*}
To see that this map descends to a relative map, we use  the same diagrams of spaces as in the proof of Lemma~\ref{lem:R8}. 
%of that map recall that from the diagram
%\[\begin{tikzcd}
%LM \times LM & LM \times M \ar[l] \\
%M \times LM \ar[u] & M \times M \ar[l]\ar[u]
%\end{tikzcd}\]
%of spaces over $M \times M$ we obtain by pulling back along the diagonal the diagram
%\[\begin{tikzcd}
%\Figeight & LM \ar[l] \\
%LM \ar[u] & M \ar[l]\ar[u].
%\end{tikzcd}
%\]
%As aready used in the proof of \ref{???} we have models for those squares together with the natural map from the second square into the first one.
%Thus we also have a map of the total cofiber of the second one into the total cofiber of the first one. But the total cofiber of the second one is exactly the cone of $\RR \to \Figeight$ whereas for the first one we obtain two copies of the cone $M \to LM$. Writing out the map in our models gives the stated formula.
%
%%Instead of showing that the model we have of the cut map can made relative, we can use that the map $M \to LM$ is split, and thus also the map $\Apl(LM) \to \Apl(LM, M)$ is split. Thus the relative cut map can equivalently be defined as the compositiong
%%\[
%%\Apl(LM, M) \otimes \Apl(LM, M) \to \Apl(LM \times LM) \overset{\text{cut}}{\longrightarrow} \Apl(\Figeight)
%
%\Nnote{to be finished} Furthermore applying naturality in the first argument, we apply it to the square
%\[
%\begin{tikzcd}
%LM \times LM \ar[r] \ar[d] & LM \times M \ar[d] \\
%M \times LM \ar[r] & M \times M
%\end{tikzcd}
%\]
%to obtain a relative version
\end{proof}

\subsubsection{Model for the relative intersection}
We are left to find a model for the relative intersection step of \eqref{diag:defn of coproduct}. We will use the decomposition of this map given by 
the relative intersection product using the oriented intersection context \eqref{diag:FM_2}:
\begin{equation}\label{diag:def of FM_2 rel int}
\begin{tikzcd}
C_*(LM) \ar[r]& C_*(LM, LM|_{FM_2}) & \ar[l, "\sim"] C_*(LM|_{M}, LM|_{UTM}) \ar[r, "\cap p_\EE^*\tau"] & C_{* -n}(LM|_{M}) \\
C_*(\RR) \ar[r, equal] \ar[u,"f"] \ar[u]& C_*(\RR) \ar[r,equal] \ar[u,"f"]& C_*(\RR) \ar[r, "\cap f^*p_\EE^*\tau"] \ar[u,"f"]& C_{* -n}(\RR) \ar[u,"f"].
\end{tikzcd}
\end{equation}
%But the objects and maps (except for capping with a certain class) are obtained by taking fiber products and naturality of taking fiber products with the intersection context.
%Recall first the diagram of spaces (Diagram~\eqref{diag:FM_2})
%\begin{equation*}
%\begin{tikzcd}
 % UTM \ar[r] \ar[d] & FM_2 \ar[d] \\
%  M \ar[r] & M \times M.
%\end{tikzcd}
%\end{equation*}
The middle map is the ``excision'' map
$$C_*(\Figeight, LM|_{UTM})\cong C_*(LM|_M,LM|_{UTM})\arsim C_*(LM,LM|_{FM_2}),$$
and we need a cochain model for a homotopy inverse of that map. We start by giving a model of the spaces involved, starting from appropriate models of $UTM$ and $FM_2$. 
%which we can do using Theorem~\ref{thm:Eilenberg-Moore}, once we have a model of the our chosen intersection context these spaces are defined from, namely the diagram: 
%There is exactly one arrow in \eqref{diag:alg intersection} that goes the "wrong way" but is an equivalence.
%We will here choose  nice enough models for $UTM$ and $FM_2$, that will allow us to write down an explicit and simple formula for the needed homotopy inverse.

\medskip

%\Nnote{we have a standing assumption that $M$ is simply connected, otherwise we cannot apply EM.}
Suppose now that $A=A_M$ is a Poincare duality model for $M$, as given by Theorem~\ref{lastthm}. Then $A$ has a coproduct map $\Delta \colon s^{-n}A \to A \otimes A$ (dual to the intersection product of $M$, see Example~\ref{ex:PDint}).
% The cone of this map gives a square-zero extension (after forgetting the differential) of $A \otimes A$ by $A[n]$ which
Lambrechts-Stanley conjectured in \cite{LamSta08b} explicit commutative dg algebra models for configuration spaces. This conjecture was shown to hold over the reals by Idrissi and Campos-Willwacher, see \cite{Idrissi}, \cite[Appendix A]{Campos-Willwacher}. 

For  $FM_2$, this model is the quotient of the truncated polynomial algebra
\[
\F_A = \left(\frac{ A \otimes A [\omega_{1,2}] }{( \omega_{1,2}^2 = 0 , (a \otimes 1)\omega_{1,2} = (1 \otimes a) \omega_{1,2})} \ , \  d\omega_{1,2} = \Delta(1) \right),
\]
where $\omega_{1,2}$ is a degree $n-1$ class. 
The spherical fibration $UTM$ more classically admits a model 
\[
\U_A = \left( \frac{A[\vartheta]}{(\vartheta^2 = 0)} \ , \ d\vartheta = e \right),
\]
where $\vartheta$ has degree $n-1$, representing the fiber, and  $e = (m \circ \Delta)(1) \in A$ is the Euler class of $M$. 

These algebras fit into the commutative diagram
\begin{equation}\label{diag:Poincare duality FM_2}
\begin{tikzcd}
\U_A & \ar[l] \F_A \\
A \ar[u] & \ar[l,"m"] \ar[u] A \otimes A,
\end{tikzcd}
\end{equation}
where the vertical maps are the natural inclusions and the top map takes $\vartheta$ to $\omega_{1,2}$.

\begin{theorem}\label{thm:LSpotpourri}
Let $A$ be a Poincare duality model for a simply-connected manifold $M$. 
  Then the following hold:
\begin{enumerate}
\item The diagram \eqref{diag:Poincare duality FM_2} is a real model for \eqref{diag:FM_2}, i.e. there exists a zig-zag of quasi-isomorphisms of squares of commutative dg $\mathbb{R}$-algebras connecting \eqref{diag:Poincare duality FM_2} to the diagram obtained from \eqref{diag:FM_2} by applying $\Apl(-)$. 
\item The map $\phi:\cone(A\to \U_A)\to \cone(A\otimes A\to \F_A)$ taking $(x, y + \vartheta z) \in A \oplus s\U_A$ to $(\Delta(z), (z\otimes 1)\omega_{1,2} ) \in A \otimes A \oplus s\F_A$, is a model for the homotopy inverse of the map of pairs $\iota: (UTM,M)\arsim (FM_2,M\x M)$, and is a map of $A \otimes A$-modules. 

\item A representative of the Thom class $\tau \in \cone(A \to \U_A)$ is given by
\[
\tau=(e, \vartheta),
\]
where $e = m \circ \Delta(1) \in A$ is the Euler class as above.
\end{enumerate}
\end{theorem}

\begin{proof}[Proof sketch] 
  Part (1) follows from the works \cite{Campos-Willwacher} and \cite{Idrissi}: the model of $FM_2$ given here is that of Lambrechts-Stanley, and it is a commutative dg algebra model of $FM_2$ over the reals by these two papers. Analysing the models, we see that the maps in Diagram \eqref{diag:FM_2} are modeled as stated, as the multiplication of $A$ models the diagonal, and the class $\omega_{1,2}$ corresponds to the class of the sphere in $UTM$. Going through the proof in \cite{Campos-Willwacher} or \cite{Idrissi} that $\F_A$ is quasi-isomorphic to $\Apl(FM_2,\R)$, one can strengthen the statements to obtain a zig-zag of squares of commutative dg algebras, as claimed.  See also \cite[Proposition 8.3]{NaeWil19}. 
%In loc. cit. a model for the compactified configuration space of $n$-points is obtained that is compatible with "colliding of points" (more precisely, as operadic right modules over
%a version of the little cubes operad). We only need the configuration of $2$ points, that is $FM_2$ and the configuration space of $1.5$-points, which is $UTM$ together with the maps in the diagram
%\ref{diag:FM_2}.
%Except for the space $FM_2$ this is not part of the particular flavour of little cubes operad that is considered in loc. cit., however, the proof works our case verbatim (see for instance \cite[Proposition 8.3]{NaeWil19} which gives the missing piece to apply the method of \cite[Appendix A]{Campos-Willwacher}).

  For part (2), we know that the map of pairs is a homotopy equivalence, and that it is modeled by the map $\widehat m$ coming from the diagram. So it is enough to check that $\phi$ a 1-sided homotopy inverse to $\widehat m$. The composite $\widehat m\circ \phi$ takes $(x, y + \vartheta z)$ to $ (z e, z\vartheta)$ in  $\cone(A \to \U_A)$. This map is homotopic to the identity by the homotopy $h(x, y + \vartheta z) \mapsto (y,0)$. One checks that $\phi$ is a map of $A\otimes A$--modules.
  %\Nnote{though I don't see how that last claim should hold... may I'm also unsure what the $A\otimes A$--module structure is to begin with}
%\Fnote{The $A \otimes A$-module structure comes because everything in the square is an $A \otimes A$-algebra and taking cones still gives $A \otimes A$-modules. The calculation is then $\phi( (\alpha \otimes \beta) (x, y + \vartheta z) ) = \phi( \alpha \beta x, \alpha \beta y \pm \vartheta \alpha \beta z) = (\Delta(\alpha \beta z) , \pm(\alpha \beta z) \omega_{1,2}) = (\alpha \otimes \beta \Delta(z), \pm(\alpha \otimes \beta) \omega_{1,2}) = \alpha \otimes \beta \phi(z)$. Here we used $\Delta(xyz) = (x \otimes y)\Delta(z)$.}

Part (3) follows from the analysis of the models in (1). Alternativly, note that $\cone(A \to \U_A) \sim s^{n-1}A$ and thus there is only one candidate up to a scalar for the Thom class. The scalar is determined by the condition that the image of the Thom class under the isomorphism $H^n(M, UTM) \cong H^n(M \times M, M \times M \setminus M) \to H^n(M \times M)$ is the diagonal class. By (2) this image is $\Delta(1) \in A \times A$, which is the diagonal class. %(corresponding to the chosen orientation).
%In such a model it is easy to write down an inverse to $\cone(A \otimes A \to \F) \to \cone(A \to \U)$. Namely, we can send
%\begin{equation}\label{eqn:alg retraction}
%(x,y,z) \mapsto (\Delta(y), y, \Delta(z)),
%\end{equation}
%the key point being that this is a map of $A \otimes A$-modules.
\end{proof}

\subsubsection{Proof of Theorem  \ref{thm:alggeo}}

We now assemble the results of the previous sections to give a sketch proof of Theorem~\ref{thm:alggeo}.
% ,  explicitly tracing through the three steps in our algebraic models for \eqref{diag:defn of coproduct}.
%Let $\alpha, \beta \in \cone(C_\bullet(A,A) \to A)$ be represented as element in $\ker(C_\bullet(A,A) \to A)$ and write
Let
$$\overline{a_1} \otimes \dotsb \otimes \overline{a_p} \otimes a_{p+1},\ \overline{b_1} \otimes \dotsb \otimes \overline{b_q} \otimes b_{q+1}\ \in\ \underline{C}_*(A,A)$$
be two Hochschild chains. 
By Proposition~\ref{prop:cut}, applying the cut map to their tensor product we get 
\[
\pm (\overline{a_1} \otimes \dotsb \otimes \overline{a_p}) \otimes (\overline{b_1} \otimes \dotsb \otimes \overline{b_q}) \otimes a_{p+1}b_{q+1} \in C_{*}(A,A) \otimes_A C_{*}(A,A).
\]
%Note that this is an element in the relevant cone as it evaluates to 0 under $C_\bullet(A,A) \otimes_A C_\bullet(A,A) \to C_\bullet(A,A) \oplus_A C_\bullet(A,A)$. \Nnote{Are you working with cones or with kernels??}
Next we apply the relative intersection product as given by our algebraic model of Diagram~\eqref{diag:def of FM_2 rel int}:  %defined by the diagram \eqref{diag:alg intersection}.
writing $\mathcal{L}_A:= A^{\otimes 2}\otimes_{A^{\otimes 4}} (BA)^{ \otimes 2}$ and $\RR_A:=C_{*}(A,A) \oplus_A C_{*}(A,A)$ for our models of $LM$ and $\RR$ of Section~\ref{sec:J}, and applying Theorem~\ref{thm:Eilenberg-Moore} to the pullbacks of $LM$ along the maps of Diagram~\eqref{diag:FM_2}, as modeled by \eqref{diag:Poincare duality FM_2}, the intersection product is modeled by 
% \[
% \begin{tikzcd}[scale = 0.8, column sep = tiny]
% BA \otimes_{A \otimes A} BA \ar[d] & \ar[l] BA \otimes BA \otimes_{A \otimes A} \cone(A\otimes A \to \F) \ar[r, "\sim"]  \ar[d]& BA \otimes BA \otimes_{A\otimes A} \cone(A \to \U)  \ar[d]& \ar[l, "\cup Th"] (BA \otimes BA) \otimes_{A \otimes A} A[n] \ar[d] \\
% C_\bullet(A,A) \oplus_A C_\bullet(A,A) \ar[r, equal] & C_\bullet(A,A) \oplus_A C_\bullet(A,A) \ar[r,equal] & C_\bullet(A,A) \oplus_A C_\bullet(A,A) &\ar[l, "\cup Th"] C_\bullet(A,A) \oplus_A C_\bullet(A,A)[n],
% \end{tikzcd}
% \]
\begin{equation*}%\label{diag:alg intersection}
\begin{tikzcd}[column sep = small]
\mathcal{L}_A\ar[d] & \ar[l]  \mathcal{L}_A \otimes_{A^{\otimes 2}}\cone(A^{\otimes 2} \to \F_A)\ar[r, "\sim"]  \ar[d]&  \mathcal{L}_A \otimes_{A^{\otimes 2}} \cone(A \to \U_A)\ar[d]& \ar[l, " \cup\tau_M"'] \mathcal{L}_A \otimes_{A^{\otimes 2}}A \ar[d] \\
\RR_A \ar[r, equal] & \RR_A \ar[r,equal] & \RR_A &\ar[l, "\cup \tau_M"'] \RR_A,
\end{tikzcd}
\end{equation*}
where the first map has degree $n$, with source
$$\mathcal{L}_A \otimes_{A^{\otimes 2}} A  =  (A^{\otimes 2}\otimes_{A^{\otimes 4}} (BA)^{ \otimes 2})\otimes_{A^{\otimes 2}} A \cong C_{*}(A,A) \otimes_A C_{*}(A,A).$$
Recall from Theorem~\ref{thm:LSpotpourri}(3) that the Thom class is given by $\tau_M=(e,\vartheta)$ in our model, so applying the first map to our element gives  
\[
  \pm (\overline{a_1} \otimes \dotsb \otimes \overline{a_p}) \otimes (\overline{b_1} \otimes \dotsb \otimes \overline{b_q}) \otimes (a_{p+1}b_{q+1}e, a_{p+1}b_{q+1} \vartheta) \]
in $C_{*}(A,A) \otimes_A C_{*}(A,A) \otimes_A \cone(A \to \U_A)\cong \mathcal{L}_A \otimes_{A^{\otimes 2}} \cone(A \to \U_A)$.
Now we apply the explicit inverse of $\cone(A \otimes A \to \F_A) \to \cone(A \to \U_A)$ given in Theorem \ref{thm:LSpotpourri} which yields
\[
  \pm (\overline{a_1} \otimes \dotsb \otimes \overline{a_p}) \otimes (\overline{b_1} \otimes \dotsb \otimes \overline{b_q}) \otimes (\Delta( a_{p+1} b_{q+1}), (a_{p+1}b_{q+1}\otimes 1) \omega_{1,2}) \]
in $\mathcal{L}_A \otimes_{A^{\otimes 2}} \cone(A \otimes A \to \F_A)$.
Next applying $\cone(A^{\otimes 2} \to \F_A) \to A^{\otimes 2}$, we obtain
\[
  \pm (\overline{a_1} \otimes \dotsb \otimes \overline{a_p}) \otimes (\overline{b_1} \otimes \dotsb \otimes \overline{b_q}) \otimes \Delta( a_{p+1} b_{q+1}) \]
in  $\mathcal{L}_A \otimes_{A^{\otimes 2}} A^{\otimes 2}\cong A^{\otimes 2}\otimes_{A^{\otimes 4}} (BA)^{ \otimes 2}$.
% which we will have to show that it lies (up to coboundaries) in the image of
% \[
% \operatorname{Tot}\left( \begin{tikzcd}
% 0 \ar[d] \\
% 0 \oplus_A C_{*}(A,A)
% \end{tikzcd}\right) 
% \longrightarrow
% \operatorname{Tot}\left( \begin{tikzcd}
% BA \otimes BA \ar[d] \\
% C_\bullet(A,A) \oplus_A C_\bullet(A,A)
% \end{tikzcd}\right)
% \]
Finally, the reparametrization map $J$ is given by Proposition \ref{prop:J} after applying the last identification and yields the formula for the coproduct as
\[
\sum \pm (\overline{a_1} \otimes \dotsb \otimes \overline{a_p} \otimes \overline{a_{p+1}}e_i \otimes \overline{b_1} \otimes \dotsb \otimes \overline{b_q}) \otimes b_{q+1}f_i \in sC_{*}(A,A)
\]
matching the formula for the algebraic Goresky-Hingston product of Definition \ref{defn:algebraic coproduct} (up to switching the factors, which does not make a difference on cohomology by the graded commutativity of the product, see Theorem~\ref{thm1}).

\bibliographystyle{plain}
\bibliography{biblio}

\end{document}